\documentclass[12pt,leqno]{amsart}
\pdfoutput=1
\usepackage{amssymb,amsmath,amsthm,bm,amsfonts}
\usepackage{graphicx}
\usepackage[dvipsnames]{xcolor}
\usepackage[normalem]{ulem}
\usepackage[bookmarksopen,bookmarksdepth,colorlinks,citecolor=red,pagebackref,hypertexnames=true]{hyperref}
\usepackage[msc-links,nobysame,non-sorted-cites, initials]{amsrefs}
\usepackage[inline]{enumitem}
\usepackage{mathtools}
\mathtoolsset{showonlyrefs}

\usepackage{mathrsfs}
\usepackage{tikz}
\usepackage{tikz-3dplot}
\usepackage{pgfplots}
\pgfplotsset{compat=1.18}
\usepgfplotslibrary{patchplots}
\usetikzlibrary{intersections}

\usepackage[T1]{fontenc} 
\usepackage{libertine}   
\usepackage[libertine]{newtxmath}

\usepackage{multirow}
\usepackage{caption}
\usepackage{subcaption}
\usepackage{float}


\numberwithin{equation}{section}

\newcounter{TheoremCounter}

\theoremstyle{plain}

\newtheorem{theorem}[TheoremCounter]{Theorem}
\newtheorem{corollary}{Corollary}
\numberwithin{corollary}{TheoremCounter}

\newcounter{GenCounter}
\numberwithin{GenCounter}{section}

\newtheorem{lemma}[GenCounter]{Lemma}
\newtheorem{proposition}[GenCounter]{Proposition}

\newcounter{QCounter}
\newtheorem{question}[QCounter]{Question}

\theoremstyle{definition}

\newtheorem{definition}[GenCounter]{Definition}
\newtheorem{remark}[GenCounter]{Remark}

\makeatletter
\renewcommand\p@subfigure{} 
\makeatother

\DeclareMathOperator{\supp}{supp}

\DeclareMathOperator{\BMO}{BMO}
\DeclareMathOperator{\dist}{dist}
\DeclareMathOperator{\triadic}{triadic}
\DeclareMathOperator{\pv}{p.v.}
\DeclareMathOperator{\dense}{\normalfont{dense}_0}
\DeclareMathOperator{\densesup}{\normalfont{\textbf{dense}}}
\DeclareMathOperator{\size}{\normalfont{\textbf{size}}}
\DeclareMathOperator{\smal}{small}
\DeclareMathOperator{\lac}{lac}
\DeclareMathOperator{\topt}{top}
\DeclareMathOperator{\sig}{sig}
\DeclareMathOperator{\Dil}{Dil}
\DeclareMathOperator{\cone}{cn}
\DeclareMathOperator{\children}{ch}
\DeclareMathOperator{\SO}{SO}
\DeclareMathOperator{\Id}{Id}
\DeclareMathOperator{\Mod}{Mod}
\DeclareMathOperator{\tops}{tops}
\DeclareMathOperator{\dens}{dense}
\DeclareMathOperator{\siz}{size}
\DeclareMathOperator{\heavy}{heavy}
\DeclareMathOperator{\light}{light}
\DeclareMathOperator{\spn}{span}
\DeclareMathOperator{\ann}{ann}

\DeclareMathOperator{\ecc}{Ecc}
\DeclareMathOperator{\lip}{Lip}
\DeclareMathOperator{\CS}{CS}

\newcommand{\Grdn}{\mathrm{\bf{Gr}}(d,n)}
\newcommand{\Gr}{\mathrm{\bf{Gr}}}
\newcommand{\scl}{\mathrm{scl}}

\newcommand{\N}{\mathbb{N}}
\renewcommand{\P}{\mathbb{P}}
\newcommand{\R}{\mathbb{R}}
\newcommand{\T}{\mathbb{T}}
\newcommand{\C}{\mathbb{C}}
\newcommand{\Z}{\mathbb{Z}}

\newcommand{\Abs}[1]{\left|#1\right|}
\newcommand{\abs}[1]{|#1|}
\newcommand{\Norm}[2]{\left\|#1\right\|_{#2}}
\newcommand{\norm}[2]{\|#1\|_{#2}}
\newcommand{\indf}[1]{\bm{1}_{#1}}
\newcommand{\apl}[3]{#1 \colon #2 \longrightarrow #3}

 \def\Xint#1{\mathchoice
	{\XXint\displaystyle\textstyle{#1}}%
	{\XXint\textstyle\scriptstyle{#1}}%
	{\XXint\scriptstyle\scriptscriptstyle{#1}}%
	{\XXint\scriptscriptstyle\scriptscriptstyle{#1}}%
	\!\int}
\def\XXint#1#2#3{{\setbox0=\hbox{$#1{#2#3}{\int}$}
		\vcenter{\hbox{$#2#3$}}\kern-.5\wd0}}

\def\avgint{\Xint-}

\makeatletter
\DeclareRobustCommand\widecheck[1]{{\mathpalette\@widecheck{#1}}}
\def\@widecheck#1#2{%
	\setbox\z@\hbox{\m@th$#1#2$}%
	\setbox\tw@\hbox{\m@th$#1%
		\widehat{%
			\vrule\@width\z@\@height\ht\z@
			\vrule\@height\z@\@width\wd\z@}$}%
	\dp\tw@-\ht\z@
	\@tempdima\ht\z@ \advance\@tempdima2\ht\tw@ \divide\@tempdima\thr@@
	\setbox\tw@\hbox{%
		\raise\@tempdima\hbox{\scalebox{1}[-1]{\lower\@tempdima\box
				\tw@}}}%
	{\ooalign{\box\tw@ \cr \box\z@}}}
\makeatother

\tikzset{surface/.style={draw=blue!70!black, fill=blue!40!white, fill opacity=.6}}

\tikzset{3d stuff/.is family,
	3d stuff/.cd,
	parse domain/.code args={#1:#2}{\def\xmin{#1}\def\xmax{#2}},
	sphere segment/.is family,
	sphere segment/.cd,
	r/.initial=1, 
	phi/.initial=0:30, 
	theta/.initial=0:30
}

\colorlet{darkgreen}{green!80!black}

\author[M.\ Fl\'orez-Amatriain]{Mikel Fl\'orez-Amatriain}
\address[M.\ Fl\'orez-Amatriain]{PhD student under the supervision of Ioannis Parissis and Luz Roncal at BCAM - Basque Center for Applied Mathematics,
		48009 Bilbao, Spain, and Universidad del Pa\'is Vasco, Bilbao, Spain}
\email{\href{mailto:mflorez@bcamath.org}{\textnormal{mflorez@bcamath.org}}}


\title[$L^p$-estimates for singular integral operators along codimension one subspaces]{$L^p$-estimates for singular integral operators along codimension one subspaces}

\begin{document}

	\thanks{M. Fl\'orez-Amatriain is partially supported by the projects CEX2021-001142-S, grant PID2021-122156NB-I00 funded by MICIU/AEI/10.13039/501100011033 and FEDER, UE, PID2023-146646NB-I00 funded by MICIU/AEI/10.13039/50110001103, grants BERC 2022-2025 of the Basque Government and predoc Basque Government grant 2022 ``Programa Predoctoral de Formaci\'on de Personal Investigador No Doctor''}


	\begin{abstract} 
		In this paper we study maximal directional singular integral operators in $ \R^n $ given by a H\"ormander--Mihlin multiplier on an $ (n-1)$-dimensional subspace and acting trivially in the perpendicular direction. The subspace is allowed to depend measurably on the first $ n-1 $ variables of $ \R^n $. Assuming the subspace to be \emph{non degenerate} in the sense that it is away from a cone around $e_n$ and the function $ f $ to be frequency supported in a cone away from $ \R^{n-1} $, we prove $ L^p $-bounds for these operators for $ p > 3/2 $. If we assume, additionally, that $ \widehat{f} $ is supported in a single frequency band, we are able to extend the boundedness range to $ p >1 $. The non-degeneracy assumption cannot in general be removed, even in the band-limited case.
	\end{abstract}	

	\maketitle

	\section{Introduction}
	\label{S:Intro}
	
	The main objects of interest in this paper are maximal versions of directional singular integrals in higher dimensions which act trivially in the variable perpendicular to an $(n-1)$-dimensional subspace. More precisely, we will prove $L^p$-bounds in the case where the choice of subspace depends measurably on the first $n-1$ variables of $\R^n $. These operators fall under the scope of maximal singular Radon transforms and have been considered in different forms in several papers; a detailed account will be given in Section \ref{Ss:IntroPast} below. The present formulation is from \cite{BDPPR} where the authors setup and investigate such maximal directional multipliers in codimension $ 1 $; see also \cite{DPP22} for the case of directional averages in any codimension. As we explain in Section \ref{Ss:IntroPast} below, the study of these operators is motivated by Zygmund's conjecture on the differentiation of integrals along Lipschitz vector fields and by the corresponding question of Stein concerning singular integrals along Lipschitz directions.

	\subsection{Basic definitions and main results}
	\label{Ss:IntroMain}
	
	Let $ n \in \N $, $ n \geq 2 $, $ d=n-1 $ and $ \apl{m}{\R^d}{\C} $ be such that $ m \in L^{\infty}(\R^d) $. Let $ \Grdn $ be the oriented Grassmannian, that is, the space of oriented $d$-dimensional subspaces of $ \R^n $. For $ \sigma \in \Grdn $, we will denote by $ v_{\sigma} \in \mathbb{S}^{n-1} $ the unit vector which is \emph{the exterior normal} to the oriented hyperplane $ \sigma $, and by $ V_{\sigma} \in \SO(n)$ a rotation in $ \R^n $ such that $ \apl{V_{\sigma}}{\sigma}{\R^d}$ and $V_\sigma v_\sigma=e_n$. Then, we define the \emph{singular integral operator along $ \sigma $} as
	\begin{equation}
		\label{Eq:SingOperDef}
		T_{m} f(x; \sigma, V_{\sigma} ) \coloneqq \int_{\R^n} m (V_{\sigma} \Pi_{\sigma} \xi) \widehat{f}(\xi) e^{2\pi i \langle x , \xi \rangle} d \xi,\qquad x\in\R^n,
	\end{equation} 
	where $ \apl{\Pi_{\sigma}}{\R^n}{\sigma} $ denotes the orthogonal projection on $ \sigma $; see Figure \ref{Fig:SIOProjection}. Suppose that $ m \in \mathcal{M}_A(d) $ is a \emph{H\"ormander--Mihlin multiplier} on $ \R^d $, namely 
	\begin{equation}
		\label{Eq:MihlinDef}
		\norm{m}{\mathcal{M}_A(d)} \coloneqq \sup_{0\leq \abs{\beta} \leq A} \sup_{ \eta \in \R^d } \abs{\eta}^{\abs{\beta}} \Abs{ D^\beta m(\eta ) } < +\infty,
	\end{equation}
	for some positive integer $ A $ which will be taken to be sufficiently large depending on the integrability exponent $ p $ and the dimension. Typical examples of the singular integral operators \eqref{Eq:SingOperDef} are the $ d $-dimensional Riesz transforms, given by the multiplier
	\begin{equation}
		\label{Eq:RieszDef}
		m( \eta ) = \frac{ \eta }{ \abs{ \eta } }, \quad \eta \in \R^d \qquad \Longrightarrow \qquad T_{m} f(x; \sigma, V_{\sigma} ) = \int_{\R^n} \frac{ V_{\sigma} \Pi_{\sigma} \xi }{ \abs{ V_{\sigma} \Pi_{\sigma} \xi } } \widehat{f}(\xi) e^{2\pi i \langle x , \xi \rangle} d \xi;
	\end{equation}
	see Figure \ref{Fig:SIOMultiplier}. As a consequence of the classical H\"ormander--Mihlin multiplier theorem and Fubini's theorem, $ T_{m} f(\cdot; \sigma, V_{\sigma} ) $ is bounded on $ L^p(\R^n) $ for $ 1 < p < \infty $, uniformly in $ \sigma \in \Grdn $.
	
	\begin{figure}
		\centering
		\begin{subfigure}{0.49\textwidth}
			\tdplotsetmaincoords{75}{140} 
			\resizebox{\textwidth}{!}{%
			\begin{tikzpicture}[scale=1.3,tdplot_main_coords] 
				
				\clip[tdplot_screen_coords] (-3.1, -1.5) rectangle (3.1, 2);
				
				\draw[dashed,black,thick] (0,0,0) -- (0,0,-2);
				
				\filldraw[draw=cyan,fill=cyan!30,fill opacity=0.6,draw opacity=0] 
				(1,2,0) -- (-1,-2,0) --
				(-2,-2,{2*0.15/sqrt(0.8875)-2*0.3/sqrt(0.8875)}) -- (-2,2,{-2*0.15/sqrt(0.8875)-2*0.3/sqrt(0.8875)}) -- cycle;
				
				\draw[draw=cyan,opacity=0.6] 
				(-1,-2,0) --
				(-2,-2,{2*0.15/sqrt(0.8875)-2*0.3/sqrt(0.8875)}) -- (-2,2,{-2*0.15/sqrt(0.8875)-2*0.3/sqrt(0.8875)}) -- (1,2,0) ;
				
				\draw[red,thick] (2*0.3,-2*0.15,{-2*sqrt(0.8875)}) -- (-2*0.3,2*0.15,{2*sqrt(0.8875)}) node[anchor=north west]{$\sigma^{\perp}$} ;
				
				\filldraw[draw=gray, fill=gray!30, opacity=0.6] 
				(-2,-2,0) -- (-2,2,0) -- (2,2,0) -- (2,-2,0) -- cycle;
				
				\draw[thick,->,black] (0,0,0) -- (2,0,0) node[anchor=north]{$\xi_1$};
				\draw[dashed,black,thick] (0,0,0) -- (-2,0,0);
				
				\draw[dashed,black,thick] (0,0,0) -- (0,-2,0);
				
				\tdplotsetthetaplanecoords{-39.8585}
				\tdplotdrawarc[tdplot_rotated_coords,dashed,orange!80!black,dash pattern=on 1pt off 1pt]{(0,0,0)}{1.55594}{90-19.5975}{90}{anchor=south west}{}
				
				\filldraw[draw=cyan,draw opacity=0,fill=cyan!30, fill opacity=0.6] 
				(2,-2,{2*0.15/sqrt(0.8875)+2*0.3/sqrt(0.8875)}) --
				(2,2,{-2*0.15/sqrt(0.8875)+2*0.3/sqrt(0.8875)}) --
				(1,2,0) -- (-1,-2,0) -- cycle;
				
				\draw[draw=cyan,draw opacity=0.6] 
				(-1,-2,0) -- (2,-2,{2*0.15/sqrt(0.8875)+2*0.3/sqrt(0.8875)}) -- (2,2,{-2*0.15/sqrt(0.8875)+2*0.3/sqrt(0.8875)}) -- (1,2,0) ;
				\draw[draw=cyan,dashed,dash pattern=on 1pt off 1pt,draw opacity=0.6] 
				(1,2,0) -- (-1,-2,0) ;
				
				\draw[thick,->,black] (0,0,0) -- (0,2,0) node[anchor=north]{$\xi_2$};
				
				\draw[thick,->,black] (0,0,0) -- (0,0,2) node[anchor=north east]{$\xi_n$};
				
				\draw[very thick,->,blue] (0,0,0) -- (-0.3,0.15,{sqrt(0.8875)}) node[anchor=north west]{$v_{\sigma}$};
				
				\draw[dashed,orange!80!black,dash pattern=on 1pt off 1pt] (0.8,-0.8,1.5) -- ({0.8-(0.108-0.45*sqrt(0.8875))}, {-0.8-(-0.054+0.225*sqrt(0.8875))}, {1.5-(-0.36*sqrt(0.8875)+1.33125)});
				
				\shade[ball color = orange!80!black] (0.8,-0.8,1.5) circle (0.08cm) node[anchor=east,orange!80!black]{$\xi$};
				
				\shade[ball color = orange!80!black] ({0.8-(0.108-0.45*sqrt(0.8875))}, {-0.8-(-0.054+0.225*sqrt(0.8875))}, {1.5-(-0.36*sqrt(0.8875)+1.33125)}) circle (0.08cm) node[anchor=east,orange!80!black]{$\Pi_{\sigma}\xi$};
				
				\shade[ball color = orange!80!black] (1.19439,-0.997195,0) circle (0.08cm) node[anchor=east,orange!80!black]{$V_{\sigma} \Pi_{\sigma} \xi$};
				
				\node[cyan] at (-1.2,2.6,-0.4) {$\sigma$}; 
				\node[gray] at (-1.2,2.6,0.6) {$\R^d$}; 
				
			\end{tikzpicture}
			}
			\caption{Action of the projection $ \Pi_{\sigma} $ and rotation $ V_{\sigma} $.}
			\label{Fig:SIOProjection}
		\end{subfigure}
		\hfill
		\begin{subfigure}{0.49\textwidth}
			\pgfplotsset{
				colormap={blue}{rgb=(0.4,0.8,1) rgb=(0,0,1) rgb=(0,0,0.1)},
			}
			\tdplotsetmaincoords{75}{140}
			\resizebox{\textwidth}{!}{%
				\begin{tikzpicture}[scale=1.3,tdplot_main_coords] 
					
					\clip[tdplot_screen_coords] (-3.1, -1.5) rectangle (3.1, 2);
					
					\draw[dashed,black,thick] (0,0,0) -- (0,0,-2);
					
					\filldraw[draw=gray, fill=gray!30, opacity=0.6] 
					(-2,-2,0) -- (-2,2,0) -- (0,2,0) -- (0,-2,0) -- cycle;
					
					\begin{axis}
						[anchor=center, 
						unit vector ratio=1 1 1, 
						scale=0.82, 
						axis lines=none,
						view={140}{15} , 
						domain=-2:2, 
						colormap name=blue]
						\addplot3[surf, shader=interp, opacity=0.6, samples=30] 
						{-1*x/sqrt(x^2 + y^2)};
						\addplot3 [mesh, line width=0.1mm, draw=blue!30!black, samples=30, opacity=0.6] 
						{-1*x/sqrt(x^2 + y^2)};
					\end{axis}
					
					\filldraw[draw=gray, fill=gray!30, opacity=0.6] 
					(0,-2,0) -- (0,2,0) -- (2,2,0) -- (2,-2,0) -- cycle;
					
					\draw[thick,->,black] (0,0,0) -- (2,0,0) node[anchor=north east]{};
					\draw[dashed,black] (0,0,0) -- (-2,0,0);
					
					\draw[dashed,black,thick] (0,0,0) -- (0,-2,0);
					
					\draw[thick,->,black] (0,0,0) -- (0,2,0) node[anchor=north west]{};
					
					\draw[thick,->,black] (0,0,0) -- (0,0,2) node[anchor=south]{};
					
					\draw[dashed,orange!80!black,dash pattern=on 1pt off 1pt] (1.19439,-0.997195,0) -- (1.19439,-0.997195,{-1*1.19439/sqrt(1.19439^2+0.997195^2)});
					
					\shade[ball color = orange!80!black] (1.19439,-0.997195,0) circle (0.05cm) node[anchor=east, orange!80!black]{$\eta$};
					
					\shade[ball color = orange!80!black] (1.19439,-0.997195,{-1*1.19439/sqrt(1.19439^2+0.997195^2)}) circle (0.05cm) node[anchor=north, orange!80!black]{$m(\eta)$};
					
					\node[gray] at (-1.2,2.7,0.05) {$\R^d$}; 
					\node[blue!70!black] at (-0.3,-1.2,0.9) {$m$}; 
					
				\end{tikzpicture}
			}
			\caption{The Riesz multiplier $ m(\eta_1,\eta_2) = \frac{ -\eta_1 }{ \abs{(\eta_1,\eta_2)}} $, satisfying \eqref{Eq:MihlinDef}.}
			\label{Fig:SIOMultiplier}
		\end{subfigure}
		
		\caption{Graphical explanation of the action of the multiplier of the singular integral operator $ T_{m} $.}
		\label{Fig:SIO}
	\end{figure}

	Now, consider a family of multipliers $ \mathscr{M} = \{m_\sigma\in \mathcal M_A(d) : \, \sigma \in \Gr(d,n) \} $. Then, we define the \emph{maximal operator associated to $ \mathscr{M} $} by the formula 
	\begin{equation}
		\label{Eq:SingOperMaxDef}
		T^{\star}_{ \mathscr{M} } f(x) \coloneqq \sup_{ \sigma\in \Gr(d,n) } \sup_{ V_{\sigma} \in \SO(n) } \Abs{ T_{m_{\sigma} } f( x; \sigma, V_{\sigma} ) }, \qquad x\in \R^n.
	\end{equation}
	In this paper, we will work with a linearized version of this maximal operator. For $ \apl{ \sigma }{ \R^n }{ \Grdn } $ a measurable function, define the operator
	\begin{equation}
		\label{Eq:SingOperLinDef}
		T_{ \mathscr{M},\sigma(\cdot) } f(x) \coloneqq T_{ m_{\sigma(x)} } f( x; \sigma(x), V_{\sigma(x)} ), \qquad x\in \R^n.
	\end{equation}
	The operator $ T_{ \mathscr{M},\sigma(\cdot) } $ is a linearized version of $ T^{\star}_{ \mathscr{M} } $ since we can recover the latter by selecting, for each $ x \in \R^n $, the subspace $ \sigma(x) \in \Grdn $ that achieves the supremum in \eqref{Eq:SingOperMaxDef} and the map $ \R^n \ni x \mapsto \sigma(x) \in \Grdn $ is measurable. Furthermore, we will assume that the family of multipliers $ \mathscr{M} $ satisfies
	\begin{equation}
		\label{Eq:MihlinFamDef}  
		\norm{ \mathscr{M} }{ \mathcal{M}_A( \Gr(d,n) ) } \coloneqq \sup_{ \sigma,\sigma' \in \Gr(d,n) } \left[ \norm{ m_\sigma }{ \mathcal{M}_A(d) } + \log \left( e+ \frac{1}{\dist(\sigma,\sigma')} \right) \norm{ m_\sigma - m_{\sigma'} }{ \mathcal{M}_A(d) } \right] \leq 1.
	\end{equation}
	This formulation for $ T^{\star}_{ \mathscr{M} } $ and $ T_{ \mathscr{M},\sigma(\cdot) } $ in general dimensions and codimension $ 1 $ is from \cite{BDPPR} where the authors also introduce the assumption \eqref{Eq:MihlinFamDef} and give sufficient conditions for it to hold; see \cite[Remark 1.3 and Lemma 5.4]{BDPPR}. Note that \eqref{Eq:MihlinFamDef} is automatically satisfied if $ m_{\sigma} = m $ for all $ \sigma \in \Grdn $ and some H\"ormander--Mihlin multiplier $m\in \mathcal M_A(d)$; the logarithmic H\"older-type assumption is needed in order to allow families of multipliers $ \{ m_{\sigma} \}_{\sigma} $ with the choice of $ m $ depending on $ \sigma $.
	
	We will also assume that $ \sigma $ depends on the first $ n-1 $ variables and that it is almost horizontal. The analogous hypotheses for $ n=2 $ were introduced in \cites{Bateman2013Revista,BatemanThiele}. That is, on the one hand, $ \sigma(x) = \sigma(x_1,\ldots,x_d,x_n) = \sigma(x_1,\ldots,x_d) $ for $ x \in \R^n $. On the other hand, define for $ \sigma, \sigma' \in \Gr(n-1,n) $ the distance
	\begin{equation}
		\label{Eq:DistHyperplanes}
		\dist( \sigma , \sigma' ) 
		\coloneqq \norm{ v_{\sigma}-v_{\sigma'}}{}
		\coloneqq \sqrt{ 1 - \cos( \theta( v_{\sigma} , v_{\sigma'} ) ) },
	\end{equation}
	where $ \theta( v_{\sigma} , v_{\sigma'} ) $ denotes the convex angle between $ v_{\sigma} , v_{\sigma'} \in \mathbb{S}^{d}$ and $\|v\|\coloneqq \frac{1}{\sqrt{2}} |v|$ with $|v|$ denoting the Euclidean norm of some $v\in\R^n$. Then, for some $ 0 < \varepsilon < 1 $, we will assume that
	\begin{equation}
		\label{Eq:FreqConeDef}
		\sigma(x) \in \Sigma_{1-\varepsilon} \coloneqq \left\{ \sigma \in \Gr(n-1,n) :\, \norm{ v_{\sigma}-e_n}{} < 1-\varepsilon\right\}.
	\end{equation}
	Note that $ \norm{ v_{\sigma}-e_n}{} = 1 $ holds exactly when $ v_{\sigma} \in \R^{n-1} $. We call $ \varepsilon = 0 $ the degenerate case in which $\{v_\sigma:\, \sigma\in \Sigma_1\}$ may have accumulation points in $\R^{n-1}$.
	
	Finally, let $ \Psi_{1-\varepsilon} \in L^{\infty}(\R^n) $ be a nonnegative function such that $ \Psi_{1-\varepsilon}|_{\mathbb{S}^{n-1}} \in C^\infty (\mathbb{S}^{n-1})$ satisfying $ \indf{ \Gamma_{1-\varepsilon} } \leq  \Psi_{1-\varepsilon}  \leq \indf{\widetilde{\Gamma}_{1-\varepsilon}}$, where
	\begin{equation}
		\label{Eq:FreqProjConeDef}
		\begin{split} 
			\Gamma_{1-\varepsilon} \coloneqq \left\{ \xi \in \R^n : \, \norm{ \xi' - e_n }{} < 1-\varepsilon \right\}, 
			\qquad
			\widetilde{\Gamma}_{1-\varepsilon} \coloneqq \left\{ \xi \in \R^n :\, \norm{ \xi' - e_n }{} <  \sqrt{1-\varepsilon} \right\},
		\end{split}
	\end{equation}
	with $ \xi' \coloneqq \xi / \abs{\xi} $. Then, define the corresponding frequency cutoff $ P_{\cone,1-\varepsilon} f \coloneqq (\Psi_{1-\varepsilon} \widehat{f})^\vee $. Note that in the degenerate case $ \varepsilon = 0 $ we have no restriction since $ \Psi_{1} \equiv 1$ a.e. in $\R^n$ and thus $ P_{\cone,1} f = f $. With these assumptions and definitions in place, we will prove the following.
	
	\begin{theorem} 
		\label{Thm:Main}
		Let $ 0 < \varepsilon < 1 $ and $ \mathscr{M} = \{ m_{\sigma} \in L^\infty( \R^{n-1}) : \, \sigma \in \Gr(n-1,n) \}$ be a family of multipliers satisfying \eqref{Eq:MihlinFamDef} and $ \apl{ \sigma }{ \R^n }{ \Sigma_{1-\varepsilon} } $ be a measurable function depending only on the first $ n-1 $ variables. Then,
		\begin{equation}
			\label{Eq:DesiredEst}
			\Norm{ T_{ \mathscr{M},\sigma(\cdot) } (P_{\cone,1-\varepsilon} f) }{ L^p(\R^n) }  
			\lesssim_{p,\varepsilon} \Norm{ f }{ L^p(\R^n) }, \qquad \frac{3}{2} < p < \infty. 
		\end{equation}
	\end{theorem}
	
	An intermediate step in the proof of Theorem~\ref{Thm:Main} is a single annulus frequency version of it. Let $ \apl{ \zeta }{ \R^n }{ \R_{+} } $ be a smooth radial function such that $ \supp( \zeta ) \subseteq \left\{ \xi \in \R^n :\, 1 < \abs{ \xi } < 3 \right\} $ and for $ k \in \Z $ define the corresponding frequency projection
	\begin{equation}
		\label{Eq:ProjBandDef}
		\widehat{ P_k f } ( \xi ) \coloneqq \zeta( 3^{-k} \xi ) \widehat{ f }( \xi ), \qquad \xi \in \R^n. 
	\end{equation}
	
	\begin{theorem}
		\label{Thm:MainSingleBand}
		Let $ 0 < \varepsilon < 1 $ and $ \mathscr{M} = \{ m_{\sigma} \in L^\infty( \R^{n-1}) : \, \sigma \in \Gr(n-1,n) \} $ be a family of multipliers satisfying \eqref{Eq:MihlinFamDef} and $ \apl{ \sigma }{ \R^n }{ \Sigma_{1-\varepsilon} } $ be a measurable function depending only on the first $ d $ variables. Then, 
		\begin{equation}
			\label{Eq:DesiredSingleBandEst}
			\Norm{ T_{ \mathscr{M},\sigma(\cdot) } ( P_0 f ) }{ L^p(\R^n) } 
			\lesssim_{p,\varepsilon} \Norm{ f }{ L^p(\R^n) }, \qquad 1 < p < \infty . 
		\end{equation}
	\end{theorem}
	
	Note that the assumption \eqref{Eq:MihlinFamDef} is invariant under isotropic scaling of $\R^n $ and so estimate  estimate \eqref{Eq:DesiredSingleBandEst} is also invariant under the same rescaling. In consequence, \eqref{Eq:DesiredSingleBandEst} holds with the same constant with $ P_0 $ replaced by $ P_k $, for any $ k \in \R $.
	
	Theorem~\ref{Thm:MainSingleBand} holds for the whole range $ 1 < p < \infty $. However, in the range $ 3/2 < p < \infty $ it is a weaker version of Theorem~\ref{Thm:Main} because of the presence of the single annulus restriction given by the operator $ P_0 $ and since Theorem~\ref{Thm:MainSingleBand} is equivalent to its version with $ P_{\cone,1-\varepsilon} P_0 f $ in the place of $ P_0 f $, cf. \eqref{Eq:ConeReduction}. This is because $ T_{ \mathscr{M},\sigma(\cdot) } [ ( \Id - P_{\cone,1-\varepsilon} ) P_0 f ] \lesssim M f $ with $ M $ denoting the Hardy-Littlewood maximal function, as all the frequency singularities of $ T_{ \mathscr{M},\sigma(\cdot) } $ are contained in $\widetilde{\Gamma}_{1-\varepsilon} $.
	
	The specific choice of $ \varepsilon > 0 $ in the statements of Theorems \ref{Thm:Main} and \ref{Thm:MainSingleBand} should not be taken too seriously and for example one can think of $ \varepsilon = 1/2 $. However, some $ \varepsilon > 0 $ is needed and the implicit constant blows up as $ \varepsilon \to 0 $. Indeed, it is impossible to completely remove the non-degeneracy assumption $ \sigma(x) \in \Sigma_{1-\varepsilon} $ and get uniform bounds in $ \varepsilon $ for Theorems \ref{Thm:Main} and \ref{Thm:MainSingleBand}, as shown by the counterexamples in Section \ref{S:Counterexamples}. This is in stark contrast with the result in \cite{BatemanThiele}, where the authors are able to remove all frequency restrictions for $ p > 3/2 $ and vector fields depending only on the first variable because of the two-dimensional setup and the dilation invariance of the Hilbert transform. In the general setup of the present paper there exists counterexamples for the degenerate case when any of these two conditions fail. In the two-dimensional setting the reduction in \cite{BatemanThiele} fails for suitable non dilation invariant multipliers $m\in\mathcal M_A(1)$ as the ones constructed in \cite{CGHS}; see Subsection \ref{Ss:Counterexamples2}. In higher dimensions, Theorems \ref{Thm:Main} and \ref{Thm:MainSingleBand} fail without the non-degeneracy assumption even for homogeneous multipliers; see Subsection \ref{Ss:CounterexamplesN}.
	
	Finally, as a consequence of Theorem~\ref{Thm:MainSingleBand} we recover the following result for maximally modulated H\"ormander–Mihlin multipliers. This is in the spirit of Sj\"olin's work in \cite{Sjolin} and extends the implication in \cite[Theorem 1.6]{BDPPR} to all $ 1<p<\infty $. The proof is the one of \cite[Theorem 1.6]{BDPPR} hence we omit it. The reason why we are able to show boundedness for this range $ 1<p<\infty $ is that Theorem~\ref{Thm:MainSingleBand} holds for such a range, whereas \cite[Theorem 1.1]{BDPPR} only holds for $ p>2 $.
	\setcounter{TheoremCounter}{2}
	
	\begin{corollary}
		\label{Cor:CSO}
		Let $ m \in \mathcal{M}_{A}(d) $. Then, the Carleson--Sj\"olin operator
		\begin{equation}
			\CS (f) (x) \coloneqq \sup_{ N\in\R^d } \Abs{ \int_{\R^d} m(\eta+N) \widehat{f}(\eta) e^{2\pi i \langle x , \eta \rangle} d \eta }
		\end{equation}
		is bounded from $ L^p(\R^d) $ to $ L^p(\R^d) $ for every $ 1<p<\infty $.
	\end{corollary}

	\subsection{On the novelties, techniques and difficulties}
	\label{Ss:IntroNovelties}
	Theorem~\ref{Thm:Main} and Theorem~\ref{Thm:MainSingleBand} were already proved in dimension $n=2 $ for the directional Hilbert transform in \cite[Theorem 1]{BatemanThiele} and \cite[Theorem 2.1]{Bateman2013Revista} respectively. Indeed, as we have discussed above, in the latter papers the non-degeneracy assumption was removed in the specific context of the directional Hilbert transform in two dimensions. We have generalized these results to the context of codimension $1$ directional multipliers in arbitrary ambient dimension and to more general families of multipliers $\mathscr{M}$. In particular, our main results allow for variable families of H\"ormander-Mihlin multipliers provided the dependence is logarithmically H\"older in the sense of \eqref{Eq:MihlinFamDef}. These generalisations require the non-degeneracy assumption $\sigma(\R^n)\subset \Sigma_{1-\varepsilon}$ for some $\varepsilon>0$, as we show in Section \ref{S:Counterexamples}.
	
	A version of Theorem~\ref{Thm:MainSingleBand} appears in \cite[Theorem 1.1]{BDPPR}; the result in \cite{BDPPR} is stronger for $ p>2$, since the Grassmanian-valued map $\sigma(\cdot)$ is not assumed to only depend on the first $ n-1$ variables. However,  in the present paper we are able to extend the result of \cite{BDPPR} to $p \leq 2$ under the restriction that $\sigma(\cdot)$ does not depend on the last variable, combined with the non-degeneracy assumption. Finally, Corollary~\ref{Cor:CSO} extends the range $1<p<\infty$ from \cite[Theorem 1.6]{BDPPR}.
	
	The main techniques in our analysis involve a time-frequency discretization of the operators and a smart decomposition of tiles which leads to an improved maximal estimate for collections of trees. The discretization of the operator we are studying was already developed in \cite{BDPPR}. The argument of the improved maximal estimate relies on an estimate for a Kakeya-type maximal function. This idea goes back to \cite{LaceyLiMemoirs} and this was further developed in \cite{Bateman2013Revista}. 
	
	Our results deal exclusively with the codimension $1$ case, that is $ d=n-1$. For the more singular case of codimension $n-d>1$ there are some significant challenges to overcome. On the one hand, the time-frequency analysis framework for dealing with such operators has not yet been developed. The natural definition of tiles with $(n-d)$-dimensional space components will lead to trees defined via a top frequency $\omega_{\mathbf{T}}$ which will be a $d$-dimensional subspace of $\R^n$. The corresponding trees fail (as they should) the standard orthogonality estimates, (c.f. Lemma~\ref{Lem:OrthoLacTree}), since the space $L^2(\R^n)$ is subcritical for general codimension. For the same reason, the corresponding maximal estimates are not expected to be valid on $L^2(\R^n)$ but rather on $L^{n-d+1}(\R^n)$; see \cite{DPP22}.

	On a different direction, one would naturally wonder whether the arguments and general proof strategy of this paper could be formalized in the language of \emph{outer measure spaces}; see, for example, \cites{AmUr,DoTh}. It is not immediately apparent whether this is feasible in full generality as the size and density selection algorithms of the present paper are intertwined. Indeed, the size selection algorithm should always maintain the tops of high density. This is because the range $1<p<2$ of Theorem~\ref{Thm:MainSingleBand} is only accessible via \emph{maximal estimates for a Lipschitz/Kakeya operator}, originating in the work of Lacey and Li, \cite{LaceyLiMemoirs}, and Bateman, \cite{Bateman2013Revista}, which require trees with tops which simultaneously have large size and density. In the single annulus case of Theorem~\ref{Thm:MainSingleBand} the maximal estimate improves on the density estimate for a certain range of the size of the trees involved. We note here that these subtleties are relevant for the proof of Theorem~\ref{Thm:Main}, but their role is implicit through the use of Theorem~\ref{Thm:MainSingleBand} within the proof of Theorem~\ref{Thm:Main}. It would be natural to pursue such a method of proof firstly in the case $p>2$ where the maximal estimate is not required.

	\subsection{Past literature and related work}
	\label{Ss:IntroPast}
	
	Let $ \apl{v}{\R^n}{\mathbb{S}^{n-1}} $ be a vector field and write $ V = v(\R^n) $. Define the \emph{directional maximal function along $ v(\cdot) $}
	\begin{equation}
		M_{v(\cdot)} f(x) \coloneqq \sup_{ r>0 } \avgint_{-r}^{r} \Abs{ f(x-v(x) t) } dt, \qquad x \in \R^n,
	\end{equation}
	as well as the \emph{directional Hilbert transform along $ v(\cdot) $}
	\begin{equation}
		H_{v(\cdot)} f(x) \coloneqq \pv \avgint_{\R} \frac{f(x-v(x) t)}{t} dt, \qquad x \in \R^n.
	\end{equation}
	These operators are linearized versions of the corresponding maximal operators $ M_{V} $ and $ H_{V} $, defined by taking the supremum over the directions $ V \subseteq \mathbb{S}^{n-1} $. Note that for $ n=2 $, the operator $ H_{v(\cdot)} $ is a particular example of \eqref{Eq:SingOperLinDef}. However, in higher dimensions, $ H_{v(\cdot)} $ is a singular integral operator along codimension $ n-1 $ subspaces and not codimension $ 1 $ as considered here. We will write $ T_{m,V} $ to denote maximal singular integral operators along codimension $ n-1 $. This is a special case of \eqref{Eq:SingOperMaxDef} when $ d=1 $ and $ \mathscr{M} = \{m\} $.
	
	The existence of Kakeya sets prevents the $ L^p $-boundedness of these operators when we consider vector fields $ v(\cdot) $ without any regularity; see for example \cite{Bateman2013Revista}. Then, there have been three main research lines to study the $ L^p $-boundedness of these operators depending on different restrictions on $ v(\cdot) $: $ V $ finite, $ V $ infinite but with special structure and $ v(\cdot) $ with some regularity.
	
	\textbf{$V$ finite:}
	The goal in this case is to look for the best possible bounds in terms of the cardinality $ \#V $. For dimension $n=2 $ and $ M_V $ with $ V $ arbitrary but finite, the sharp bound in $ L^2(\R^2) $ is $ O( \log (\#V)) $, as proved in \cite{Katz}. This is the same bound as in the case that $ V $ is a $ \delta $-net on $ \mathbb{S}^{1} $ with $ \#V \simeq \delta^{-1} $. This implies the Kakeya maximal conjecture in dimension two, as shown in \cite{Cordoba}. For higher dimensions $ n\geq 3$, sharp bounds for the $ \delta $-net case are proved in \cite{Jongchon}. Other versions of $ M_{V} $ have also been studied. Sharp $ L^2(\R^n) $-bounds for the single-scale version of the operator $ M_{V} $ and $ V $ arbitrary are shown in \cite{DPP21}. Also, the same authors considered more general subspace averaging operators in \cite{DPP22}. 
	
	In the case of singular integrals, for $n=2 $ optimal $ L^2(\R^2) $-bounds with $ V $ arbitrary for $T_{m,V} $ have been proved in \cite{Demeter}. To get this result, the $n=2$ case of Theorem~\ref{Thm:MainSingleBand}, which was previously proved in \cite{LaceyLiTrans}, is used. In higher dimensions $n\geq 3$, sharp $ L^p(\R^n) $-bounds for $ H_{V}$ are shown in \cite{JM}. Regarding codimension $1$ operators, sharp $ L^2(\R^n) $-bounds for $ T^{\star}_{ \mathscr{M} } $ are proved in \cite{BDPPR}. 

	\textbf{$V$ lacunary:}
	Another line of investigation is to consider $ V $ with some special structure. Suppose that $ \# V = \infty $ but $ V $ is a lacunary set. For $ n=2 $, $ M_V $ is bounded on $ L^p(\R^2) $ for $ 1 < p \leq \infty $, as was proved in \cites{NSW,SS}. For higher dimensions, the authors in \cite{PR} extended the definition of lacunarity and proved that $ M_V $ is bounded on $ L^p(\R^n) $ if $ V $ is a lacunary set. 
	
	This result is different for the maximal directional Hilbert transform $ H_V $, which is unbounded on $ L^p(\R^n) $ whenever $\# V = \infty$ because of the existence of non-Kakeya counterexamples; thus $H_V$ is unbounded on $L^p(\R^n)$ even if $V$ is lacunary. This was shown for $ n=2$ in \cite{Karagulyan} and for $ n \geq 3 $ in \cite{LMP}. If one considers $ \# V $ finite and $ V $ contained in a first order lacunary set, sharp $L^p(\R^2)$-bounds for $T_{m,V}$ were proved in \cite{DDP} and for higher order lacunary directions for the Hilbert transform $ H_{V} $ in \cite{DPP18}. The three-dimensional case of the latter was proved in \cite{DPP20} and for any dimension $ n $ and codimension $ n-1 $ H\"ormander--Mihlin multiplier $T_{m,V}$ in \cite{ADPP}.
	
	\textbf{$v(\cdot)$ with some regularity:}
	Another research direction is to look for regularity assumptions on $ v(\cdot) $ that allow us to avoid Kakeya-type counterexamples. To avoid other radial counterexamples, it is necessary to consider truncated versions of $ M_{v(\cdot)} $ and $ H_{v(\cdot)} $. Kakeya counterexamples still show unboundedness when $ v(\cdot) $ is H\"older continuous with index $ < 1 $. In the case of the maximal function, Zygmund's conjecture states that the truncated version of $ M_{v(\cdot)} $ should be bounded on $ L^p(\R^2) $ for $ 2 < p \leq \infty $ and weak-type $ (2,2) $ if $ v(\cdot) $ is Lipschitz. The analogous conjecture for $ H_{v(\cdot)} $ is usually attributed to Stein; see for instance \cite{LaceyLiTrans}.
	 
	It is known that $ M_{v(\cdot)} $, \cite{Bourgain}, \cite{NSW1}, and $ H_{v(\cdot)}$, \cite{SteinStreet}, are bounded on $ L^2(\R^2) $ if $ v $ is real analytic. On the other hand, the authors in \cite{BatemanThiele} proved the equivalent of Theorem~\ref{Thm:Main} for $ n=2 $ and for the Hilbert transform $ H_{v(\cdot)} $. The dilation invariance of the Hilbert transform, together with the two-dimensional setting allows the authors in \cite{BatemanThiele} to remove the restriction to the frequency cone and the almost horizontal restriction of the vector field $ v(\cdot) $. This proof uses the equivalent result of Theorem~\ref{Thm:MainSingleBand} for $ n=2 $ which was proved in \cite{Bateman2013Revista}; see also \cite{LaceyLiTrans}. The authors in \cite{DPGTZ} prove a Littlewood--Paley diagonalization result if $ v(\cdot) $ is a Lipschitz vector field with small Lipschitz constant, which as a consequence gives an alternative proof of the result in \cite{BatemanThiele}. In \cites{Guo15,Guo17}, the author proves the equivalent of Theorem~\ref{Thm:Main} for $ n=2 $, loosening the restriction of the vector field depending only on the first variable to being constant on a suitable family of Lipschitz curves. This dependence on the first $ n-1 $ variables can be completely removed for $ p > 2 $ in the case of Theorem~\ref{Thm:MainSingleBand}, as was shown in \cite{LaceyLiTrans} for $ n=2 $ and in \cite{BDPPR} for $ n \geq 3 $.

	\subsection{Structure of the paper}
	\label{Ss:IntroStructure}
	
	In Section \ref{S:Model}, we will use time-frequency analysis to reduce the study of the operator $ T_{\mathscr{M},\sigma(\cdot)} \circ P_0 $ to that of a discretized model operator. The difficulty arises from the higher-dimensional geometry and the directional nature of the problem at hand. The discretization in \cite{BDPPR} will be used here. In Section \ref{S:Decom} we prove Theorem~\ref{Thm:MainSingleBand} by showing the desired estimate for the model operator constructed in Section \ref{S:Model}. Proving $ L^p $-estimates for $ p<2 $ requires a maximal estimate, an idea which was introduced in \cite{LaceyLiMemoirs} and further developed in \cites{Bateman09Proc,Bateman2013Revista,BatemanThiele}. This maximal estimate is only possible by making a more involved decomposition of the tiles than the one used for example in the proof of the Carleson theorem in \cite{LT}, as it was noted in \cite{Bateman2013Revista}. In Section \ref{S:MaxLemma} we prove the maximal estimate mentioned above in the codimension $ 1 $ setting. In Section \ref{S:Main}, we use Theorem~\ref{Thm:MainSingleBand} to prove Theorem~\ref{Thm:Main}. The method for achieving this reduction for $ n=2 $ and the directional Hilbert transform was developed in \cite{BatemanThiele}; here we develop it for any $ n $, codimension $ 1 $ and families of multipliers $ \mathscr{M} $. Finally, in Section \ref{S:Counterexamples}, we give some counterexamples which show the necessity of the non-degeneracy assumption $ \sigma(x) \in \Sigma_{1-\varepsilon} $, with $ \varepsilon > 0 $, in Theorems \ref{Thm:Main} and \ref{Thm:MainSingleBand}.

	\subsection*{Acknowledgments} 
	I would like to thank my PhD advisors Ioannis Parissis and Luz Roncal for their guidance throughout this project and for providing many helpful suggestions on the paper. I would also like to thank Francesco Di Plinio and Anastasios Fragkos for helpful discussions on the topic of the paper. I would also like to thank Gennady Uraltsev for several suggestions that helped improve the presentation of this paper.

	\section{Reduction to a model operator} 
	\label{S:Model}
	
	In this section, we will use time-frequency analysis to reduce the proof of Theorem~\ref{Thm:MainSingleBand} to an estimate for a model operator; see \eqref{Eq:ModelOperDef} below. The approach is in the spirit of the foundational work on time-frequency analysis in \cite{LT}, but adapted to our setting; the framework we need is the one developed in \cite{BDPPR}. In Subsection \ref{Ss:Tiles} we will define the collection of tiles needed for the discretization of the operator and in Subsection \ref{Ss:Adapted} we will define suitable wave packets which are time-frequency adapted to the tiles above. In Subsection \ref{Ss:Model} we define the model operator and state the desired estimate for it.
	
	We will fix the dimension $ n\geq 2 $, $ d=n-1 $, and the integrability exponent $ 1 < p < \infty $ for the rest of the section and we will take a number $ M \geq 0 $ sufficiently large depending on $ p $ and the dimension $ n $, which will be used as the decay and smoothness parameter for the wave packets.
	
	\subsection{Some basic reductions}
	\label{Ss:Reductions}
	
	First of all, note that when defining the singular integral \eqref{Eq:SingOperDef}, we have to choose a rotation $ V_{\sigma} \in \SO(n) $. The precise choice of such a rotation turns out to be irrelevant in the definition of the maximal operator \eqref{Eq:SingOperMaxDef}, as shown in \cite[Proposition 3.2]{BDPPR}, and one can select a smooth rotation $ V_{\sigma} $ mapping $ \sigma $ to $ \R^d $ for each $ \sigma $, as in \cite[Lemma 2.1]{BDPPR}. This choice is already implicit in the single tree estimate in \cite[Lemma 6.14]{BDPPR} which we will use in the present paper without further comment.

	On the other hand, by hypothesis $ \sigma(x) \in \Sigma_{1-\varepsilon} $, with $ \Sigma_{1-\varepsilon} $ as defined in \eqref{Eq:FreqConeDef}. For the proof and the construction of the model operator we will assume that $ \sigma(x) \in \Sigma_{\gamma} $ for some small dimensional constant $\gamma$. We can recover the result for all $ 0 < \varepsilon < 1 $ by applying the anisotropic rescaling 
	\begin{equation}
		D_{\lambda} ( \xi ) = D_{\lambda} ( \xi_1,\ldots,\xi_d,\xi_n ) \coloneqq \left( \frac{\xi_1}{\lambda},\ldots,\frac{\xi_d}{\lambda}, \xi_n \right) 
	\end{equation}
	for some small $ \lambda > 0 $. More precisely, assume that $ \sigma \in \Sigma_{1-\varepsilon} $ and write $ \sigma = \spn \{ u^1,\ldots,u^d \} $, for some orthonormal basis $ \{ u^1,\ldots,u^d \} \subseteq \mathbb{S}^d $. Then,
	\begin{equation}
		V_{\sigma} \Pi_{\sigma} D_{\lambda} ( \xi ) 
		= \left( u^1 \cdot D_{\lambda} ( \xi ) ,\ldots, u^d \cdot D_{\lambda} ( \xi ) \right)
		= \left( A_{\lambda}(u^{1}) u^{1,\lambda} \cdot \xi ,\ldots, A_{\lambda}(u^{d}) u^{d,\lambda} \cdot \xi \right),
	\end{equation} 
	where, for $u^{j}=(u^{j}_1,\ldots, u^j_n)$, $j\in \{1,\ldots, d\}$, we have defined
	\begin{equation}
		u^{j,\lambda} \coloneqq \frac{ 1 }{ A_{\lambda}(u^{j}) } \left( \frac{u^{j}_1}{\lambda},\ldots,\frac{u^{j}_d}{\lambda},u^{j}_n \right),
		\quad 
		A_{\lambda}(u^{j}) \coloneqq \Abs{ \left( \frac{u^{j}_1}{\lambda},\ldots,\frac{u^{j}_d}{\lambda},u^{j}_n \right) }, \qquad j=1,\ldots,d.
	\end{equation}
	Note that taking $\lambda =O(\gamma\varepsilon) $ we get that the hyperplane defined as $ \sigma^{\lambda} \coloneqq \spn \{ u^{1,\lambda},\ldots,u^{d,\lambda} \} $ is almost horizontal in the sense that $ \sigma^{\lambda} \in \Sigma_{\gamma} $. Note that $\sigma^\lambda$ only depends on the first $d$ coordinates of $x$, adopting the same property from $\sigma$. Now, for fixed $ \lambda>0 $, we need to see that the multiplier
	\begin{equation}
		m_{\sigma,\lambda}(\eta) \coloneqq m\left( A_{\lambda}(u^{1}) \eta_1,\ldots, A_{\lambda}(u^{d}) \eta_d \right)
	\end{equation}
	satisfies the conditions of Theorems \ref{Thm:Main} and \ref{Thm:MainSingleBand}, namely \eqref{Eq:MihlinFamDef}. This is the case with some $\lambda$-dependence in the constant  $\norm{ \{m_{\sigma,\lambda}\}_\sigma }{ \mathcal{M}_A( \Gr(d,n) ) }$  of \eqref{Eq:MihlinFamDef},  because of the anisotropic rescaling. Thus, when we apply this rescaling and the results for the almost horizontal case, we recover our main results for the case $ \norm{ v_{\sigma}-e_n}{} < 1-\varepsilon $, with the constant blowing up with $ \varepsilon $, as $ \varepsilon \to 0 $. 
	
	Thus, we can assume that the hyperplanes $ \sigma(x) \in \Sigma_{\gamma} $ for some small $ \gamma $ and then the singularities of the operator $ T_{m_{\sigma(x)}} $ are contained on the cone $ \Gamma_{\gamma} $. Since $ ( \Id - P_{\cone,\gamma} ) \circ  P_0 f $ is frequency supported away from $ \Gamma_{\gamma} $, we have the estimate
	\begin{equation}
		\sup_{ \sigma \in \Sigma_{\gamma} } \Abs{ T_{ m_{\sigma} } [ ( \Id - P_{\cone,\gamma} ) \circ P_0 f ] } \lesssim M f,
	\end{equation}
	where $M$ denotes the Hardy--Littlewood maximal operator. Thus, the proof of Theorem~\ref{Thm:MainSingleBand} is reduced to the proof of the corresponding  bound for the maximal operator
	\begin{equation}
		\label{Eq:ConeReduction}
		T_{ \mathscr{M},\Sigma_{ \gamma } }^{\star,\cone} f (x) \coloneqq \sup_{ \substack{ \sigma \in \Sigma_{\gamma} \\ \sigma(\underline{x},x_n) = \sigma(\underline{x})}  } \Abs{ T_{ m_{\sigma} } ( P_0  P _{\cone,\gamma} f ) ( x; \sigma, V_{\sigma} ) }, \quad x=(x_1,\ldots,x_d,x_n) = (\underline{x},x_n) \in \R^n.
	\end{equation}
	Observe that the multiplier operators $ g \mapsto T_{ m_{\sigma} } ( P_0  P_{\cone,\gamma} g ) $ have Fourier support in the subset
	\begin{equation}
		\Delta \coloneqq \widetilde{\Gamma}_{\gamma} \cap \left\{ \xi \in \R^n :\, 1 < \abs{ \xi } < 3 \right\},
	\end{equation}
	where $ \widetilde{\Gamma}_{\gamma} $ is defined in \eqref{Eq:FreqProjConeDef}; see Figure \ref{Fig:SupportConeReduction}. Finally, write $\Delta' = \{ \xi' :\, \xi \in \Delta \} \subset \mathbb{S}^d $ and let $ \widetilde{ \Delta' } \subset \mathbb{S}^d $ be the $ 3^{6} \sqrt{d} $-neighborhood of $ \Delta' $ in $ \mathbb{S}^d $. 
	
	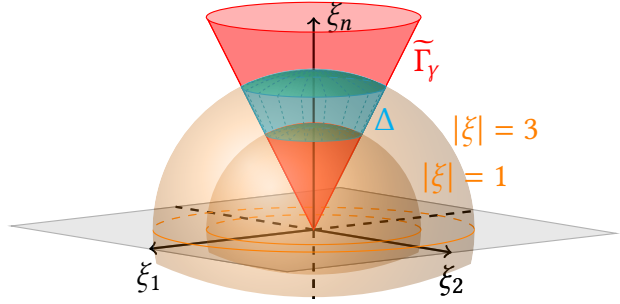
\begin{figure}
		\tdplotsetmaincoords{82}{140} 
		
		\resizebox{0.56\textwidth}{!}{%
			\begin{tikzpicture}[tdplot_main_coords,scale=1.36,
				bullet/.style={fill,circle,inner
					sep=1pt},sphere segment/.style={
					/utils/exec=\tikzset{3d stuff/sphere segment/.cd,#1}%
					\pgfkeys{/tikz/3d stuff/parse domain/.expanded=\pgfkeysvalueof{/tikz/3d stuff/sphere segment/phi}}
					\edef\phimin{\xmin}
					\edef\phimax{\xmax}
					\pgfkeys{/tikz/3d stuff/parse domain/.expanded=\pgfkeysvalueof{/tikz/3d stuff/sphere segment/theta}}
					\edef\thetamin{\xmin}
					\edef\thetamax{\xmax},
					insert path={%
						plot[variable=\x,smooth,domain=\phimax:\phimin] 
						(xyz spherical cs:radius=\pgfkeysvalueof{/tikz/3d stuff/sphere segment/r},
						longitude=\x,latitude=\thetamin)
						-- plot[variable=\x,smooth,domain=\thetamin:\thetamax] 
						(xyz spherical cs:radius=\pgfkeysvalueof{/tikz/3d stuff/sphere segment/r},
						longitude=\phimin,latitude=\x)
						--plot[variable=\x,smooth,domain=\phimin:\phimax] 
						(xyz spherical cs:radius=\pgfkeysvalueof{/tikz/3d stuff/sphere segment/r},
						longitude=\x,latitude=\thetamax)
						-- plot[variable=\x,smooth,domain=\thetamax:\thetamin] 
						(xyz spherical cs:radius=\pgfkeysvalueof{/tikz/3d stuff/sphere segment/r},
						longitude=\phimax,latitude=\x) --cycle}}]
				
				\clip[tdplot_screen_coords] (-3.1, -0.64) rectangle (3.1, 2.15);
				
				\coordinate (O) at (0,0,0);
				
				\begin{scope}
					
					\filldraw[draw=gray, fill=gray!30, opacity=0.6] 
					(-2,-2,0) -- (-2,2,0) -- (2,2,0) -- (2,-2,0) -- cycle;
					
					\draw[canvas is xy plane at z=2, draw=red, fill=red, fill opacity=0.3] (-30:1) arc (-30:-230:1) -- (O) --cycle;
					
					\draw[dashed,black,thick] (0,0,0) -- (0,0,-1.6);
					\draw[thick,->] (0,0,0) -- (0,0,2) node[anchor=west]{$\xi_n$};
					
					\draw[dashed,black,thick] (0,0,0) -- (-2,0,0);
					\draw[thick,->,black] (0,0,0) -- (2,0,0) node[anchor=north]{$\xi_1$};
					
					\draw[dashed,black,thick] (0,0,0) -- (0,-2,0);
					\draw[thick,->,black] (0,0,0) -- (0,2,0) node[anchor=north]{$\xi_2$};
					
					\begin{scope}
						\clip (0,0,0) circle (1.5cm); 
						
						\draw[orange, dashed] 
						plot[domain=140:320,smooth,variable=\t]
						({1.5*cos(\t)},{1.5*sin(\t)}, 0); 
						
						\draw[orange, dashed] 
						plot[domain=140:320,smooth,variable=\t]
						({cos(\t)},{sin(\t)}, 0); 
					\end{scope}
					
					\draw[canvas is xy plane at z=2, draw=red, fill=red, fill opacity=0.3] (-30:1) arc (-30:130:1) -- (O) --cycle;
					
				\end{scope}

				\begin{scope}
					
					\clip (0,0,0) circle (1.5cm); 
					
					\draw[canvas is xy plane at z=2, draw=white, fill=white, fill opacity=1] (-30:1) arc (-30:130:1) -- (O) --cycle;
					
					\filldraw[draw=gray, fill=gray!30, opacity=0.6] 
					(-2,-2,0) -- (-2,2,0) -- (2,2,0) -- (2,-2,0) -- cycle;
					
					\draw[canvas is xy plane at z=2, draw=cyan, fill=cyan, fill opacity=0.4] (-30:1) arc (-30:-230:1) -- (O) --cycle;
					
					\draw[dashed,black,thick] (0,0,0) -- (0,0,-1.6);
					\draw[thick,->] (0,0,0) -- (0,0,2);
					
					\draw[dashed,black,thick] (0,0,0) -- (-2,0,0);
					\draw[thick,->,black] (0,0,0) -- (2,0,0) node[anchor=north]{$\xi_1$};
					
					\draw[dashed,black,thick] (0,0,0) -- (0,-2,0);
					\draw[thick,->,black] (0,0,0) -- (0,2,0) node[anchor=north]{$\xi_2$};
					
					\begin{scope}
						\clip (0,0,0) circle (1.5cm); 
						
						\draw[orange, dashed] 
						plot[domain=140:320,smooth,variable=\t]
						({1.5*cos(\t)},{1.5*sin(\t)}, 0); 
						
						\draw[orange, dashed] 
						plot[domain=140:320,smooth,variable=\t]
						({cos(\t)},{sin(\t)}, 0); 
					\end{scope}
					
					\begin{scope}
						\clip[sphere segment={r=1, phi=140:300, theta=63.43495:90}]; 
						\shade[tdplot_screen_coords,ball color=cyan,opacity=0.8] 
						(0,0,0) circle[radius=1];
					\end{scope}
					
					\begin{scope}
						\clip[sphere segment={r=1, phi=-60:140, theta=63.43495:90}]; 
						\shade[tdplot_screen_coords,ball color=cyan,opacity=0.8] 
						(0,0,0) circle[radius=1];
					\end{scope}
					
					\draw[canvas is xy plane at z=2, draw=cyan, fill=cyan, fill opacity=0.4] (-30:1) arc (-30:130:1) -- (O) --cycle;
					
					\begin{scope}
						\clip[sphere segment={r=1.5, phi=140:300, theta=63.43495:90}]; 
						\shade[tdplot_screen_coords,ball color=cyan,opacity=0.8] 
						(0,0,0) circle[radius=1.5];
					\end{scope}
					
					\begin{scope}
						\clip[sphere segment={r=1.5, phi=-60:140, theta=63.43495:90}]; 
						\shade[tdplot_screen_coords,ball color=cyan,opacity=0.8] 
						(0,0,0) circle[radius=1.5];
					\end{scope}
					
					\def\alp{140}
					\def\bet{-40}
					
					\draw[cyan] ({cos(\alp)/sqrt(5)}, {sin(\alp)/sqrt(5)},{2/sqrt(5)}) -- ({1.5*cos(\alp)/sqrt(5)}, {1.5*sin(\alp)/sqrt(5)},{3/sqrt(5)});
					
					\draw[cyan] ({cos(\bet)/sqrt(5)}, {sin(\bet)/sqrt(5)},{2/sqrt(5)}) -- ({1.5*cos(\bet)/sqrt(5)}, {1.5*sin(\bet)/sqrt(5)},{3/sqrt(5)});
					
					\draw[cyan, dashed] 
					plot[domain=140:320,smooth,variable=\t]
					({cos(\t)/sqrt(5)},{sin(\t)/sqrt(5)},{2/sqrt(5)});
					
					\draw[cyan, dashed] 
					plot[domain=140:320,smooth,variable=\t]
					({1.5*cos(\t)/sqrt(5)},{1.5*sin(\t)/sqrt(5)},{3/sqrt(5)});
					
					\foreach \Angle in {-40,-20,...,140}   
					\tdplotsetthetaplanecoords{\Angle}
					\tdplotdrawarc[tdplot_rotated_coords, cyan!80!black, very thin, dash pattern=on 1pt off 1pt]
					{(0,0,0)}{1.5}{90-63.43495}{0}{anchor=south west}{};
					
					\foreach \Angle in {-40,-20,...,140}   
					\tdplotsetthetaplanecoords{\Angle}
					\tdplotdrawarc[tdplot_rotated_coords,cyan!80!black,very thin,dash pattern=on 1pt off 1pt]
					{(0,0,0)}{1}{90-63.43495}{0}{anchor=south west}{};
					
					\foreach \Angle in {-40,-20,...,140}    
					\draw[cyan!80!black, very thin,dash pattern=on 1pt off 1pt] ({cos(\Angle)/sqrt(5)}, {sin(\Angle)/sqrt(5)},{2/sqrt(5)}) -- ({1.5*cos(\Angle)/sqrt(5)}, {1.5*sin(\Angle)/sqrt(5)},{3/sqrt(5)});
					
					\draw[cyan] 
					plot[domain=-40:140,smooth,variable=\t]
					({cos(\t)/sqrt(5)},{sin(\t)/sqrt(5)},{2/sqrt(5)});
					
					\draw[cyan] 
					plot[domain=-40:140,smooth,variable=\t]
					({1.5*cos(\t)/sqrt(5)},{1.5*sin(\t)/sqrt(5)},{3/sqrt(5)});
					
				\end{scope}

				\begin{scope}
					
					\clip (0,0,0) circle (1cm); 
					
					\draw[canvas is xy plane at z=2, draw=white, fill=white, fill opacity=1] (-30:1) arc (-30:130:1) -- (O) --cycle;
					
					\filldraw[draw=gray, fill=gray!30, opacity=0.6] 
					(-2,-2,0) -- (-2,2,0) -- (2,2,0) -- (2,-2,0) -- cycle;
					
					\draw[canvas is xy plane at z=2, draw=red, fill=red, fill opacity=0.4] (-30:1) arc (-30:-230:1) -- (O) --cycle;
					
					\draw[dashed,black,thick] (0,0,0) -- (-2,0,0);
					\draw[thick,->,black] (0,0,0) -- (2,0,0) node[anchor=north]{$\xi_1$};
					
					\draw[dashed,black,thick] (0,0,0) -- (0,-2,0);
					\draw[thick,->,black] (0,0,0) -- (0,2,0) node[anchor=north]{$\xi_2$};
					
					\begin{scope}
						\clip (0,0,0) circle (1.5cm); 
						
						\draw[orange, dashed] 
						plot[domain=140:320,smooth,variable=\t]
						({1.5*cos(\t)},{1.5*sin(\t)}, 0); 
						
						\draw[orange, dashed] 
						plot[domain=140:320,smooth,variable=\t]
						({cos(\t)},{sin(\t)}, 0); 
					\end{scope}
					
					\draw[dashed,black,thick] (0,0,0) -- (0,0,-1.6);
					\draw[thick,->] (0,0,0) -- (0,0,2);

					\draw[canvas is xy plane at z=2, draw=red, fill=red, fill opacity=0.4] (-30:1) arc (-30:130:1) -- (O) --cycle;
					
					\begin{scope}
						\clip[sphere segment={r=1, phi=140:300, theta=63.43495:90}]; 
						\shade[tdplot_screen_coords,ball color=cyan,opacity=0.8] 
						(0,0,0) circle[radius=1];
					\end{scope}
					
					\begin{scope}
						\clip[sphere segment={r=1, phi=-60:140, theta=63.43495:90}]; 
						\shade[tdplot_screen_coords,ball color=cyan,opacity=0.8] 
						(0,0,0) circle[radius=1];
					\end{scope}
					
					\draw[cyan, dashed] 
					plot[domain=140:320,smooth,variable=\t]
					({cos(\t)/sqrt(5)},{sin(\t)/sqrt(5)},{2/sqrt(5)});
					
					\draw[cyan, dashed] 
					plot[domain=140:320,smooth,variable=\t]
					({1.5*cos(\t)/sqrt(5)},{1.5*sin(\t)/sqrt(5)},{3/sqrt(5)});
					
					\foreach \Angle in {-40,-20,...,140}   
					\tdplotsetthetaplanecoords{\Angle}
					\tdplotdrawarc[tdplot_rotated_coords, cyan!80!black, very thin, dash pattern=on 1pt off 1pt]
					{(0,0,0)}{1.5}{90-63.43495}{0}{anchor=south west}{};
					
					\foreach \Angle in {-40,-20,...,140}   
					\tdplotsetthetaplanecoords{\Angle}
					\tdplotdrawarc[tdplot_rotated_coords,cyan!80!black,very thin,dash pattern=on 1pt off 1pt]
					{(0,0,0)}{1}{90-63.43495}{0}{anchor=south west}{};
					
					\foreach \Angle in {-40,-20,...,140}    
					\draw[cyan!80!black, very thin,dash pattern=on 1pt off 1pt] ({cos(\Angle)/sqrt(5)}, {sin(\Angle)/sqrt(5)},{2/sqrt(5)}) -- ({1.5*cos(\Angle)/sqrt(5)}, {1.5*sin(\Angle)/sqrt(5)},{3/sqrt(5)});
					
					\draw[cyan] 
					plot[domain=-40:140,smooth,variable=\t]
					({cos(\t)/sqrt(5)},{sin(\t)/sqrt(5)},{2/sqrt(5)});
					
					\draw[cyan] 
					plot[domain=-40:140,smooth,variable=\t]
					({1.5*cos(\t)/sqrt(5)},{1.5*sin(\t)/sqrt(5)},{3/sqrt(5)});
					
				\end{scope}
				
				\begin{scope}
					\clip (0,0,0) circle (1.5cm); 
					\clip (0,0,2) circle (2.4cm);
					
					\shade[ball color = orange, opacity=0.3] (0,0,0) circle (1cm); 
				\end{scope}
				
				\begin{scope}
					\clip (0,0,0) circle (1.5cm); 
					
					\draw[orange] 
					plot[domain=-40:140,smooth,variable=\t]
					({cos(\t)},{sin(\t)}, 0); 
				\end{scope}
				
				\begin{scope}
					\clip (0,0,0) circle (1.5cm); 
					\clip (0,0,3) circle (3.6cm);
					\shade[ball color = orange, opacity=0.3] (0,0,0) circle (1.5cm); 
				\end{scope}

				\begin{scope}
					\clip (0,0,0) circle (1.5cm); 
					
					\draw[orange] 
					plot[domain=-40:140,smooth,variable=\t]
					({1.5*cos(\t)},{1.5*sin(\t)}, 0); 
				\end{scope}
				
				\node[red] at (-0.75,0.75,1.6) {$\widetilde{\Gamma}_{\gamma}$}; 
				\node[cyan] at (-0.47,0.47,1.05) {$\Delta$}; 
				\node[orange] at (-1,1,0.5) {$\abs{\xi}=1$}; 
				\node[orange] at (-1.2,1.2,0.95) {$\abs{\xi}=3$}; 
				
			\end{tikzpicture}
		}
		\caption{The Fourier support of $ g \mapsto T_{ m_{\sigma} } ( P_0  P_{\cone,\gamma} g ) $.}
		\label{Fig:SupportConeReduction}
	\end{figure}

	\subsection{Grids and tiles} 
	\label{Ss:Tiles} 
	
	In this subsection we define the tiles that we are going to use to construct the model operator \eqref{Eq:ModelOperDef}. The tiles are tensor products $ t = R_t \times \omega_t $, formed by the space component $ R_t $ and the frequency component $ \omega_t $. These will be defined by means of triadic grids, which we use throughout the paper.
	
	\begin{definition}[Triadic grid]
		Given a positive integer $ \kappa $ we define a $ 3^{\kappa} $-separated \emph{triadic grid} $\mathcal{G}$ on $\R^d$ to be a collection of cubes in $ \R^d$ satisfying the following properties:
		\begin{enumerate} [label=(g\arabic*)]
			\item \label{Item:Grid1} (quantized length) $\mathcal{G} =\bigcup_{j\in \mathbb Z} \mathcal{G}_{j}$, where $\mathcal{G}_{j}=\{ Q \in \mathcal{G}: \,\ell(Q) =3^{-j \kappa} \} $;
			\item \label{Item:Grid2} (partition) $ \R^d = \bigcup_{Q\in \mathcal{G}_{j} } Q$ for all $j\in \mathbb Z$;
			\item \label{Item:Grid3} (grid) $ Q \cap Q' \in\{ Q, \, Q', \, \emptyset \}$ for every $ Q, Q' \in \mathcal{G} $.
		\end{enumerate}
	\end{definition}
	
	Above, $ \ell(Q) $ denotes the sidelength of the cube $ Q $. Also, we will denote by $ c(Q) \in \R^d $ the center of $Q$. Below we will fix $ \kappa $ to be a large dimensional constant, for example $ \kappa > 3^3 + \log_{3} d $ will suffice. Note that the separation of scales assumption implies that if $ Q, Q' \in \mathcal{G} $ with $ Q \subsetneqq Q' $, then $ \ell(Q) < 3^{-\kappa} \ell(Q') $. We will denote by $ Q^{(i)} $, $i\in \mathbb{N}$, the $i$-th triadic parent of $ Q $ and by $ \children(Q) $ the collection of triadic children of $Q$. There exists a unique $ Q^{\circ} \in \children(Q) $ which is concentric with $ Q $ which we will call the \emph{center} of $ Q $. On the other hand, we define the \emph{peripheral children} of $Q$ as
	\begin{equation}
		\children_{\square} (Q) 
		\coloneqq \children(Q) \setminus \{ Q^{\circ} \}.
	\end{equation}
	We can and will adopt a universal enumeration in the form
	\begin{equation}
		\children_{\square} (Q) 
		\coloneqq \bigcup_{ \tau=1 }^{3^{\kappa d}-1} Q_{\tau}.
	\end{equation}
	Here \emph{universal} means that for $ Q,Q' \in \mathcal{G} $ and fixed $ \tau \in \{1,\ldots,3^{\kappa d}-1\} $, the cubes $ Q_{\tau}, Q_{\tau}' $ are in the same \emph{relative positions} inside $Q,Q'$, respectively.
	
	We will use the notation $ \xi' = \xi/\abs{\xi} \in \mathbb{S}^d $ for $ \xi \in \R^n $ and $ S' = \{ \xi' :\, \xi \in S \} \subseteq \mathbb{S}^d $ for $ S \subseteq \R^n $.
	
	\begin{definition}[Projected grids]
		Given a triadic grid $ \mathcal{G} $ in $ \R^d $ consider 
		\begin{equation}
			\mathcal{G} \times \{1\} \coloneqq \left\{ Q \times \{1\} : \, Q \in \mathcal{G}, (Q \times \{1\})' \subseteq \widetilde{\Delta'} \right\}.
		\end{equation}
		Note that this is a subcollection of a copy of the triadic grid $ \mathcal{G} $ living on the affine copy  of $\R^d$ 
		$$
		 e_n + \R^d \subseteq \R^n.
		 $$
		  The \emph{projected grid} we will need is
		\begin{equation}
			(\mathcal{G} \times \{1\})' \coloneqq \left\{ (Q \times \{1\})' :\, Q \times \{1\} \in \mathcal{G} \times \{1\} \right\}  \subseteq \mathbb{S}^d,
		\end{equation}
		which is a collection of caps in a small neighborhood of $ e_n $ on $ \mathbb{S}^d $. It is obvious that $ (\mathcal{G} \times \{1\})' $ inherits the grid property from $ \mathcal{G} $.
	\end{definition}
	
	Below for $ v \in \mathbb{S}^d $ denote by $ \sigma_v $ the unique element $ \sigma \in \Grdn $ such that $ v_{\sigma} = v $.
	
	\begin{definition}[Parallelepipeds in $\R^n$]
		Let $\mathcal{L}$ be a triadic grid in $\R^d$, $ \mathcal{I} $ a triadic grid in $ \R $ and $ V \subseteq \mathbb{S}^d $ a collection of directions. For $ L \in \mathcal{L} $, $ I \in \mathcal{I} $ and $ v \in V $ a direction; define the \emph{parallelepiped} $ R(L,I,v) $ by
		\begin{equation}
			\label{Eq:DefParallelepiped}
			R(L,I,v) \coloneqq \left( L \times \R \right) \cap \left( \sigma_{v} + I \right),
		\end{equation}
		where above we are using the notation $ \sigma_{v} + I \coloneqq \left\{ y + (0,\ldots,0,s) :\, y \in \sigma_{v}, \, s \in I \right\} $; see Figure \ref{Fig:TileSpace}.
	\end{definition}
	
	Given $ R = R(L,I,v) $, we will call $L=L_{R}$ the \emph{horizontal component}, $I=I_{R}$ the \emph{vertical component} and $ v=v_{R} $ the \emph{direction} of $R$. The \emph{scale} of $ R $ is defined by $ \scl(R) \coloneqq \ell(L) $.
	
	Let $ V_{\Delta} \coloneqq \{(c(Q) \times \{1\})' : \, Q \times \{1\} \in \mathcal{G} \times \{1\} \} $. Define the collection of \emph{admissible} parallelepipeds (with respect to given grids $ \mathcal{L}, \mathcal{I}, \mathcal{G} $) as
	\begin{equation}
		\mathcal{R}_{\mathcal{L},\mathcal{I},V} \coloneqq \left\{ R(L,I,v) : \, L \in \mathcal{L}, \, I \in \mathcal{I}, \, v \in V_{\Delta}, \ell(L) > \min\{ \ell(Q)^{-1} : \, Q \times \{1\} \in \mathcal{G} \times \{1\} \} \right\}.
	\end{equation}
	Given $ R = R(L,I,v) $ an admissible parallelepiped, the \emph{eccentricity} of $ R $ is defined as
	\begin{equation}
		\ecc(R) \coloneqq (Q \times \{1\})' \subseteq \mathbb{S}^d,
	\end{equation}
	where $ Q \in \mathcal{G} $ is the unique triadic cube with $ (Q \times \{1\})' \in (\mathcal{G} \times \{1\})' $ and such that $ (c(Q) \times \{1\})' = v $ and $ \ell(Q) = \ell(I) / \ell(L) $, see Figure \ref{Fig:Tile}.
	
	\begin{definition}[Tiles]
		A tile $ t = R_t \times \omega_t $ is a cartesian product of an admissible parallelepiped $ R_t = R(L_t,I_t,v_t) $ and a set $ \omega_t \subseteq \R^n $ defined as
		\begin{equation}
			\omega_t \coloneqq \left\{ \xi \in \R^n :\, \ell(I_t)^{-1}/2 < \abs{\xi} < 6 \ell(I_t)^{-1} , \, \xi' \in \ecc(R_t) \right\}.
		\end{equation}
		Note that $ \abs{ R_t } \abs{ \omega_t } \simeq 1 $ with the implicit constant depending only upon dimension. We define
		\begin{equation}
			\ecc(t) \coloneqq \ecc(R_t) \coloneqq (Q_t \times \{1\})'
		\end{equation}
		which also implicitly defines $e_n+ \R^d \supseteq Q_t \in \mathcal{G} $ and by $ \ecc(t)^{\circ} = (Q_{t}^{\circ} \times \{1\})' $ the center of the eccentricity. The \emph{directional support} of $ t $ is defined as
		\begin{equation}
			\label{Eq:DirecSuppDef}
			\alpha_t 
			\coloneqq \bigcup_{\tau=1}^{3^{\kappa d}-1} (Q_{t,\tau} \times \{1\})'
			\eqcolon \bigcup_{\tau=1}^{3^{\kappa d}-1} \alpha_{t,\tau}
			= \ecc(t) \setminus \ecc(t)^{\circ}.
		\end{equation}
		Define also $ \scl(t) \coloneqq \scl(R_t) = \ell(L_t) = \ell(I_t) / \ell(Q_t) $ the \emph{scale} of $ t $ and $ \ann(t) \coloneqq \abs{I_t}^{-1} \in 3^{\Z} $ the \emph{annulus} of $ t $. The \emph{center} of the frequency component $ \omega_t $ is
		\begin{equation}
			\omega_t^{\circ} \coloneqq \left\{ \xi \in \R^n : \, \abs{I_t}^{-1}/2 < \abs{\xi} < 6 \abs{I_t}^{-1} , \, \xi' \in \ecc(t)^{\circ} \right\}.
		\end{equation}
		We will denote by $ \mathcal{T} $ the collection of all tiles, as defined above, and for $ \ann \in 3^{\Z} $ we set
		\begin{equation}
			\mathcal{T}_{\ann} \coloneqq \left\{ t \in \mathcal{T} : \, \ann(t) = \ann \right\}
		\end{equation}
		to be a \emph{fixed annulus} collection of tiles.
	\end{definition}
	
	Note that the collections $ \{ \omega_t :\, t \in \mathcal{T}_{\ann} \} $ and $ \{ \alpha_{t,\tau}: \,t\in \mathcal{T}_{\ann} \} $ adopt the grid property \ref{Item:Grid3} from $ \mathcal{G} \times \{1\}, (\mathcal{G} \times \{1\})' $ when the annulus $ \ann $ is fixed.
	
	\begin{remark}
		The tiles $ \mathcal{T} $ we are considering depend of course on the choice of grids $ \mathcal{L} $, $ \mathcal{I} $ and $ \mathcal{G} $. We typically fix such a choice and suppress it from the notation as it is an unnecessary distraction. There are only a couple of cases in the arguments of this paper where we will have to use tiles with respect to several (shifted) triadic grids and the relevant families will be clear from the context.
	\end{remark}

	\begin{figure}
		\centering
		\begin{subfigure}[b]{0.58\textwidth}
			\tdplotsetmaincoords{70}{170} 
			\resizebox{\textwidth}{!}{%
				\begin{tikzpicture}[scale=1.33,tdplot_main_coords]
					
					\clip[tdplot_screen_coords] (-3.8, -1.65) rectangle (3.2, 3.55);
					
					\coordinate (O) at (0,0,0);
					
					\draw[draw=gray, draw opacity=0.6, fill=gray!30, fill opacity=0.6] 
					({asin(-0.2)}:3) arc ({asin(-0.2)}:{acos(-0.2)}:3) -- (-0.6,-0.6,0) -- cycle;
					
					\draw[dashed,black,thick] (0,0,0) -- (0,0,-0.6);
					\draw[thick,-,black] (0,0,0) -- (0,0,3) ;
					
					\draw[dashed,black,thick] (0,0,0) -- (-0.6,0,0);
					\draw[thick,-,black] (0,0,0) -- (3,0,0) node[anchor=north east]{};
					
					\draw[dashed,black,thick] (0,0,0) -- (0,-0.6,0);
					\draw[thick,-,black] (0,0,0) -- (0,3,0) node[anchor=north west]{};
					
					\begin{scope}
						\clip (0,0,0) circle (3cm); 
						\clip (0,0,2) circle (2.705cm); 
						\clip (0,2,0) circle (2.705cm);
						\clip (2,0,0) circle (2.705cm);
						
						\shade[ball color = orange, opacity=0.3] (0,0,0) circle (1cm);
					\end{scope}
					
					\begin{scope}
						\clip (0,0,0) circle (3cm); 
						
						\draw[orange] 
						plot[domain={asin(-0.6)}:{acos(-0.6)},smooth,variable=\t]
						({cos(\t)},{sin(\t)}, 0);
					\end{scope}
					
					\begin{scope}
						\clip (0,0,0) circle (3cm); 
						\clip (0,0,4) circle (4.875cm); 
						\clip (0,4,0) circle (4.875cm);
						\clip (4,0,0) circle (4.875cm);
						
						\shade[ball color = orange, opacity=0.3] (0,0,0) circle (2cm);
					\end{scope}
					
					\begin{scope}
						\clip (0,0,0) circle (3cm); 
						
						\draw[orange] 
						plot[domain={asin(-0.3)}:{acos(-0.3)},smooth,variable=\t]
						({2*cos(\t)},{2*sin(\t)}, 0);
					\end{scope}
					
					\begin{scope}
						\clip (0,0,0) circle (3cm); 
						\clip (0,0,6) circle (7.0385cm); 
						\clip (0,6,0) circle (7.0385cm);
						\clip (6,0,0) circle (7.0385cm);
						
						\shade[ball color = orange, opacity=0.3] (0,0,0) circle (3cm);
					\end{scope}

					\draw[draw=gray, draw opacity=0.6, fill=gray!30, fill opacity=0.6] 
					({asin(-0.2)}:3) arc ({asin(-0.2)}:{acos(-0.2)}:3) -- (-0.6,3.5,0) -- (3.5,3.5,0) -- (3.5,-0.6,0) -- cycle;
					
					\filldraw[draw=red,fill=red,fill opacity=0.7,draw opacity=1]
					(1,0,0) -- (1,1,0) -- (2,1,0) -- (2,0,0) -- cycle;
					
					\filldraw[draw=blue, fill=blue,fill opacity=0.7,draw opacity=1] 
					(1.4,0.4,0) -- (1.4,0.6,0) -- (1.6,0.6,0) -- (1.6,0.4,0) -- cycle;
					
					\filldraw[draw=darkgreen, fill=darkgreen,fill opacity=0.7,draw opacity=1] 
					(1,0.8,0) -- (1,1,0) -- (1.2,1,0) -- (1.2,0.8,0) -- cycle;
					
					\filldraw[draw=red,fill=red,fill opacity=0.7,draw opacity=1]
					(1,0,1) -- (1,1,1) -- (2,1,1) -- (2,0,1) -- cycle;
					
					\filldraw[draw=blue, fill=blue, opacity=0.6] 
					(1.4,0.4,1) -- (1.4,0.6,1) -- (1.6,0.6,1) -- (1.6,0.4,1) -- cycle;
					
					\filldraw[draw=darkgreen, fill=darkgreen,opacity=0.6] 
					(1,0.8,1) -- (1,1,1) -- (1.2,1,1) -- (1.2,0.8,1) -- cycle;
					
					\coordinate (S11) at ({2*1/sqrt(2)},{2*0/sqrt(2)},{2*1/sqrt(2)});
					\coordinate (S21) at ({2*1/sqrt(3)},{2*1/sqrt(3)},{2*1/sqrt(3)});
					\coordinate (S31) at ({2*2/sqrt(6)},{2*1/sqrt(6)},{2*1/sqrt(6)});
					\coordinate (S41) at ({2*2/sqrt(5)},{2*0/sqrt(5)},{2*1/sqrt(5)});
					
					\coordinate (S12) at ({3*1/sqrt(2)},{3*0/sqrt(2)},{3*1/sqrt(2)});
					\coordinate (S22) at ({3*1/sqrt(3)},{3*1/sqrt(3)},{3*1/sqrt(3)});
					\coordinate (S32) at ({3*2/sqrt(6)},{3*1/sqrt(6)},{3*1/sqrt(6)});
					\coordinate (S42) at ({3*2/sqrt(5)},{3*0/sqrt(5)},{3*1/sqrt(5)});
					
					\coordinate (S13) at ({1/sqrt(2)},{0/sqrt(2)},{1/sqrt(2)});
					\coordinate (S23) at ({1/sqrt(3)},{1/sqrt(3)},{1/sqrt(3)});
					\coordinate (S33) at ({2/sqrt(6)},{1/sqrt(6)},{1/sqrt(6)});
					\coordinate (S43) at ({2/sqrt(5)},{0/sqrt(5)},{1/sqrt(5)});
					
					\coordinate (Q11) at ({2*1.4/sqrt(3.12)},{2*0.4/sqrt(3.12)},{2*1/sqrt(3.12)});
					\coordinate (Q21) at ({2*1.4/sqrt(3.32)},{2*0.6/sqrt(3.32)},{2*1/sqrt(3.32)});
					\coordinate (Q31) at ({2*1.6/sqrt(3.92)},{2*0.6/sqrt(3.92)},{2*1/sqrt(3.92)});
					\coordinate (Q41) at ({2*1.6/sqrt(3.72)},{2*0.4/sqrt(3.72)},{2*1/sqrt(3.72)});
					
					\coordinate (Q12) at ({3*1.4/sqrt(3.12)},{3*0.4/sqrt(3.12)},{3*1/sqrt(3.12)});
					\coordinate (Q22) at ({3*1.4/sqrt(3.32)},{3*0.6/sqrt(3.32)},{3*1/sqrt(3.32)});
					\coordinate (Q32) at ({3*1.6/sqrt(3.92)},{3*0.6/sqrt(3.92)},{3*1/sqrt(3.92)});
					\coordinate (Q42) at ({3*1.6/sqrt(3.72)},{3*0.4/sqrt(3.72)},{3*1/sqrt(3.72)});
					
					\coordinate (Q13) at ({1.4/sqrt(3.12)},{0.4/sqrt(3.12)},{1/sqrt(3.12)});
					\coordinate (Q23) at ({1.4/sqrt(3.32)},{0.6/sqrt(3.32)},{1/sqrt(3.32)});
					\coordinate (Q33) at ({1.6/sqrt(3.92)},{0.6/sqrt(3.92)},{1/sqrt(3.92)});
					\coordinate (Q43) at ({1.6/sqrt(3.72)},{0.4/sqrt(3.72)},{1/sqrt(3.72)});
					
					\coordinate (P1) at ({1/sqrt(2.64)},{0.8/sqrt(2.64)},{1/sqrt(2.64)});
					\coordinate (P2) at ({1/sqrt(3)},{1/sqrt(3)},{1/sqrt(3)});
					\coordinate (P3) at ({1.2/sqrt(3.44)},{1/sqrt(3.44)},{1/sqrt(3.44)});
					\coordinate (P4) at ({1.2/sqrt(3.08)},{0.8/sqrt(3.08)},{1/sqrt(3.08)});
					
					\fill[red,fill opacity=0.7, draw opacity=1] 
					(S13) 
					to[out=-59,in=130] (S23)
					to[out=-161,in=28] (S33)
					to[out=137,in=-62] (S43)
					to[out=63,in=-135] cycle;
					
					\fill[blue,fill opacity=0.7, draw opacity=1] 
					(Q13) 
					to[out=-50,in=130] (Q23)
					to[out=-145,in=35] (Q33)
					to[out=130,in=-53] (Q43)
					to[out=50,in=-130] cycle;
					
					\draw[draw=darkgreen,line width=0.18pt,draw opacity=0.5,dash pattern=on 1pt off 1pt] 
					(1,0.8,1) -- (P1);
					\draw[draw=darkgreen,line width=0.18pt,draw opacity=0.5,dash pattern=on 1pt off 1pt] 
					(1,1,1) -- (P2);
					\draw[draw=darkgreen,line width=0.18pt,draw opacity=0.5,dash pattern=on 1pt off 1pt] 
					(1.2,1,1) -- (P3);
					\draw[draw=darkgreen,line width=0.18pt,draw opacity=0.5,dash pattern=on 1pt off 1pt] 
					(1.2,0.8,1) -- (P4);
					
					\draw[draw=red,line width=0.18pt,draw opacity=0.5,dash pattern=on 1pt off 1pt] 
					(1,0,1) -- (S13);
					\draw[draw=red,line width=0.18pt,draw opacity=0.5,dash pattern=on 1pt off 1pt] 
					(1,1,1) -- (S23);
					\draw[draw=red,line width=0.18pt,draw opacity=0.5,dash pattern=on 1pt off 1pt] 
					(2,1,1) -- (S33);
					\draw[draw=red,line width=0.18pt,draw opacity=0.5,dash pattern=on 1pt off 1pt] 
					(2,0,1) -- (S43);
					
					\fill[darkgreen,fill opacity=0.7, draw opacity=1] 
					(P1) 
					to[out=-51,in=132] (P2)
					to[out=-161,in=24] (P3)
					to[out=140,in=-46] (P4)
					to[out=30,in=-151] cycle;
					
					\draw[draw=blue,line width=0.18pt,draw opacity=0.5,dash pattern=on 1pt off 1pt] 
					(1.4,0.4,0) -- (1.4,0.4,1) ;
					\draw[draw=blue,line width=0.18pt,draw opacity=0.5,dash pattern=on 1pt off 1pt] 
					(1.4,0.6,0) -- (1.4,0.6,1) ;
					\draw[draw=blue,line width=0.18pt,draw opacity=0.5,dash pattern=on 1pt off 1pt] 
					(1.6,0.6,0) -- (1.6,0.6,1) ;
					\draw[draw=blue,line width=0.18pt,draw opacity=0.5,dash pattern=on 1pt off 1pt] 
					(1.6,0.4,0) -- (1.6,0.4,1) ;
					
					\draw[draw=darkgreen,line width=0.18pt,draw opacity=0.5,dash pattern=on 1pt off 1pt] 
					(1,0.8,0) -- (1,0.8,1) ;
					\draw[draw=darkgreen,line width=0.18pt,draw opacity=0.5,dash pattern=on 1pt off 1pt] 
					(1,1,0) -- (1,1,1) ;
					\draw[draw=darkgreen,line width=0.18pt,draw opacity=0.5,dash pattern=on 1pt off 1pt] 
					(1.2,1,0) -- (1.2,1,1) ;
					\draw[draw=darkgreen,line width=0.18pt,draw opacity=0.5,dash pattern=on 1pt off 1pt] 
					(1.2,0.8,0) -- (1.2,0.8,1) ;
					
					\draw[draw=red,line width=0.18pt,draw opacity=0.5,dash pattern=on 1pt off 1pt] 
					(1,0,0) -- (1,0,1) ;
					\draw[draw=red,line width=0.18pt,draw opacity=0.5,dash pattern=on 1pt off 1pt] 
					(1,1,0) -- (1,1,1) ;
					\draw[draw=red,line width=0.18pt,draw opacity=0.5,dash pattern=on 1pt off 1pt] 
					(2,1,0) -- (2,1,1) ;
					\draw[draw=red,line width=0.18pt,draw opacity=0.5,dash pattern=on 1pt off 1pt] 
					(2,0,0) -- (2,0,1) ;
					
					\draw[draw=red,line width=0.18pt,draw opacity=0.5,dash pattern=on 1pt off 1pt] 
					(1,0,1) -- (S11);
					\draw[draw=red,line width=0.18pt,draw opacity=0.5,dash pattern=on 1pt off 1pt] 
					(1,1,1) -- (S21);
					\draw[draw=red,line width=0.18pt,draw opacity=0.5,dash pattern=on 1pt off 1pt] 
					(2,1,1) -- (S31);
					\draw[draw=red,line width=0.18pt,draw opacity=0.5,dash pattern=on 1pt off 1pt] 
					(2,0,1) -- (S41);
					
					\draw[draw=blue,line width=0.18pt,draw opacity=0.5,dash pattern=on 1pt off 1pt] 
					(1.4,0.4,1) -- (Q11);
					\draw[draw=blue,line width=0.18pt,draw opacity=0.5,dash pattern=on 1pt off 1pt] 
					(1.4,0.6,1) -- (Q21);
					\draw[draw=blue,line width=0.18pt,draw opacity=0.5,dash pattern=on 1pt off 1pt] 
					(1.6,0.6,1) -- (Q31);
					\draw[draw=blue,line width=0.18pt,draw opacity=0.5,dash pattern=on 1pt off 1pt] 
					(1.6,0.4,1) -- (Q41);
					
					\draw[draw=blue,line width=0.18pt,draw opacity=0.5,dash pattern=on 1pt off 1pt] 
					(1.4,0.4,1) -- (Q13);
					\draw[draw=blue,line width=0.18pt,draw opacity=0.5,dash pattern=on 1pt off 1pt] 
					(1.4,0.6,1) -- (Q23);
					\draw[draw=blue,line width=0.18pt,draw opacity=0.5,dash pattern=on 1pt off 1pt] 
					(1.6,0.6,1) -- (Q33);
					\draw[draw=blue,line width=0.18pt,draw opacity=0.5,dash pattern=on 1pt off 1pt] 
					(1.6,0.4,1) -- (Q43);
					
					\draw[draw=pink,line width=0.18pt,draw opacity=0.5,dash pattern=on 1pt off 1pt] 
					(S12) -- (S11) ;
					\draw[draw=pink,line width=0.18pt,draw opacity=0.5,dash pattern=on 1pt off 1pt] 
					(S22) -- (S21) ;
					\draw[draw=pink,line width=0.18pt,draw opacity=0.5,dash pattern=on 1pt off 1pt] 
					(S32) -- (S31) ;
					\draw[draw=pink,line width=0.18pt,draw opacity=0.5,dash pattern=on 1pt off 1pt] 
					(S42) -- (S41) ;
					
					\draw[draw=cyan,line width=0.18pt,draw opacity=0.5,dash pattern=on 1pt off 1pt] 
					(Q12) -- (Q11) ;
					\draw[draw=cyan,line width=0.18pt,draw opacity=0.5,dash pattern=on 1pt off 1pt] 
					(Q22) -- (Q21) ;
					\draw[draw=cyan,line width=0.18pt,draw opacity=0.5,dash pattern=on 1pt off 1pt] 
					(Q32) -- (Q31) ;
					\draw[draw=cyan,line width=0.18pt,draw opacity=0.5,dash pattern=on 1pt off 1pt] 
					(Q42) -- (Q41) ;
					
					\draw[pink,fill opacity=0.7, fill=pink, draw opacity=1] 
					(S11) to (S12) to[out=-60,in=130] (S22) to (S21) to[out=132,in=-60] cycle;
					
					\draw[pink,fill opacity=0.7, fill=pink, draw opacity=1] 
					(S21) to (S22) to[out=-160,in=28] (S32) to (S31) to[out=27,in=-162] cycle;
					
					\draw[pink,fill opacity=0.7, fill=pink, draw opacity=1]
					(S12) 
					to[out=-60,in=130] (S22)
					to[out=-160,in=28] (S32)
					to[out=140,in=-57] (S42)
					to[out=60,in=-133] cycle;
					
					\draw[pink,dashed,dash pattern=on 1pt off 1pt] (S31) to[out=140,in=-60] (S41) ;
					\draw[pink,dashed,dash pattern=on 1pt off 1pt] (S11) to[out=-133,in=60] (S41) ;
					
					\draw[cyan,fill opacity=0.7, fill=cyan, draw opacity=1] 
					(Q11) to (Q12) to[out=-48,in=135] (Q22) to (Q21)-- cycle;
					
					\draw[cyan,fill opacity=0.7, fill=cyan, draw opacity=1] 
					(Q21) to (Q22) to[out=-145,in=40] (Q32) to (Q31) -- cycle;
					
					\draw[cyan,fill opacity=0.7, fill=cyan, draw opacity=1]
					(Q12) 
					to[out=-48,in=135] (Q22)
					to[out=-145,in=40] (Q32)
					to[out=135,in=-50] (Q42)
					to[out=46,in=-136] cycle;
					
					\draw[cyan,dashed,dash pattern=on 1pt off 1pt] (Q31) to (Q41) ;
					\draw[cyan,dashed,dash pattern=on 1pt off 1pt] (Q11) to (Q41) ;
					
					\draw[thick,->,black] (2.95,0,0) -- (3.5,0,0) node[anchor=north]{$\xi_1$};
					
					\draw[thick,->,black] (0,2.95,0) -- (0,3.5,0) node[anchor=north]{$\xi_2$};
					
					\draw[thick,->,black] (0,0,2.9) -- (0,0,3.5) node[anchor=west]{$\xi_n$};
					
					\begin{scope}
						\clip (0,0,0) circle (3cm); 
						
						\draw[orange] 
						plot[domain={asin(-0.2)}:{acos(-0.2)},smooth,variable=\t]
						({3*cos(\t)},{3*sin(\t)}, 0); 
					\end{scope}
					
					\node[red] at (2,1.2,-0.15) {$Q_{t}$}; 
					\node[blue] at (1.5,0.85,-0.15) {$Q_{t}^{\circ}$}; 
					\node[cyan] at (2.65,0.45,1.25) {$\omega_t^{\circ}$}; 
					\node[pink] at (3,0.45,1.85) {$\omega_t$}; 
					\node[darkgreen] at (1,1,-0.25) {$Q_{t,\tau}$}; 
					\node[darkgreen] at (0.475,0.45,0.35) {$\alpha_{t,\tau}$}; 
					\node[red] at (0.325,0.45,0.9) {$\ecc(t)$}; 
					\node[blue] at (0,0.45,0.58) {$\ecc(t)^{\circ}$}; 
					\node[orange] at (-1.05,0.9,0.2) {$\abs{\xi}=1$}; 
					\node[orange] at (-1.7,0.9,1) {$\abs{\xi}=\ann(t)/2$}; 
					\node[orange] at (-1.7,0.9,1.9) {$\abs{\xi}=6 \ann(t)$};  
					
				\end{tikzpicture}
				}
		\caption{The frequency component, annulus, eccentricity and directional support of a tile.}
		\label{Fig:TileFrequency}
	\end{subfigure}
	\begin{subfigure}[b]{0.41\textwidth}
		\tdplotsetmaincoords{85}{170} 
		\resizebox{\textwidth}{!}{%
			\begin{tikzpicture}[scale=1.8,tdplot_main_coords] 
				
				\clip[tdplot_screen_coords] (-2.2, -1.2) rectangle (1.4, 1.9);
				
				\draw[dashed,black,thick] (0,0,0) -- (0,0,-0.5);
				
				\draw[draw=cyan] 
				(2,2,{-2*0.15/sqrt(0.8875)+2*0.3/sqrt(0.8875)+1}) -- (2,2,{-2*0.15/sqrt(0.8875)+2*0.3/sqrt(0.8875)+0.75});
				\draw[draw=cyan]
				(2,-1,{1*0.15/sqrt(0.8875)+2*0.3/sqrt(0.8875)+1}) -- (2,-1,{1*0.15/sqrt(0.8875)+2*0.3/sqrt(0.8875)+0.75});
				\draw[draw=cyan]
				(-1,-1,{1*0.15/sqrt(0.8875)-1*0.3/sqrt(0.8875)+1}) -- (-1,-1,{1*0.15/sqrt(0.8875)-1*0.3/sqrt(0.8875)+0.75}); 
				\draw[draw=cyan]
				(-1,2,{-2*0.15/sqrt(0.8875)-1*0.3/sqrt(0.8875)+1}) -- (-1,2,{-2*0.15/sqrt(0.8875)-1*0.3/sqrt(0.8875)+0.75});
				
				\draw[dashed,red,dash pattern=on 1pt off 1pt] (0.5,0.5,-0.5) -- (0.5,0.5,{-0.5*0.15/sqrt(0.8875)+0.5*0.3/sqrt(0.8875)}) ;
				\draw[dashed,red,dash pattern=on 1pt off 1pt] (1.5,0.5,-0.5) -- (1.5,0.5,{-0.5*0.15/sqrt(0.8875)+1.5*0.3/sqrt(0.8875)}) ;
				\draw[dashed,red,dash pattern=on 1pt off 1pt] (1.5,1.5,-0.5) -- (1.5,1.5,{-1.5*0.15/sqrt(0.8875)+1.5*0.3/sqrt(0.8875)}) ;
				\draw[dashed,red,dash pattern=on 1pt off 1pt] (0.5,1.5,-0.5) -- (0.5,1.5,{-1.5*0.15/sqrt(0.8875)+0.5*0.3/sqrt(0.8875)}) ;
				
				\filldraw[draw=cyan,fill=cyan!30,fill opacity=0.6,draw opacity=0] 
				(1,2,0) -- (-1/2,-1,0) --
				(-1,-1,{1*0.15/sqrt(0.8875)-1*0.3/sqrt(0.8875)}) -- (-1,2,{-2*0.15/sqrt(0.8875)-1*0.3/sqrt(0.8875)}) -- cycle;
				
				\draw[draw=cyan,opacity=0.6] 
				(-1/2,-1,0) --
				(-1,-1,{1*0.15/sqrt(0.8875)-1*0.3/sqrt(0.8875)}) -- (-1,2,{-2*0.15/sqrt(0.8875)-1*0.3/sqrt(0.8875)}) -- (1,2,0) ;
				
				\filldraw[draw=gray,fill=gray!30,opacity=0.6] 
				(-1,-1,0) -- (-1,2,0) -- (2,2,0) -- (2,-1,0) -- cycle;
				
				\draw[black,->,thick] (0,0,0) -- (2,0,0) node[anchor=north]{$x_1$};
				\draw[dashed,black,thick] (0,0,0) -- (-1,0,0);
				
				\draw[dashed,black,thick] (0,0,0) -- (0,-1,0);
				
				\filldraw[draw=red,fill=red,fill opacity=0.6,draw opacity=1] 
				(0.5,0.5,0) -- (1.5,0.5,0) -- (1.5,1.5,0) -- (0.5,1.5,0) -- cycle;
				
				\filldraw[draw=cyan,draw opacity=0,fill=cyan!30, fill opacity=0.6] 
				(2,-1,{1*0.15/sqrt(0.8875)+2*0.3/sqrt(0.8875)}) --
				(2,2,{-2*0.15/sqrt(0.8875)+2*0.3/sqrt(0.8875)}) --
				(1,2,0) -- (-1/2,-1,0) -- cycle;
				
				\draw[draw=cyan,draw opacity=0.6] 
				(-1/2,-1,0) -- (2,-1,{1*0.15/sqrt(0.8875)+2*0.3/sqrt(0.8875)}) -- (2,2,{-2*0.15/sqrt(0.8875)+2*0.3/sqrt(0.8875)}) -- (1,2,0) ;
				\draw[draw=cyan,dashed,dash pattern=on 1pt off 1pt,draw opacity=0.6] 
				(1,2,0) -- (-1/2,-1,0) ;
				
				\draw[black,->,thick] (0,0,0) -- (0,2,0) node[anchor=north]{$x_2$};
				
				\draw[black,->,thick] (0,0,0) -- (0,0,1.7) node[anchor=west]{$x_n$};
				
				\draw[dashed,red,dash pattern=on 1pt off 1pt] (0.5,0.5,{-0.5*0.15/sqrt(0.8875)+0.5*0.3/sqrt(0.8875)}) -- (0.5,0.5,{-0.5*0.15/sqrt(0.8875)+0.5*0.3/sqrt(0.8875)+0.75});
				\draw[dashed,red,dash pattern=on 1pt off 1pt]  (1.5,0.5,{-0.5*0.15/sqrt(0.8875)+1.5*0.3/sqrt(0.8875)}) -- (1.5,0.5,{-0.5*0.15/sqrt(0.8875)+1.5*0.3/sqrt(0.8875)+0.75});
				\draw[dashed,red,dash pattern=on 1pt off 1pt]  (1.5,1.5,{-1.5*0.15/sqrt(0.8875)+1.5*0.3/sqrt(0.8875)}) -- (1.5,1.5,{-1.5*0.15/sqrt(0.8875)+1.5*0.3/sqrt(0.8875)+0.75}) ;
				\draw[dashed,red,dash pattern=on 1pt off 1pt]  (0.5,1.5,{-1.5*0.15/sqrt(0.8875)+0.5*0.3/sqrt(0.8875)}) -- (0.5,1.5,{-1.5*0.15/sqrt(0.8875)+0.5*0.3/sqrt(0.8875)+0.75});
				
				\filldraw[draw=cyan, fill=cyan!30,fill opacity=0.6, draw opacity=1] 
				(2,2,{-2*0.15/sqrt(0.8875)+2*0.3/sqrt(0.8875)+0.75}) --
				(2,-1,{1*0.15/sqrt(0.8875)+2*0.3/sqrt(0.8875)+0.75}) --
				(-1,-1,{1*0.15/sqrt(0.8875)-1*0.3/sqrt(0.8875)+0.75}) -- (-1,2,{-2*0.15/sqrt(0.8875)-1*0.3/sqrt(0.8875)+0.75}) -- cycle;
				
				\draw[dashed,red]
				(0.5,0.5,{-0.5*0.15/sqrt(0.8875)+0.5*0.3/sqrt(0.8875)+0.75}) -- (0.5,0.5,{-0.5*0.15/sqrt(0.8875)+0.5*0.3/sqrt(0.8875)+1});
				\draw[dashed,red,dash pattern=on 1pt off 1pt]  (1.5,0.5,{-0.5*0.15/sqrt(0.8875)+1.5*0.3/sqrt(0.8875)+0.75}) -- (1.5,0.5,{-0.5*0.15/sqrt(0.8875)+1.5*0.3/sqrt(0.8875)+1});
				\draw[dashed,red,dash pattern=on 1pt off 1pt]  (1.5,1.5,{-1.5*0.15/sqrt(0.8875)+1.5*0.3/sqrt(0.8875)+0.75}) -- (1.5,1.5,{-1.5*0.15/sqrt(0.8875)+1.5*0.3/sqrt(0.8875)+1}) ;
				\draw[dashed,red,dash pattern=on 1pt off 1pt]  (0.5,1.5,{-1.5*0.15/sqrt(0.8875)+0.5*0.3/sqrt(0.8875)+0.75}) -- (0.5,1.5,{-1.5*0.15/sqrt(0.8875)+0.5*0.3/sqrt(0.8875)+1});
				
				\coordinate (P11) at (0.5,0.5,{-0.5*0.15/sqrt(0.8875)+0.5*0.3/sqrt(0.8875)+0.75});
				\coordinate (P21) at (1.5,0.5,{-0.5*0.15/sqrt(0.8875)+1.5*0.3/sqrt(0.8875)+0.75});
				\coordinate (P31) at (1.5,1.5,{-1.5*0.15/sqrt(0.8875)+1.5*0.3/sqrt(0.8875)+0.75});
				\coordinate (P41) at (0.5,1.5,{-1.5*0.15/sqrt(0.8875)+0.5*0.3/sqrt(0.8875)+0.75});
				
				\coordinate (P12) at (0.5,0.5,{-0.5*0.15/sqrt(0.8875)+0.5*0.3/sqrt(0.8875)+1});
				\coordinate (P22) at (1.5,0.5,{-0.5*0.15/sqrt(0.8875)+1.5*0.3/sqrt(0.8875)+1});
				\coordinate (P32) at (1.5,1.5,{-1.5*0.15/sqrt(0.8875)+1.5*0.3/sqrt(0.8875)+1});
				\coordinate (P42) at (0.5,1.5,{-1.5*0.15/sqrt(0.8875)+0.5*0.3/sqrt(0.8875)+1});
				
				\draw[draw=red, dashed] 
				(P11) -- (P21) -- (P31) -- (P41) -- cycle;
				
				\draw[draw=red,fill=red,fill opacity=0.6,draw opacity=1] 
				(P32) -- (P22) -- (P21) -- (P31) -- cycle;
				\draw[draw=red,fill=red,fill opacity=0.6,draw opacity=1] 
				(P32) -- (P42) -- (P41) -- (P31) -- cycle;
				
				\filldraw[draw=cyan, fill=cyan!30,fill opacity=0.6, draw opacity=1] 
				(2,2,{-2*0.15/sqrt(0.8875)+2*0.3/sqrt(0.8875)+1}) --
				(2,-1,{1*0.15/sqrt(0.8875)+2*0.3/sqrt(0.8875)+1}) --
				(-1,-1,{1*0.15/sqrt(0.8875)-1*0.3/sqrt(0.8875)+1}) -- (-1,2,{-2*0.15/sqrt(0.8875)-1*0.3/sqrt(0.8875)+1}) -- cycle;
				
				\draw[dashed,red,dash pattern=on 1pt off 1pt] (0.5,0.5,{-0.5*0.15/sqrt(0.8875)+0.5*0.3/sqrt(0.8875)+1}) -- (0.5,0.5,1.7) ;
				\draw[dashed,red,dash pattern=on 1pt off 1pt] (1.5,0.5,{-0.5*0.15/sqrt(0.8875)+1.5*0.3/sqrt(0.8875)+1}) -- (1.5,0.5,1.7) ;
				\draw[dashed,red,dash pattern=on 1pt off 1pt] (1.5,1.5,{-1.5*0.15/sqrt(0.8875)+1.5*0.3/sqrt(0.8875)+1}) -- (1.5,1.5,1.7) ;
				\draw[dashed,red,dash pattern=on 1pt off 1pt] (0.5,1.5,{-1.5*0.15/sqrt(0.8875)+0.5*0.3/sqrt(0.8875)+1}) -- (0.5,1.5,1.7) ;
				
				\draw[draw=red,fill=red,fill opacity=0.6,draw opacity=1] 
				(P12) -- (P22) -- (P32) -- (P42) -- cycle;
				
				\draw[very thick,->,cyan] (0,0,0) -- (-0.3,0.15,{sqrt(0.8875)}) node[anchor=north west]{$v_{t}$};
				
				\draw[very thick,|-|,blue] (0,0,1) -- (0,0,0.75) node[anchor=south west]{$I_{t}$};
				
				\node[cyan] at (-1.1,1,-0.38) {$\sigma_{v_t}$};
				\node[cyan] at (-0.5,1,1.25) {$\sigma_{v_t}+I_{t}$};
				\node[red] at (1,1,1.4) {$R_t$};
				\node[red] at (1.2,0.8,-0.2) {$L_t$};
				
			\end{tikzpicture}
		}
		\caption{A parallelepiped in $ \R^n $.}
		\label{Fig:TileSpace}
	\end{subfigure}
	\caption{Example of the frequency and space component of a tile.}
	\label{Fig:Tile}
	\end{figure}
	
	In the next lemmas, we state some geometric properties of the tiles that we will use. We begin with the case in which the tiles have the same annulus. For a number $ A>0 $ and a parallelepiped $ R = R(L,I,v) $ we denote by $ A R = R(A L,A I,v) $ the scaled parallelepiped with the same center as $ R $ with the scaling factor $ A $.
	
	\begin{lemma}[Geometric Lemma] 
		\label{Lem:Geom} 
		There exists a constant $K_n$ depending only upon dimension, such that the following holds: For every number $ A \geq 1 $ and two tiles $ t, t' \in \mathcal{T}_{\ann} $ for a fixed $ \ann \in 3^{\Z} $ and such that $ \ecc(t') \subseteq \ecc(t) $ and $ A R_t \cap A R_{t'} \neq \emptyset $, there holds $ A R_t \subseteq K_n A R_{t'} $.
	\end{lemma}
	
	We omit the proof of the lemma above, and just note that it is enough to prove it in two dimensions. Indeed, assuming $n\ge 3$, two elements of $\mathrm{Gr}(d,n)$ intersect on an element $\sigma_0\in\mathrm{Gr}(d-1,n)$, and the lemma can be proved in the orthogonal subspace $\sigma_0 ^\perp\in \mathrm{Gr}(2,n)$ reducing the proof to a two-dimensional argument.
 
	Next, we state a geometric property for the case in which the tiles have different annulus but still have similar directions relative to $ \ell(Q_t) $ since their eccentricities are nested.
	
	\begin{lemma}
		\label{Lem:Geom2} 
		There exists a constant $K_n$ depending only upon dimension, such that the following holds: For every two tiles $ t, t' \in \mathcal{T} $ with $ \ann(t) \leq \ann(t') $ and $ \ecc(t') \subseteq \ecc(t) $, the following hold.
		\begin{enumerate}[label=(\roman*)]
			\item[(i)] \label{Item:GeomV1} If $ L_t = L_{t'} $ and $ R_t \cap R_{t'} \neq \emptyset $, then $ R_{t'} \subseteq  R(L_t,K_n I_t,v_t)$.
			\item[(ii)] \label{Item:GeomV2} If $ L_t = L_{t'} $, $ K_n \abs{ I_{t'} } \leq \abs{ I_t } $ and $ R_t \cap K_n R_{t'} \neq \emptyset $ then $ R(L_{t'},K_n I_{t'},v_{t'}) \subseteq R(L_t,K_n I_t,v_t) $.
		\end{enumerate}
	\end{lemma}
	
	The two-dimensional version of this Lemma \ref{Lem:Geom2} can be found in \cite[Lemma 6]{BatemanThiele} and it is straightforward to generalise it to higher dimensions. 
	
	We will always denote by $ K_n $ the constant given by these geometric Lemmas \ref{Lem:Geom} and \ref{Lem:Geom2} and making it bigger if necessary, we will assume that it is of the form $ 3^{k_n} $ for some $ k_n \in \N $. Finally, we need the following properties in which we do not have that the eccentricities are nested.
	
	\begin{lemma} 
		\label{Lem:Geom3} 
		Let $ t,t' \in \mathcal{T} $ such that $ L_{t} = L_{t'} $ and $ C>0 $ a constant. Then,
		\begin{enumerate}[label=(\roman*)]
			\item \label{Item:Geom2C} If $ \abs{ c(Q_{t}) - c(Q_{t'}) } \leq C $ and $ R_t \cap R_{t'} \neq \emptyset$, then $ \displaystyle{ R_t \subseteq R\Big( L_t , 2 \Big(1 + \frac{ \abs{I_{t}} }{ \abs{I_{t'}} } + \frac{ \sqrt{d} \ell(L) C }{ \abs{I_{t'}} } \Big) I_{t'}, v_{t'} \Big) } $.
			\item \label{Item:Geom2F} If $ \abs{ c(Q_{t}) - c(Q_{t'}) } \geq C $, then $ \displaystyle{ \Abs{ R_t \cap R_{t'} } \leq \frac{ \abs{ I_t } \abs{ I_{t'} } \ell(L)^{d-1} }{C} } $.
		\end{enumerate}
	\end{lemma}
	
	The two-dimensional proof of Lemma \ref{Lem:Geom3} can be found in \cite[Lemma 13]{BatemanThiele} and again it is easy to generalise to higher dimensions. Lemmas \ref{Lem:Geom2} and \ref{Lem:Geom3} are only needed in the proofs of Lemmas \ref{Lem:MainConstructionG} and \ref{Lem:MainConstructionH}.
	
	\subsection{Trees} 
	\label{Ss:Defs}
	
	In this subsection we will define \emph{trees} which will be used in Section \ref{S:Decom} to decompose an arbitrary collection of tiles. To do so, we fix for the rest of the subsection $ \ann \in 3^{\Z} $ and a finite collection of tiles $ \P \subset \mathcal{T}_{\ann} $, where recall that $ \mathcal{T}_{\ann} $ is the collection of tiles $ t \in \mathcal{T} $ satisfying $ \ann(t) = \ann $. 
	
	\begin{definition}[$ 3^{\rho} $-trees] \label{Def:Tree} 
		Given a nonnegative integer $ \rho $ we define a \emph{$ 3^{\rho} $-tree} as a collection of tiles $ \mathbf{T} \subseteq \P $ such that there exists a \emph{top} 
		$ ( \xi_{\mathbf{T}}, R_{\mathbf{T}} ) $ with $ \xi_{\mathbf{T}} \in \mathbb{S}^d $ and $ R_{\mathbf{T}} \in \{ R_t: \, t \in \mathcal{T}_{\ann} \} $  such that
		\begin{enumerate}[label=(T\arabic*)]
			\item $ \xi_{\mathbf{T}} \in \ecc(t) $ for all $ t \in \mathbf{T} $;
			\item $ \scl(R_t) \leq \scl(R_{\mathbf{T}}) $ and $ R_t \cap 3^{\rho} R_{\mathbf{T}} \neq \emptyset $ for every $ t \in \mathbf{T} $.
		\end{enumerate}
		A $3^\rho$-\emph{tree} will be called \emph{lacunary} if $ \xi_{\mathbf{T}} \in \ecc(t) \setminus \ecc(t)^{\circ} $ for all $ t \in \mathbf{T} $, and it will be called \emph{overlapping} if $ \xi_{\mathbf{T}} \in \ecc(t)^{\circ} $.
	\end{definition}
	
	For $ \rho = 0 $ a $ 1 $-tree will be just called a tree and in this case we will omit the $ 1 $ from the notation. Note that the geometric Lemma \ref{Lem:Geom} implies that if $ \mathbf{T} $ is a $ 3^{\rho } $ tree, then
	\begin{equation}
		\label{Eq:TreeShadow}
		\bigcup_{ t \in \mathbf{T} } R_{t} \subseteq K_n 3^{\rho} R_{ \mathbf{T} }.
	\end{equation}
	Given $ t,t' \in \mathcal{T}_{\ann} $ and $ k \in \N $ we define
	\begin{equation}
		t\leq_{k} t' \qquad \overset{\mathrm{def}}{ \Longleftrightarrow} 
		\qquad 
		\ecc(t') \subseteq \ecc(t) \quad \text{and} \quad R_t \cap K_n^{k-1} R_{t'}\neq \emptyset.
	\end{equation}
	If $ k=1 $, we will just write $ t \leq t' $. We note that $ \leq $ is \emph{not} a partial order as it fails to satisfy the transitivity property. However, we can deal with the lack of transitivity by increasing the value of $ k $. More precisely, if $ t \leq t' \leq t'' $, then $ t \leq_{2} t'' $. This is an immediate consequence of the geometric Lemma \ref{Lem:Geom}, which also implies that if $ t \leq_k t' $, then $ R_t \subseteq K_n^k R_{t'} $. 
	
	Now, given a tree $\mathbf{T}$, there exists a \emph{top tile} $ \topt(\mathbf{T})=R_{\mathbf{T}}\times \omega_{\mathbf{T}} \in \mathcal{T}_{\ann} $ such that $ \xi_{\mathbf{T}} \in \ecc(\mathbf{T}) \coloneqq \ecc(\topt(\mathbf{T})) $ and $ t \leq \topt(\mathbf{T}) $ for all $ t \in \mathbf{T} $. This fact allows us to assume that the point $ \eta_{\mathbf{T}} \in Q_{\mathbf{T}} $ that satisfies $ (\eta_{\mathbf{T}} \times \{1\})' = \xi_{\mathbf{T}} $ has no triadic coordinates.
	
	\begin{remark}\label{rmrk:pairincorp}
		Some further remarks concerning $ \leq $ are in order. Firstly we note that we only define $ \leq $ for tiles $ t \in \mathcal{T}_{\ann} $ for some fixed $ \ann \in 3^{\Z} $. With this in mind we note that if $ t,t' \in \mathcal{T}_{\ann} $ are \emph{not comparable} under $ \leq $ and  $ \omega_t \cap \omega_{t'} \neq \emptyset $ then necessarily $ R_t \cap R_{t'} = \emptyset $. This basic principle will be used throughout the paper. This fails rather dramatically if $ \ann(t) \neq \ann(t') $ because the frequency (and space) components of each tile are two-parameter objects and, for example, we could have $ \omega_t \cap \omega_{t'} \neq \emptyset $ without $ \omega_t \subseteq \omega_{t'} $ or $ \omega_{t'} \subseteq \omega_{t} $.
	\end{remark}
	
	Next we state a result that will allow us to decompose a $ 3^{\rho} $-tree into $1$-trees.
	
	\begin{lemma}
		\label{Lem:Split} 
		If $\,\mathbf{T} $ is a $ 3^{\rho} $-tree with top $(\xi_\mathbf{T},R_{\mathbf T})$ there exist $O(3^{\rho n})$ disjoint trees $\mathbf{T}_{\nu}$ such that $ \mathbf{T} = \bigcup_{ \nu } \mathbf{T}_{\nu}$, and each $\mathbf T_\nu$ has top $ (\xi_{\mathbf{T}}, R_{\mathbf{T_\nu}}) $ with $ R_{\mathbf{T_\nu}} $ a congruent copy of $ R_{\mathbf{T}} $.
	\end{lemma}
	
	We omit the easy proof. Finally, we will use the following definition to encode orthogonality properties of lacunary trees.
	
	\begin{definition}[strongly disjoint families] 
		\label{Def:StrongDisjFam} 
		Let $\mathfrak{T}$ be a family of lacunary trees with $ \mathbf{T} \subset \mathbb{P} \subseteq \mathcal{T}_{\ann} $ for every $ \mathbf{T} \in \mathfrak{T} $. The family $ \mathfrak{T} $ is called strongly disjoint if for $ t \in \mathbf{T} \in \mathfrak{T} $ and $ t' \in \mathbf{T}' \in \mathfrak{T} $ with $ \mathbf{T} \neq \mathbf{T}' $ and $ \ecc(t) \subseteq \ecc(t')^{\circ} $ then $ R_{t'} \cap  K_n^2 R_{\mathbf{T}}= \emptyset $.
	\end{definition}

	\subsection{Adapted classes}
	\label{Ss:Adapted} 
	
	In this subsection, we define the class of functions adapted to each tile $ t \in \mathcal{T} $. Recall the big number $ M > 0 $ we have fixed in the beginning of the section and let $ \Theta_M $ be the unit ball of the Banach space of functions
	\begin{equation}
		u \in \mathcal C^{M}(\R^n), \qquad \norm{u}{ \star,M }\coloneqq \sup_{0\leq \abs{\beta} \leq M} \Norm{ \left\langle x \right\rangle^{M} \partial_x^\beta u(x) }{ L^{\infty}(\R^n) }<\infty,
	\end{equation}
	where $ \langle x \rangle \coloneqq (1+\abs{x}^2)^{\frac{1}{2}} $. As commented before, $ M $ will be chosen sufficiently large depending on the integrability exponent $ p $ and the dimension. We will often use the special bump function $ \chi^M \in \mathcal C^\infty $ given by $ \chi^M(x) \coloneqq \langle x \rangle^{-M} $. It is convenient to associate to each parallelepiped $ R = R(L_R,I_R,v_R) $ the translation and dilation operator
	\begin{equation}
		\Dil_{R}^{(p)} f (x) \coloneqq \frac{1}{\abs{R}^{1/p}} f\left( \frac{ \Pi_{ \sigma_{ v_{R} } }(x-c(R)) }{ \ell(L_{R}) } +
		\frac{ \Pi_{ e_n }(x-c(R))}{ \ell(I_{R}) } \right), \qquad x\in \R^n. 
	\end{equation}
	Furthermore, we define the $L^p$-normalized bump function adapted to the parallelepiped $R$
	\begin{equation}
		\label{Eq:DefBumpParallelepiped}
		\chi_{R,p}^{M} (x) \coloneqq \Dil_{R}^{(p)} \chi^M(x).
	\end{equation}
	Finally, for a tile $t \in \mathcal{T}$, define the modulation operator $  \Mod_t g(x) \coloneqq e^{ 2 \pi i \langle x, v_t \ann(t) \rangle} g(x) $, for every $ x \in \R^n $. Now, we are ready to define the adapted classes of wave packets.
	
	Given a tile $ t = R_t \times \omega_t \in \mathcal{T} $ and a large positive integer $ M \gg 1 $, we define $ \mathcal{F}_t^M $ as the collection of functions $ \varphi = \Mod_{t} \Dil_{R_t}^{(2)} \Phi $ for some $ \Phi \in \Theta_M $ such that $ \supp \widehat {\varphi} \subseteq \omega_t^{\circ} $. Furthermore, define $\mathcal{A}_t^{M}$ as the collection of functions $ \vartheta = \vartheta (\cdot,\sigma) $ on $\R^{n} \times \Sigma_{\gamma} $ satisfying $ \vartheta(\cdot,\sigma) \in \mathcal{F}_t^M $ for all fixed $ \sigma \in \Sigma_{\gamma} $ and $ \vartheta(\cdot,\sigma) = 0 $ for all $ \sigma \in \Sigma_{\gamma} $ such that $ v_\sigma \notin \alpha_t$ and satisfying the adaptedness condition
	\begin{equation}
		\Abs{ \vartheta ( \cdot,\sigma) - \vartheta (\cdot,\sigma') } \leq  \max\left\{ \scl(t) \ann(t) \dist(\sigma,\sigma') , \, \frac{1}{\log( e+ \left[\dist(\sigma,\sigma') \right]^{-1})} \right\} \chi_{R_t,2}^M
	\end{equation}
	for all $\sigma,\sigma' \in \Gr(d,n)$ with $ \scl(t) \ann(t) \dist(\sigma ,\sigma') \lesssim  1 $.
	
	On the other hand, define the tile coefficient maps 
	\begin{equation}
		\label{Eq:CoeffDef1}
		F_M[f](t) \coloneqq \sup_{\varphi \in \mathcal {F}_{t}^M } \Abs{ \langle f, \varphi\rangle }
	\end{equation}
	and, for $ \apl{ \sigma }{ \R^n }{ \Sigma_{ \gamma } } $ and $ 1 \leq \tau \leq 3^{\kappa d}-1 $,
	\begin{equation}
		\label{Eq:CoeffDef2}
		A_{\sigma,\tau,M}[g] (t) \coloneqq \sup_{\vartheta \in \mathcal {A}_{t}^M } \Abs{ \left\langle g, \vartheta\left(\cdot, \sigma(\cdot) \right) \indf{\alpha_{t,\tau}}\left(v_{\sigma(\cdot)}\right)\right\rangle }.
	\end{equation}
	
	\subsection{Discretization of the single annulus operator}
	\label{Ss:Model}
	
	Let $ \P \subset\mathcal{T}_{\ann} $, where recall that $ \mathcal{T}_{\ann} $ is the collection of tiles with $ \ann(t) = \ann \in 3^{\Z} $. Then, define the model bisublinear operator
	\begin{equation}
		\label{Eq:ModelOperDef}
		\Lambda_{\P;\sigma,\tau,M} (f,g) 
		\coloneqq \sum_{t\in \P}F_M[f](t) A_{\sigma,\tau,M}[g] (t).
	\end{equation}
	Here the coefficient maps $ F_M[f] $ and $ A_{\sigma,\tau,M}[g] $ are defined in \eqref{Eq:CoeffDef1} and \eqref{Eq:CoeffDef2} respectively. We point out that $ F_M[f](t) $ represents the localization of $ f $ to the tile $ t $ as the inner product between $ f $ and the fast decaying function $ \varphi_t $ adapted to the tile $ t $; on the other hand, $ A_{\sigma,\tau,M}[g] (t) $ represents the localization of $ g $ to the tile $ t $ at the appropriate scale and direction of the singular operator. This is the analogue to the model operator of \cite[Inequality (3)]{LT}.
	
	We have the following bound for the model operator.
		
	\begin{proposition}
		\label{Prop:Model}
		Let $ \apl{ \sigma }{ \R^n }{ \Sigma_{1-\varepsilon} } $ be a measurable function depending only on the first $ n-1 $ variables. Then, there exists $ M=M(n,p) $ sufficiently large such that
		\begin{equation}
			\label{Eq:DesiredModelEst}
			\Lambda_{\P; \sigma, \tau_{0}, M} (f\indf{F},g\indf{E})
			\lesssim \abs{F}^{\frac{1}{p}} \abs{E}^{1-\frac{1}{p}} , \qquad 1<p<\infty,
		\end{equation}
		uniformly over all $ f,g \in L^\infty_0(\R^n)$ with the normalization $ \norm{f}{\infty} = \norm{g}{\infty} = 1 $, all sets $ F,E \subset \R^n $ of finite measure, a fixed choice of $ \tau_{0} \in \{1,\ldots,3^{\kappa d}-1\} $, and a finite collection of tiles $ \mathbb{P} \subset \mathcal{T}_{\ann} $.
	\end{proposition}

	The authors in \cite{BDPPR} proved that there exists $A=A(M,n)$ such that, if $\mathscr M$ satisfies \eqref{Eq:MihlinFamDef} for this choice of $A$ then the proof of Theorem~\ref{Thm:MainSingleBand} is reduced to proving Proposition \ref{Prop:Model}. We remark that the decay and smoothness parameter $M$ will have to be chosen sufficiently large depending on the dimension and the integrability index $p$; for Theorem~\ref{Thm:MainSingleBand} this means that $M$, and thus also $A$, tend to $\infty$ as $p\to 1$. The same happens for Theorem~\ref{Thm:Main} when $p\to 3/2$ and $p\to \infty$.

	\subsection{Orthogonality estimates for lacunary trees} 
	\label{Ss:Orthogonality}
	
	To finish this section, we recall two standard orthogonality estimates for lacunary trees. The proofs can be found  in \cite[Lemmas 6.6 and 6.9]{BDPPR}. We recall that the definition of the coefficient map $ F_N[f]$ was given in \eqref{Eq:CoeffDef1}.

	\begin{lemma} \label{Lem:OrthoLacTree} 
		Let $\mathfrak{T}$ be a strongly disjoint family of lacunary trees, $\mathbb T\coloneqq \cup_{t\in \mathfrak{T}}t$, and $\mathbf{T}$ a lacunary tree. Then, if $\,2n \leq K\leq  N/4$ there holds
		\begin{enumerate}[label=\normalfont{(\arabic*)}]
			\item \label{Item:LemOrtho2} 
			$ \displaystyle{ \sum_{t\in\mathbf{T}} F_N[f](t)^2 
				\lesssim \int_{\R^n} \abs{f}^2 \chi_{ R_{\mathbf{T}},\infty}^{2K} 
				\lesssim \int_{\R^n} \abs{f}^2 } $.
			
			\item \label{Item:LemOrthoStrong} 
			$ \displaystyle{ \bigg( \sum_{\mathbf{T}\in\mathfrak T} \sum_{t\in\mathbf{T}} F_N[f](t)^2 \bigg)^{1/2}
				\lesssim \norm{ f }{L^2(\R^n)} + \bigg( \sup_{ t\in\T} \frac{F_N[f](t)}{\abs{ R_t } ^{1/2}} \bigg( \sum_{\mathbf{T}\in\mathfrak T} \abs{ R_{\mathbf{T}} }\bigg)^{1/2} \bigg)^{1/3} \norm{f}{L^2(\R^n)}^{2/3} } $.
		\end{enumerate}
	\end{lemma}

	\section{Decompositions of tiles and proof of the estimate for the single annulus model operator} 
	\label{S:Decom}
	
	In this section we prove Proposition~\ref{Prop:Model}, which is the desired estimate \eqref{Eq:DesiredModelEst} for the model operator. As in the classical time-frequency proofs, for example in \cite{LT}, the local $ L^p $-case for $ p > 2 $ is somewhat easier and relies on the standard size and density estimates produced by the size and density lemmas, respectively. In the classical proofs the case $ p \leq 2 $ is typically dealt with by removing exceptional sets in the restricted weak-type estimates. Instead, here we rely on an argument going back to \cites{Bateman2013Revista,LaceyLiMemoirs}, involving a maximal estimate for a Kakeya-type maximal function; this maximal function bound provides improved estimates for collections of trees which simultaneously have big size and density. See the rightmost estimate in \eqref{Eq:DecomEstimatesDS} and corresponding proof in Section \ref{S:MaxLemma}.
	
	For the rest of this section we will take $ d=n-1 $ and fix a finite collection of tiles $ \P \subseteq \mathcal{T}_{\ann} $ where recall that $ \mathcal{T}_{\ann} $ is the collection of tiles with a fixed annulus $ \ann \in 3^{\Z} $. We also fix two sets $ E,F \subset \R^n $ of finite measure, two functions $ f,g \in L^\infty_0(\R^n)$ with $ \norm{f}{\infty} = \norm{g}{\infty} = 1 $ and a $ \tau_{0} \in \{1,\ldots,3^{\kappa d}-1\} $, which remain fixed throughout the rest of the section.

	\subsection{Density and size of a collection of tiles} 
	\label{Sss:DefsDenseSize}
	
	First, we have to define two gauges associated to any collection of tiles. 
	
	\begin{definition}[Size] 
		Let $ \P \subset \mathcal{T}_{\ann} $ be a finite collection of tiles. For $ f \in L_0^{\infty}(\R^n)$, define
		\begin{equation} \label{Eq:SizeDef}
			\size(\mathbb{Q}) 
			\coloneqq \sup_{\substack{ \mathbf{T}\text{ lac. tree} \\ \mathbf{T} \subseteq \mathbb{Q}}} \bigg(\frac{1}{\abs{ R_{\mathbf{T}} }} \sum_{t\in{\mathbf{T}}} F_{10n}[f](t)^2 \bigg)^{1/2}, \qquad \mathbb{Q} \subseteq \mathbb{P} .
		\end{equation}
	\end{definition}
	
	For the next definition, we need the following notation, where the directional support $ \alpha_t $ of a tile $ t $ was defined in \eqref{Eq:DirecSuppDef}
	\begin{equation}
		\label{Eq:DefSigma}
		E_t \coloneqq E \cap v_{ \sigma }^{-1}(\alpha_{t,\tau_0}) \coloneqq E \cap \left\{ x \in \R^n : \, v_{\sigma(x)} \in \alpha_{t,\tau_0} \right\}.
	\end{equation}
	
	\begin{definition}[Density] Let $ \P \subset \mathcal{T}_{\ann} $ be a finite collection of tiles and $t\in\P$. For the set $ E \subset \R^n $ of finite measure that we have fixed in Subsection \ref{Ss:Model} and $ \mathbb{Q} \subseteq \mathbb{P} $, define
		\begin{equation} \label{Eq:DenseDef}	
			\dense(t) \coloneqq \int_{E_t} \chi_{R_t,1}^{10n}, \qquad 
			\densesup(t)  \coloneqq \sup_{ t' \in \mathcal{T}_{\ann}, \, t' \geq t } \dense (t'), \qquad
			\densesup(\mathbb{Q}) \coloneqq \sup_{t\in \mathbb{Q}}\densesup(t),
		\end{equation}
		where recall the definition of $ \chi_{R_t,1}^{10n} $ in \eqref{Eq:DefBumpParallelepiped}.
	\end{definition}

	\subsection{A general decomposition algorithm}
	\label{Ss:DecomAlgorithm}
	
	We explain here a useful algorithm that allows us to sort any collection of tiles $ \T \subseteq \mathcal{T}_{\ann} $ into subcollections with pairwise incomparable maximal spatial components.
	
	\begin{lemma}
		\label{Lem:DecomAlgorithm}
		Let $ \T \subseteq \mathcal{T}_{\ann} $ be a collection of tiles and $ \mu \geq 0 $ be a fixed nonnegative integer. There exists $ \{ P_j \}_{j=1}^{N} \subseteq \T $ and disjoint collections $ \T_{\mu}(P_j) \subseteq \T $ such that
		\begin{equation}
			\label{Eq:SelectionIncomparableDecom}
			\T = \bigcup_{ j=1 }^{N} \T_{\mu}(P_j)
		\end{equation}
		and the following hold:
		\begin{enumerate}[label=\normalfont{(\roman*)}] \setlength\itemsep{.5em}
			\item \label{Item:DecomAlgorithm1} Each $ \T_{\mu}(P_j) $ is a $ K_n 3^{\mu} $-tree with top $ P_j $. Furthermore, if $ \mu = 0 $ then each $ \T_{\mu}(P_j) $ is a $1$-tree.
			
			\item \label{Item:DecomAlgorithm2} If $\, \T $ consists of pairwise incomparable tiles, then for each fixed $ j \in \{1,\ldots,N\} $ the collection $ \{ R_t :\, t \in \T_{\mu}(P_j) \} $ is pairwise disjoint.
			
			\item \label{Item:DecomAlgorithm3} For each fixed $ j \in \{1,\ldots,N\} $ there holds $ \bigcup_{ t \in \T_{\mu}(P_j) } R_t\subseteq K_n 3^{\mu} R_{P_j} $.
			
			\item \label{Item:DecomAlgorithm4} If $ J \subseteq \{1,\ldots,N\} $ is such that $ \bigcap_{ j \in J } \ecc(P_j) \neq \emptyset $ then the collection $ \{ 3^{\mu} R_{P_j} \}_{ j \in J} $ is pairwise disjoint.
			
			\item \label{Item:DecomAlgorithm5} If $\, \T $ is a lacunary/overlapping tree then each $ \T_{\mu}(P_j) $ is lacunary/overlapping. 
		\end{enumerate}
	\end{lemma}
	
	\begin{proof}
		Choose $ P_1 \in \T $ with $ R_{P_1} $ of maximal measure and set
		\begin{equation}
			\label{Eq:DecomAlgorithm}
			\T_{\mu}(P_1) 
			\coloneqq \left\{ t \in \T : \, 3^{\mu} R_{t} \cap 3^{\mu} R_{P_1} \neq \emptyset, \ecc(t) \cap \ecc(P_1) \neq \emptyset \right\}.
		\end{equation}
		Now replace $ \T $ by $ \T \setminus \T_{\mu}(P_1) $ and repeat so that $ P_1,\ldots,P_N $ are selected and \eqref{Eq:SelectionIncomparableDecom} is satisfied.
		
		First, let $ t \in \T_{\mu}(P_j) $. Then $ 3^{\mu} R_{t} \cap 3^{\mu} R_{P_j} \neq \emptyset $ and $ \ecc(t) \cap \ecc(P_j) \neq \emptyset $. Since $ \scl(t) \leq \scl(P_j) $ we get by the geometric Lemma \ref{Lem:Geom} that $ R_t \subseteq 3^{\mu} R_t \subseteq K_n 3^{\mu} R_{P_j} $. This shows (iii) and also (i) by the definition of $ 3^{\rho} $-tree and by the construction of $ \T_{\mu}(P_j) $ which implies that $ P_j $ is a top. Furthermore, (iv) is immediately satisfied by the definition of $ \T_{\mu}(P_j) $.
		
		For (ii) consider $ t,t' \in \T_{\mu}(P_j) $ for some fixed $ j \in \{1,\ldots,N\} $ with $ R_t \cap R_{t'} \neq \emptyset $. Then, by the definition of $ \T_{\mu}(P_j) $ we have $ \ecc(t) \cap \ecc(P_j) \neq \emptyset $ and $ \ecc(t') \cap \ecc(P_j) \neq \emptyset $. Furthermore, by the selection process $ \scl(P_j) \geq \scl(t),\scl(t') $ so $ \ell(Q_{P_j}) \leq \ell(Q_t), \ell(Q_{t'}) $ and thus $ \ecc(P_j) \subseteq \ecc(t) \cap \ecc(t') \neq \emptyset $. In consequence, either $ t \leq t' $ or $ t' \leq t $, which is impossible since all tiles were pairwise incomparable; see also Remark~\ref{rmrk:pairincorp}. This shows (ii).
		
		To prove (v) note that, for $ j \in \{1,\ldots,N\} $ fixed, we have that $ \ecc(t) \cap \ecc(P_j) \neq \emptyset $ and $ \ell(Q_{P_j}) \leq \ell(Q_t) $ for all $ t \in \T_{\mu}(P_j) $. Thus all $ \ecc(t) $ are nested, namely $ \ecc(t) \supseteq \ecc(P_j) $. If $ \T $ is an overlapping tree then $ \ecc(t) \subseteq \ecc(t')^{\circ} $ if $ \ell(Q_t) < \ell(Q_{t'}) $ for $ t,t' \in \T_{\mu}(P_j) $ so $ \ecc(P_j) \subseteq \ecc(t)^{\circ} $ for all $ t \in \T_{\mu}(P_j) $ and the tree $ \T_{\mu}(P_j) $ is overlapping. A similar argument shows that $ \T_{\mu}(P_j) $ is a lacunary tree if $ \T $ is lacunary.
	\end{proof}
	
	For technical reasons it will be useful to introduce a modified notion of $ \size $ defined as
	\begin{equation}
		\size'(\mathbb{Q}) 
		\coloneqq \sup_{\substack{ \mathbf{T}\text{ lac. tree} \\ \mathbf{T} \subseteq \mathbb{Q}, \ \topt(\mathbf{T}) \in \mathbf{T}}} \bigg(\frac{1}{\abs{ R_{\mathbf{T}} }} \sum_{t\in{\mathbf{T}}} F_{10n}[f](t)^2 \bigg)^{1/2}, 
		\qquad \mathbb{Q} \subseteq \mathbb{P} \subseteq \mathcal{T}_{\ann},
	\end{equation}
	where the supremum is taken over lacunary trees with $ \topt(\mathbf{T}) \in \mathbf{T} $.
	
	\begin{lemma}
		\label{Lem:ModifiedSize}
		There holds
		\begin{equation}
			\size'(\mathbb{Q}) \simeq \size(\mathbb{Q}),
			\qquad \mathbb{Q} \subseteq \mathbb{P} \subseteq \mathcal{T}_{\ann}, 
		\end{equation}
		with the implicit constant depending only on dimension.
	\end{lemma}
	
	\begin{proof}
		Let $ \mathbf{T} \subseteq \mathbb{Q} $ be a lacunary tree. Use Lemma \ref{Lem:DecomAlgorithm} with parameter $ \mu = 0 $ to write $ \mathbf{T} = \bigcup_{j=1}^{N} \mathbf{T}(P_j) $, where $ \mathbf{T}(P_j) $ is a lacunary tree with top $ P_j \in \mathbf{T}(P_j) $ for every $ j \in \{1,\ldots,N\} $. These are consequences of (v) and (i) of Lemma \ref{Lem:DecomAlgorithm}. Thus, 
		\begin{equation}
			\begin{split}
				\sum_{t\in \mathbf{T}} F_{10n}[f](t)^2
				& = \sum_{j=1}^{N} \sum_{t\in \mathbf{T}(P_j)} F_{10n}[f](t)^2
				\leq \sum_{j=1}^{N} \size'(\mathbb{Q})^2 \abs{ R_{P_j} }
				\lesssim \size'(\mathbb{Q})^2 \abs{ R_{\mathbf{T}} },
			\end{split}
		\end{equation}
		where in the last inequality we have used that $ \{P_j\}_{j=1}^{N} $ is pairwise disjoint by (iv) of Lemma \ref{Lem:DecomAlgorithm} since $ \mathbf{T} $ is a tree and that $ \bigcup_{j=1}^{N} R_{P_j} \subseteq K_n R_{\mathbf{T}} $, again using that $ \mathbf{T} $ is a tree and \eqref{Eq:TreeShadow}. This shows that $ \size(\mathbb{Q}) \lesssim \size'(\mathbb{Q}) $. The converse inequality is obvious.
	\end{proof}

	\subsection{Sorting by density}
	\label{Ss:SelectionDensity}
	
	The first step of the decomposition of tiles consists of an application of the standard density lemma. The proof is the same as in the one dimensional case of \cite[Proposition 3.1]{LT}, but with some extra steps to deal with the lack of transitivity of the relation $ \leq $ of the tiles.
	
	\begin{lemma}[Selection of Density tops] 
		\label{Lem:Dense} 
		Let $ \P \subset \mathcal{T}_{\ann} $ be a finite collection of tiles. There exists a disjoint decomposition
		\begin{equation}
			\label{Eq:DenseDecom}
			\P = \P^{\light} \cup \P^{\heavy}
		\end{equation}
		satisfying $ \densesup(\P^{\light}) \leq \frac{1}{2} \densesup(\P)$, and a collection of pairwise incomparable tiles $ \T^{\tops}_{\dens} \subseteq \mathcal{T}_{\ann} $ such that
		\begin{enumerate}[label=\normalfont{(\roman*)}] \setlength\itemsep{.5em}
			\item \label{Item:Dense1} $ \displaystyle{ \dense(t) > \frac{1}{2} \densesup(\P) } $ for all $ t \in \T^{\tops}_{\dens} $,
			
			\item \label{Item:Dense2} $ \displaystyle{ \P^{\heavy} = \bigcup_{ t \in \T^{\tops}_{\dens} } \left\{ s \in \P^{\heavy} :\, s \leq_2 t \right\} } $,
			
			\item \label{Item:Dense3} $ \displaystyle{ \sum_{ t \in \T^{\tops}_{\dens} } \abs{ R_{t} }  
				\lesssim \frac{\abs{E}}{\densesup(\P)} } $.
		\end{enumerate}
	\end{lemma}
	
	\begin{remark}
		\label{Rmrk:DenseMap}
		Because of (ii) above we can define a map $ \P^{\heavy} \ni s \mapsto t(s) \in \T_{\dens}^{\tops} $ such that $ s \leq_{2} t(s) $ for all $ s \in \P^{\heavy} $ and
		\begin{equation}
			\label{Eq:DenseMap}
			\P^{\heavy} 
			= \bigcup_{ t \in \T_{\dens}^{\tops} } \left\{ s \in \P^{\heavy} :\, t(s) = t \right\}
			\eqqcolon \bigcup_{ t \in \T_{\dens}^{\tops} } \P^{\heavy}_{t}. 
		\end{equation}
	\end{remark}
	
	\begin{proof}
		Set $ \delta \coloneqq \densesup(\P) $ and let 
		\begin{equation}
			\P^{\heavy} \coloneqq \left\{ s \in \P : \, \densesup(s) > \delta/2 \right\}, \qquad \P^{\light} \coloneqq \left\{ s \in \P :\, \densesup(s) \leq \delta/2 \right\}.
		\end{equation}
		For each $ s \in \P^{\heavy} $ there exists a tile $ t_s \in \mathcal{T}_{\ann} $ such that $ s \leq t_s $ and $ \dense(t_s) > \delta/2 $. Note that possibly this $ t_s $ is not unique but we choose one for each $ s \in \P^{\heavy} $. This defines the map described in Remark~\ref{Rmrk:DenseMap}. Now let $ \T = \T_{\dens}^{\tops} $ denote the maximal elements with respect to $ \leq $ of the collection $ \left\{ t_s :\, s \in \P^{\heavy} \right\} $. It remains to prove (iii), since the other properties are immediately satisfied. By the definition of $ \dense $, we can decompose 
		\begin{equation}
			\T = \T_{\dens}^{\tops} = \bigcup_{ k=0 }^{\infty} \T_k,
		\end{equation}
		where for each $ t \in \T_k $, $k$ is the least positive integer for which
		\begin{equation}
			\label{Eq:DenseProofDensityDecomp}
			\Abs{ E \cap v_{ \sigma }^{-1}(\alpha_{t,\tau_0}) \cap 3^{k} R_{t} } \geq \frac{1}{100} \delta 3^{5 n k} \Abs{ 3^{k} R_{t} }
		\end{equation}
		holds, with $ \alpha_{t,\tau_0} $ defined in \eqref{Eq:DirecSuppDef} and recall that $ \tau_0 \in \{1,\ldots,3^{\kappa d} - 1 \} $ was fixed at the beginning of this section. Now we apply Lemma \ref{Lem:DecomAlgorithm} with the parameter $ \mu = k $ to get
		\begin{equation}
			\sum_{ t \in \T } \abs{ R_{t} } 
			= \sum_{k=0}^{\infty} \sum_{ t \in \T_k } \abs{ R_{t} }
			= \sum_{k=0}^{\infty} \sum_{ j=1 }^{N_k} \sum_{ t \in \T_k(P_{k,j}) } \abs{ R_{t} } 
			= \sum_{k=0}^{\infty} \sum_{ j=1 }^{N_k} \Abs{ \bigcup_{ t \in \T_k(P_{k,j}) } R_{t} }
			\leq \sum_{k=0}^{\infty} \sum_{ j=1 }^{N_k} K_n^n \Abs{ 3^{k} R_{P_{k,j}} },
		\end{equation}
		where in the third equality we have used (ii) of Lemma \ref{Lem:DecomAlgorithm} and in the last inequality (iii) of the same lemma. Now, using \eqref{Eq:DenseProofDensityDecomp} since $ P_{k,j} \in \T_k $ for every $ j \in \{1,\ldots,N_k\} $
		\begin{equation}
			\begin{split}
				\sum_{ t \in \T } \abs{ R_{t} } 
				& \lesssim \delta^{-1} \sum_{k=0}^{\infty}  3^{-5nk} \sum_{ j=1 }^{N_k} \Abs{ E \cap v_{ \sigma }^{-1}(\alpha_{P_{k,j},\tau_0}) \cap 3^{k} R_{P_{k,j}} } \\
				& = \delta^{-1} \sum_{k=0}^{\infty} 3^{-5nk} \int_{E} \sum_{ j=1 }^{N_k} \indf{ v_{ \sigma }^{-1}(\alpha_{P_{k,j},\tau_0}) }(x) \indf{ 3^{k} R_{P_{k,j} } }(x) dx.
			\end{split}
		\end{equation}
		Fix $ x \in E $ and consider the index set $ J_x \subseteq \{1,\ldots,N_k\} $ such that $ \indf{ v_{ \sigma }^{-1}(\alpha_{P_{k,j},\tau_0}) }(x) \indf{ 3^{k} R_{P_{k,j} } }(x) \neq 0 $ for every $ j \in J_x $. Then, $ v_{\sigma}(x) \in \ecc(P_{k,j}) $ for all $ j \in J_x $ and thus $ \{ \ecc(P_{k,j}) \}_{j\in J_x} $ are nested. In consequence, the collection $ \{ 3^{k} R_{P_{k,j}} \}_{j\in J_x} $ is pairwise disjoint by (iv) of Lemma \ref{Lem:DecomAlgorithm} and the proof follows.
	\end{proof}

	\subsection{Sorting by size}
	\label{Ss:SelectionSize}
	
	The next step is to sort the tiles by their size. The heuristic is to apply the standard size lemma to the collection of tiles $ \P^{\heavy} $ from \eqref{Eq:DenseDecom}. However, we cannot do it directly since we would lose the sorting into the tops $ \T^{\tops}_{\dens} $ and the density estimate (iii) of Lemma \ref{Lem:Dense}. Instead, we follow a more careful procedure. 
	
	The combined density-size decomposition below is the main goal of this subsection.
	
	\begin{proposition}
		\label{Prop:CombinedDensitySize}
		Let $ \P \subseteq \mathcal{T}_{\ann} $ be a finite collection of tiles. There exists a disjoint decomposition
		\begin{equation}
			\label{Eq:SizeDecom}
			\P = \P^{\light} \cup \P^{\heavy}_{\smal} \cup \bigcup_{j} \bigcup_{\nu=1}^{V_n} \mathbf{T}_{j,\nu}'
		\end{equation}
		with $ V_n = O_n(1) $ uniformly in $ j $ and such that $ \densesup(\P^{\light}) \leq \frac{1}{2} \densesup(\P) $, $ \size(\P^{\heavy}_{\smal}) \leq \frac{1}{2} \size(\P) $ and $ \densesup(\P^{\heavy}_{\smal}) > \frac{1}{2} \densesup(\P) $ if $ \,\P^{\heavy}_{\smal} $ is non-empty and where each $ \mathbf{T}_{j,\nu}' $ is a tree. Furthermore, for each $ j $ the collection $ \{ R_{ \mathbf{T}_{j,\nu}' } \}_{\nu=1}^{V_n} $ is formed by $ V_n $ congruent copies of one $ R_{\mathbf{T}_{j,\nu_0}'} $ and they satisfy
		\begin{equation}
			\label{Eq:SizeEstimates}
			\sum_{j,\nu} \Abs{ R_{ \mathbf{T}_{j,\nu}' } } 
			\lesssim \min \left\{ \frac{ \abs{E} }{ \densesup(\P) }, \, \frac{ \norm{f}{L^2(\R^n)}^2}{ \size(\P)^2 } \right\}.
		\end{equation}
		If $ \{ \mathbf{T}_{j,\nu}' \}_{\nu=1}^{V_n} \neq \emptyset $ then for every $ j $ and $ \nu $ there holds $ \densesup( \mathbf{T}_{j,\nu}' ) \simeq \densesup( \P ) $ and for every $ j $ there exists $ \nu_{0} \in \{1,\ldots,V_n\} $ and a lacunary tree $ \mathbf{T}_{j,\nu_0} \subseteq \mathbf{T}_{j,\nu_0}' $ with $ R_{ \mathbf{T}_{j,\nu_0} } = R_{ \mathbf{T}_{j,\nu_0}' } $ and such that 
		\begin{equation}
			\sum_{ t \in \mathbf{T}_{j,\nu_0} } F_{10n}[f](t)^2 \simeq \size( \P )^2 \Abs{ R_{\mathbf{T}_{j,\nu_0}} }.
		\end{equation}
	\end{proposition}
	
	For the proof of Proposition \ref{Prop:CombinedDensitySize} we will need the following technical definition, introduced in \cite[Definition 6.11]{BDPPR}.
	
	\begin{definition}[Signature]
		Let $ \mathcal{G} $ be a triadic grid in $ \R^d $ and $ \eta \in \R^d $ having no triadic coordinates. The signature of $ \eta $ is defined as the number
		\begin{equation}
			\sig(\eta) \coloneqq \sum_{j=1}^{\infty} \frac{ a_{j} }{ 3^j }, \qquad 
			a_j \coloneqq \left\{ 
			\begin{array}{lcl} 
				0, & \text{ if } & \eta \in Q^{\circ}, \text{ for some } Q \in \mathcal{G}_{ j } , \\ 
				1, & \text{ if } & \eta \not\in Q^{\circ}, \text{ for all } Q \in \mathcal{G}_{ j }.
			\end{array} 
			\right.
		\end{equation} 
	\end{definition}
	
	As noted in \cite[Remark 6.12]{BDPPR}, the importance of this definition is that for two different lacunary trees $ \mathbf{T}_1, \mathbf{T}_2 $, if $ t_j \in \mathbf{T}_j $ for $ j=1,2 $, with $ \ecc(t_1) \subseteq \ecc(t_2)^{\circ} $, then $ \sig( \eta_{\mathbf{T}_1} ) < \sig( \eta_{\mathbf{T}_2} ) $, where recall that $ \eta_{\mathbf{T}_j} \in Q_{\mathbf{T}_j} $ is the point that satisfies $ (\eta_{\mathbf{T}_j} \times \{1\})' = \xi_{\mathbf{T}_j} $. This fact is used to show that the collection of lacunary trees $ \mathfrak{T}^{\lac} $ below is strongly disjoint.
	
	\begin{proof}[Proof of Proposition \ref{Prop:CombinedDensitySize}] As the proof of the proposition is rather lengthy, we divide it into clearly marked and mostly self-contained blocks.

		\emph{Decomposition of the collection of tiles:} By Lemma \ref{Lem:Dense} we have the decomposition $ \P = \P^{\light} \cup \P^{\heavy} $. We abbreviate $ \delta \coloneqq \densesup(\P) $ and $ \sigma \coloneqq \size( \P^{\heavy} ) $. We can assume that $ \sigma > \size(\P)/2 $ for otherwise there is nothing to prove. Now, consider all lacunary trees $ \mathbf{T} \subseteq \P^{\heavy} $ with $ \topt(\mathbf{T}) \in \mathbf{T} $ and such that
		\begin{equation}
			\label{Eq:BigSizeTrees}
			\sum_{ t \in \mathbf{T} } F_{10n}[f](t)^2 > \frac{ \sigma^2}{C_n} \Abs{ R_{\mathbf{T}} }
		\end{equation}
		with $ C_n $ a large dimensional constant that will be chosen in the course of the proof. Among those, let $ \mathbf{T}_1 $ be the one with $ \eta_{\mathbf{T}_1} $ having minimal signature, where recall again that $ \eta_{\mathbf{T}_1} \in Q_{\mathbf{T}_1} $ is the point that satisfies $ (\eta_{\mathbf{T}_1} \times \{1\})' = \xi_{\mathbf{T}_1} $. Next, let $ \mathbf{T}_1' \subseteq \P^{\heavy} $ be the maximal $ K_n^2 $-tree with the same top as $ \mathbf{T}_1 $ and containing $ \mathbf{T}_1 $. Replace $ \P^{\heavy} $ by $ \P^{\heavy} \setminus \mathbf{T}_1' $ and repeat. The algorithm terminates after finitely many steps since the initial collection $ \P $ is finite. That way we choose a collection of lacunary trees $ \{ \mathbf{T}_j \}_j $ with $ \topt(\mathbf{T}_j) \eqqcolon s_j \in \P^{\heavy} $ and the collection of $ K_n^2 $-trees $ \{ \mathbf{T}_j' \}_j $ satisfying $ \mathbf{T}_j' \supseteq \mathbf{T}_j $. On the other hand, for the remaining collection of tiles, set $ \P_{\smal}^{\heavy} \coloneqq \P^{\heavy} \setminus \bigcup_j \mathbf{T}_j' $ and by Lemma \ref{Lem:ModifiedSize}
		\begin{equation}
			\size \left( \P_{\smal}^{\heavy} \right)
			\leq \gamma_n \size'\left( \P_{\smal}^{\heavy} \right)
			\leq \gamma_n \frac{ \sigma }{ \sqrt{C_n} },
		\end{equation}
		with $ \gamma_n $ the implicit constant in $ \size(\mathbb{Q}) \lesssim \size'(\mathbb{Q}) $ from Lemma \ref{Lem:ModifiedSize}. Choosing $ C_n = 4 \gamma_n^2 $ we get
		\begin{equation}
			\P = \P^{\light} \cup \P^{\heavy}_{\smal} \cup \bigcup_{j} \mathbf{T}_{j}'
		\end{equation}
		with $ \densesup(\P^{\light}) \leq \frac{1}{2} \densesup(\P) $ and $ \size(\P^{\heavy}_{\smal}) \leq \frac{1}{2} \sigma \leq \frac{1}{2} \size(\P) $.
		
		Now, by Lemma \ref{Lem:Split} each $ \mathbf{T}_{j}' $ can be written as a disjoint union of $ O_n(1) = V_n $ $1$-trees $ \{ \mathbf{T}_{j,\nu}' \}_{\nu=1}^{V_n} $ each one with top $ (\xi_{\mathbf{T}_{j}'},R_{\mathbf{T}_{j,\nu}'}) $ and $ R_{ \mathbf{T}_{j,\nu}' } $ a congruent copy of $ R_{ \mathbf{T}_{j} } $ with $ \{ R_{ \mathbf{T}_{j,\nu}' } \}_{\nu=1}^{V_n} $ being pairwise disjoint. This shows \eqref{Eq:SizeDecom} and furthermore we get
		\begin{equation}
			\label{Eq:SizeCongruentCopies}
			\sum_{\nu=1}^{V_n} \Abs{ R_{\mathbf{T}_{j,\nu}'}} \simeq \Abs{ R_{ \mathbf{T}_{j} } }.
		\end{equation}

		\emph{The size and density estimates in \eqref{Eq:SizeEstimates}:} Let us denote by $ \T_{\siz}^{\tops} $ the collection $ \{ \topt(\mathbf{T}_j) \}_{j} \subseteq \P^{\heavy} $. By Lemma \ref{Lem:Dense} and Remark \ref{Rmrk:DenseMap} we have the disjoint decomposition
		\begin{equation}
			\label{Eq:DecomSizeTopsByDense}
			\T_{\siz}^{\tops} = \bigcup_{ t \in \T_{\dens}^{\tops} } \left\{ s \in \T_{\siz}^{\tops} :\, t(s) = t \right\} \eqqcolon \bigcup_{ t \in \T_{\dens}^{\tops} } \T_{\siz,t}^{\tops},
		\end{equation}
		where $ \T_{\dens}^{\tops} \subseteq \mathcal{T}_{\ann} $ is a pairwise incomparable collection of tiles satisfying $ \dense(t) > \densesup(\P)/2 $. Thus,
		\begin{equation}
			\sum_j \Abs{ R_{ \mathbf{T}_{j} } } = \sum_{ s \in \T_{\siz}^{\tops} } \abs{R_s} = \sum_{ t \in \T_{\dens}^{\tops} } \sum_{ s \in \T_{\siz,t}^{\tops} } \abs{R_s}.
		\end{equation}
		We apply Lemma \ref{Lem:DecomAlgorithm} with $ \mu = 0 $ to the collection $ \T_t = \T_{\siz,t}^{\tops} $ for each $ t \in \T_{\dens}^{\tops} $ to get
		\begin{equation}
			\label{Eq:SizeTopsDecomAlgorithm}
			\T_t = \T_{\siz,t}^{\tops} = \bigcup_{\ell=1}^{N_t} \T_{t}(P_{t,\ell}),
		\end{equation}
		where $ P_{t,\ell} \in \T_{\siz,t}^{\tops} $ for every $ \ell \in \{1,\ldots,N_t\}$. The collection of $ \T_{\siz,t}^{\tops} $ consists of pairwise incomparable tiles since each tile in $ \T_{\siz,t}^{\tops} $ is contained in a different $ K_n^2 $-tree $ \{ \mathbf{T}_j' \}_j $ constructed above. Then, (ii) of Lemma \ref{Lem:DecomAlgorithm} implies for each $ \ell \in \{1,\ldots,N_t\} $
		\begin{equation}
			\sum_{ s \in \T_{t}(P_{t,\ell}) } \abs{ R_s } 
			= \Abs{ \bigcup_{ s \in \T_{t}(P_{t,\ell}) } R_s } 
			\lesssim \Abs{ R_{P_{t,\ell}}},
		\end{equation}
		where in the last inequality we have used (iii) of Lemma \ref{Lem:DecomAlgorithm}. Note that $ t(P_{t,\ell}) = t $ for every $ \ell \in \{1,\ldots,N_t\} $ which implies that $ \ecc(t) \subseteq \ecc(P_{t,\ell}) $ and by (iv) of Lemma \ref{Lem:DecomAlgorithm}
		\begin{equation}
			\sum_{ \ell=1 }^{N_t} \abs{ R_{P_{t,\ell}} } 
			= \Abs{ \bigcup_{ \ell=1 }^{N_t} R_{P_{t,\ell}} }
			\lesssim \abs{ R_t },
		\end{equation}
		where in the last inequality we have used the geometric Lemma \ref{Lem:Geom}. In conclusion,
		\begin{equation}
			\sum_j \Abs{ R_{ \mathbf{T}_{j} } } 
			= \sum_{ t \in \T_{\dens}^{\tops} } \sum_{ \ell=1 }^{N_t} \sum_{ s \in \T_{t}(P_{t,\ell}) } \abs{R_s} 
			\lesssim \sum_{ t \in \T_{\dens}^{\tops} } \abs{R_t}
			\lesssim \frac{ \abs{E} }{ \delta },
		\end{equation}
		by (iii) of Lemma \ref{Lem:Dense}. This shows the first estimate in \eqref{Eq:SizeEstimates}.
		
		To prove the second estimate in \eqref{Eq:SizeEstimates}, first we show that the lacunary trees $ \{ \mathbf{T}_j \}_j $ are strongly disjoint. For this let $ t_1 \in \mathbf{T}_{j_1} $ and $ t_2 \in \mathbf{T}_{j_2} $ with $ j_1 \neq j_2 $ and $ \ecc(t_1) \subseteq \ecc(t_2)^{\circ} $. Then, the tree $ \mathbf{T}_{j_1} $ was selected first because of signature considerations and $ \scl(R_{t_2}) \leq \scl( R_{\mathbf{T}_{j_1}} ) $. Suppose that $ \{ \mathbf{T}_j \}_j $ are not strongly disjoint, that is $ R_{t_2} \cap K_n^2 R_{ \mathbf{T}_{j_1} } \neq \emptyset $. Then $ t_2 $ would qualify for $ \mathbf{T}_{j_1}' $, leading to a contradiction. Thus, since $ \{ \mathbf{T}_j \}_j $ is a collection of strongly disjoint lacunary trees we can apply (2) in Lemma \ref{Lem:OrthoLacTree} together with \eqref{Eq:BigSizeTrees} to get
		\begin{equation}
			\sum_{j} \sum_{ \nu=1 }^{ V_n } \abs{ R_{\mathbf{T}_{j,\nu}'} }
			\simeq \sum_{j} \abs{ R_{\mathbf{T}_j} } 
			\lesssim \sigma^{-2} \sum_{j} \sum_{ t \in \mathbf{T}_j } F_{10n}[f](t)^2
			\lesssim \sigma^{-2} \norm{f}{L^2(\R^n)},
		\end{equation}
		where in the first equality we have used \eqref{Eq:SizeCongruentCopies}. This finishes the proof of the second estimate in \eqref{Eq:SizeEstimates}.
		
		\emph{Big size and density of the trees:} Note that all the trees constructed above are inside $ \P^{\heavy} $ so they satisfy $ \densesup(\mathbf{T}_{j,\nu}') \geq \delta/2 $. For the size we can only conclude that for each $ j $ there exists $ \nu_0 \in \{1,\ldots,V_n\} $ such that
		\begin{equation}
			\sum_{ t \in \mathbf{T}_{j,\nu_0} } F_{10n}[f](t)^2 \gtrsim \sigma^2 \Abs{ R_{ \mathbf{T}_{j,\nu_0} } },
		\end{equation}
		where $ \mathbf{T}_{j,\nu_0} $ is a lacunary tree contained in $ \mathbf{T}_j $. Indeed note that the splitting $ \mathbf{T}_j' = \bigcup_{\nu=1}^{V_n} \mathbf{T}_{j,\nu}' $ induces a corresponding splitting of $ \mathbf{T}_j $
		\begin{equation}
			\mathbf{T}_j = \bigcup_{\nu=1}^{V_n} \mathbf{T}_{j,\nu}, \qquad \mathbf{T}_{j,\nu} = \mathbf{T}_{j} \cap \mathbf{T}_{j,\nu}' ,
		\end{equation}
		with each $ R_{ \mathbf{T}_{j,\nu} } $ congruent to $ R_{\mathbf{T}_{j} } $ and $ V_n $ depends only on the dimension. The trees $ \mathbf{T}_{j,\nu} $ are produced via Lemma \ref{Lem:Split} and it is easy to see that each $ \mathbf{T}_{j,\nu} $ is lacunary. Since
		\begin{equation}
			\sum_{ t \in \mathbf{T}_{j} } F_{10n}[f](t)^2 
			= \sum_{\nu=1}^{V_n} \sum_{ t \in \mathbf{T}_{j,\nu} } F_{10n}[f](t)^2 \gtrsim \sigma^2 \Abs{ R_{ \mathbf{T}_{j} } },
		\end{equation} 
		it follows by pigeonholing that there exists $ \nu_0 \in \{1,\ldots,V_n\} $ such that
		\begin{equation}
			\sum_{ t \in \mathbf{T}_{j,\nu_0} } F_{10n}[f](t)^2 
			\gtrsim \sigma^2 \Abs{ R_{ \mathbf{T}_{j} } }
			= \sigma^2 \Abs{ R_{ \mathbf{T}_{j,\nu_0} } },
		\end{equation}
		concluding the proof of the proposition.
	\end{proof}

	\subsection{The maximal estimate for the tops}
	\label{Ss:DecomMax}
	
	Besides the size and density estimates in \eqref{Eq:SizeEstimates} of Proposition \ref{Prop:CombinedDensitySize}, we need an additional estimate which is critical for accessing the range $ 1 < p < 2 $ in Theorem~\ref{Thm:MainSingleBand}.
	
	\begin{proposition}
		\label{Prop:DecomMax}
		For every $0 < \varepsilon < 1 $, the trees constructed in Proposition \ref{Prop:CombinedDensitySize} additionally satisfy
		\begin{equation}
			\label{Eq:MaxEstimate}
			\sum_{j,\nu} \Abs{ R_{\mathbf{T}_{j,\nu}'} } \lesssim_{\varepsilon} \frac{ \abs{ F }^{1-\varepsilon} \abs{E}^{\varepsilon} }{ \densesup(\P) \size(\P)^{1+\varepsilon} }.
		\end{equation}
	\end{proposition}
	
	The proof of this estimate is rather involved and relies on several auxiliary results together with a rather deep geometrical-analytical estimate for a Kakeya-type maximal function from \cites{Bateman09Proc,Bateman2013Revista}. We postpone the proof until Section \ref{S:MaxLemma} in order not to interrupt the logical flow of the main argument.
	
	\subsection{Iterative decomposition}
	\label{Ss:Decom}
	
	The combined density and size decomposition of Proposition \ref{Prop:CombinedDensitySize} together with the maximal estimate of Proposition \ref{Prop:DecomMax} lead to the following final decomposition of all tiles into trees with the desirable density, size and maximal estimates.
	
	\begin{lemma} 
		\label{Lem:Decom}
		Let $ \ann \in 3^{\Z} $ be fixed and $ \P \subseteq \mathcal{T}_{\ann} $ be a finite collection of tiles. There exist integers $K,J\leq 0$ depending only on the dimension and a disjoint decomposition
		\begin{equation} \label{Eq:Decom}
			\P = \bigcup_{ k\geq K } \bigcup_{ j\geq J } \bigcup_{\tau=1}^{N_{k,j}} \mathbf{T}_{k,j,\tau}
		\end{equation}
		with each $ \mathbf{T}_{k,j,\tau} $ being a tree and for each $ k\geq K$ ,$j\geq J$
		\begin{equation} 
			\label{Eq:DecomEstimatesDS}
			\sum_{\tau=1}^{N_{k,j}} \Abs{ R_{\mathbf{T}_{k,j,\tau}} } \lesssim_{\varepsilon} \min \left\{  2^{k} \abs{E}, 2^{2 j} \abs{ F }, \, 2^{k} 2^{(1+\varepsilon) j } \abs{F}^{1-\varepsilon} \abs{E}^{\varepsilon} \right\},
		\end{equation}
		for every $ \varepsilon > 0 $.
	\end{lemma}
	
	\begin{proof}
		Starting with a finite collection $ \P $ note that we have the a priori estimate $ \densesup(\P)\lesssim 1 $ and also $ \size(\P) \lesssim 1 $ as consequence of (1) of Lemma \ref{Lem:OrthoLacTree}. Then, we can assume that $2^{-K-1} < \densesup(\P) \leq 2^{-K} $ and $ 2^{-J-1}<\size(\P) \leq 2^{-J} $ for some nonpositive integers $J,K$ depending only on the dimension. Applying Proposition \ref{Prop:CombinedDensitySize} and using Proposition \ref{Prop:DecomMax} we get
		\begin{equation}
			\P = \P^{\light} \cup \P^{\heavy}_{\smal} \cup \bigcup_{\tau=1}^{N_{K,J}} \mathbf{T}_{K,J,\tau},
		\end{equation}
		where each $ \mathbf{T}_{K,J,\tau} $ is a tree and where $ \densesup( \P^{\light}) < 2^{-K-1} $, $ \size(\P^{\heavy}_{\smal}) \leq 2^{-J-1} $ and 
		\begin{equation}
			\sum_{\tau=1}^{N_{K,J}} \Abs{ R_{ \mathbf{T}_{K,J,\tau} } } 
			\lesssim_{\varepsilon} \min \left\{ 2^K \abs{E}, \, 2^{2J} \abs{F}, \, 2^{K} 2^{(1+\varepsilon) J } \abs{F}^{1-\varepsilon} \abs{E}^{\varepsilon} \right\}.
		\end{equation}
		As long as $ \P^{\heavy}_{\smal} \neq \emptyset $ we can iterate Proposition \ref{Prop:CombinedDensitySize} starting from $ \P^{\heavy}_{\smal} $, producing trees $ \{ \mathbf{T}_{K,j,\tau} \}_{\tau=1}^{N_{K,j}}$ for $j\geq J$. This way we decompose 
		\begin{equation}
			\P = \P^{\light} \cup \bigcup_{j=J}^{\infty} \bigcup_{\tau=1}^{N_{K,j}} \mathbf{T}_{K,j,\tau}
		\end{equation}
		with
		\begin{equation}
			\sum_{\tau=1}^{N_{K,j}} \Abs{ R_{ \mathbf{T}_{K,j,\tau} } } 
			\lesssim_{\varepsilon} \min \left\{ 2^K \abs{E}, \, 2^{2j} \abs{F}, \, 2^{K} 2^{(1+\varepsilon) j } \abs{F}^{1-\varepsilon} \abs{E}^{\varepsilon} \right\}
		\end{equation}
		for every $ j \geq J $. We can now iterate the selection algorithm starting from $ \P^{\light} $ where we have $ \densesup( \P^{\light}) \leq 2^{-K-1} $ and $ \size(\P^{\light}) \leq 2^{-J} $. Note that it may happen that no trees are constructed in some of the iterations.
	\end{proof}
	
	\begin{remark}\label{Remark:MaximalCase}
		A simple calculation shows that the density estimate is stronger than the maximal estimate for every $\varepsilon>0$, namely the rightmost estimate in \eqref{Eq:DecomEstimatesDS}, whenever $\size(\P) \abs{E}\leq \abs{F}$
		\begin{equation}
			\frac{|E|}{\densesup(\mathbb P)}\leq \frac{|E|^{\varepsilon} |F|^{1-\varepsilon}}{\densesup(\P) \size(\P)^{1-\varepsilon}}\lesssim \frac{|E|^{\varepsilon} |F|^{1-\varepsilon}}{\densesup(\P)\size(\P)^{1+\varepsilon}}
		\end{equation}
		since we always have $\size(\P)\lesssim 1$. Because of the latter apriori estimate we also note that $\size(\P) \abs{E}\lesssim \abs{F}$ \emph{always holds} whenever $|E|\leq |F|$. We conclude that that the maximal estimate needs only be used in the case that $|E|>|F|$.
	\end{remark}
	
	\subsection{Proof of the estimate for the model sum}
	\label{Ss:ProofModel}
	
	In this subsection, we will put everything together in order to prove the desired inequality \eqref{Eq:DesiredModelEst} in Proposition \ref{Prop:Model} and conclude the proof of Theorem~\ref{Thm:MainSingleBand}. First, since we decompose the collection of tiles $ \P $ into trees, we need the following single tree estimate. The version which is suitable for our setup is taken verbatim from \cite[Lemma 6.14]{BDPPR}. We remark that in \cite[Lemma 6.14]{BDPPR} the authors use a slightly different definition for lacunary trees, but this creates no problem in the proof of the tree lemma in \cite{BDPPR}*{Proof of (6.5)}.
	
	\begin{lemma}[Tree Lemma] 
		\label{Lem:Tree} 
		Let $\mathbf{T} \subseteq \mathcal{T}_{\ann} $ be a tree for some fixed $ \ann \in 3^{\Z} $. There holds
		\begin{equation}
			\Lambda_{\mathbf{T}; \sigma, \tau_{0}, M}(f,g\indf{E})\lesssim \size(\mathbf{T}) \densesup(\mathbf{T})\abs{ R_{\mathbf{T}} },
		\end{equation}
		with $ \size $ defined with respect to $ f \in L_{0}^{\infty}(\R^n) $, $ \densesup $  defined with respect to $ E \subseteq \R^n $ with $ \abs{ E } < \infty $ and $ g \in L_{0}^{\infty}(\R^n) $ with $ \norm{ g }{\infty} = 1 $.
	\end{lemma}
	
	Now we are ready to prove the desired inequality \eqref{Eq:DesiredModelEst} for the model operator in Proposition \ref{Prop:Model}.
	
	\begin{proof}[Proof of Proposition \ref{Prop:Model}]
		As in the proof of Lemma \ref{Lem:Decom}, we can assume that $ 2^{-K-1}<\densesup(\P)\leq 2^{-K} $ and $2^{-J-1}< \size(\P) \leq 2^{-J} $ for some negative integers $ J,K $. A combination of Lemma \ref{Lem:Tree} and Lemma \ref{Lem:Decom} yields that we can bound the left hand side of \eqref{Eq:DesiredModelEst} by
		\begin{equation}
			\label{Eq:EstimatesDenseSizeMax}
			\Lambda_{\P; \sigma, \tau_{0}, M} (f\indf{F},g\indf{E})
			\lesssim_{\varepsilon} \sum_{ k=K }^{\infty} \sum_{j=J}^{\infty} 2^{-k} 2^{-j} \min \left\{ 2^k \abs{E}, \, 2^{2j} \abs{F}, \, 2^{k} 2^{(1+\varepsilon) j } \abs{F}^{1-\varepsilon} \abs{E}^{\varepsilon} \right\},
		\end{equation}
		for every $\varepsilon>0$. The case $ p=2 $ follows exactly as the weak-type $ (2,2) $ bound from \cite{BDPPR}. Indeed, we have
		\begin{equation}
			\begin{split}
				\Lambda_{\P; \sigma, \tau_{0}, M} (f\indf{F},g\indf{E})
				& \lesssim \sum_{ k=K }^{\infty} \sum_{j=J}^{ \frac{1}{2} \log \frac{2^k \abs{ E }}{ \abs{ F } } } 2^{-k} 2^{j} \abs{ F } + \sum_{ k=K }^{\infty} \sum_{j= \frac{1}{2} \log \frac{2^k \abs{ E }}{ \abs{ F } } }^{\infty} 2^{-j} \abs{ E } \\
				& \lesssim \sum_{ k=K }^{\infty} 2^{-k/2} \abs{ E }^{1/2} \abs{ F }^{1/2}
				\lesssim \abs{ E }^{1/2} \abs{ F }^{1/2}.
			\end{split}
		\end{equation}
		For $ p \geq 2 $ we may assume that $ \abs{ E } \leq \abs{ F } $ for otherwise the desired estimate follows from the case $ p=2 $. Then, recalling Remark \ref{Remark:MaximalCase}, we need to use only the first two estimates in the minimum in the right hand side of \eqref{Eq:EstimatesDenseSizeMax}. Thus,
		\begin{equation}
		\begin{split}
			\Lambda_{\P; \sigma, \tau_0, M} (f\indf{F},g\indf{G}) 
			& \lesssim \sum_{j=J}^{\infty} \sum_{k=K}^{ \log \frac{2^{2j} \abs{F}}{\abs{E}} } 2^{-j} \abs{ E } + \sum_{j=J}^{\infty} \sum_{ \log \frac{2^{2j} \abs{F}}{\abs{E}} }^{\infty} 2^{-k} 2^{ j } \abs{ F } \\
			& \lesssim \sum_{j=J}^{\infty} 2^{-j} \abs{ E } \left( j + \log \frac{ \abs{F} }{ \abs{E} } \right) + \sum_{j=J}^{\infty} 2^{ -j } \abs{ E }
			\lesssim \abs{E} \left( 1 + \log \frac{ \abs{F} }{ \abs{E} } \right)
			\lesssim \abs{F}^{\frac{1}{p}} \abs{E}^{1-\frac{1}{p}}.
		\end{split}
		\end{equation}
		Finally, for $ p \leq 2 $, we can assume that $ \abs{ F } \leq \abs{ E } $ for otherwise the desired estimate follows from the case $ p=2 $. Then, splitting the sum as
		\begin{equation}
		\begin{split}
			\Lambda_{\P; \sigma, \tau_0, M} (f\indf{F},g\indf{G}) 
			& \lesssim \sum_{j=J}^{ \log \frac{ \abs{E}}{\abs{F}}} \sum_{k=K}^{j} 2^{j \varepsilon} \abs{F}^{1-\varepsilon} \abs{E}^{\varepsilon}
			+ \sum_{j=J}^{\log \frac{ \abs{E}}{\abs{F}}} \sum_{k=j}^{\infty} 2^{-k} 2^{j} \abs{F} \\
			& \phantom{==} + \sum_{j= \log \frac{ \abs{E}}{\abs{F}}}^{\infty} \sum_{k=K}^{\log \frac{2^{2 j} \abs{F}}{\abs{E}} } 2^{-j} \abs{ E }
			+ \sum_{ j= \log \frac{ \abs{E}}{\abs{F}} }^{\infty} \sum_{k=\log \frac{2^{2j} \abs{F}}{\abs{E}} }^{\infty} 2^{-k} 2^{j} \abs{F}
		\end{split}
		\end{equation}
		and doing the calculations, we get
		\begin{equation}
		\begin{split}
			\Lambda_{\P; \sigma, \tau,M} (f\indf{F},g\indf{G}) 
			& \lesssim \abs{F}^{1-2\varepsilon} \abs{E}^{2 \varepsilon} \left( 1+\log \frac{\abs{E}}{\abs{F}} \right) + \abs{F} \log \frac{\abs{E}}{\abs{F}} + \Abs{F} 
			\lesssim \abs{F}^{\frac{1}{p}} \abs{E}^{1-\frac{1}{p}},
		\end{split}
		\end{equation}
		by appropriately choosing $ \varepsilon < 1 $ depending on $ 1<p<2 $.
	\end{proof}

	\section{The maximal estimate for the tops} 
	\label{S:MaxLemma}
	
	This section is dedicated to the proof of Proposition \ref{Prop:DecomMax}. Before going to the heart of the proof we state and prove several auxiliary results.
	
	\subsection{The maximal function estimate}
	\label{Ss:MaxLemma}
	
	The main ingredient for the proof of estimate \eqref{Eq:MaxEstimate} for the sums of tops is the following geometric/analytical maximal function estimate originally proved by Bateman in two dimensions in \cite{Bateman09Proc}, see also \cite{Bateman2013Trans}.
	
	\begin{proposition}
		\label{Prop:MaxKey}
		Let $ \ann \in 3^{\Z} $, $ k \in \N $ be fixed and $ \T \subseteq \mathcal{T}_{\ann} $ be a pairwise incomparable collection of tiles $ t $ such that, for some $\mu,\lambda\in(0,1)$, there holds
		\begin{equation}
			\label{Eq:MaxKeyHypo1}
			\Abs{ v_{\sigma}^{-1}(\alpha_{t,\tau_0}) \cap 3^k R_t } \geq \mu \Abs{ 3^k R_t },
		\end{equation}
		and
		\begin{equation}
			\label{Eq:MaxKeyHypo2}
			\Abs{ F \cap 3^k R_t } > \lambda \Abs{ 3^k R_t },
		\end{equation}
		for every $ t \in \T $. Then, for every $ \varepsilon > 0 $,
		\begin{equation}
			\sum_{ t \in \T } \Abs{ R_t } \lesssim_{\varepsilon} \frac{ \abs{ F } }{ \mu \lambda^{1+\varepsilon} }.
		\end{equation}
	\end{proposition}
	
	The rest of the subsection is devoted to the proof of Proposition \ref{Prop:MaxKey}. First, we will reduce to the case $ k=0 $ and then prove the proposition for this particular case.
	
	\subsubsection{Reduction to the case $ k=0 $:}
	We first show that it suffices to prove the proposition for $ k=0 $. Indeed, assuming we have the conclusion for $ k=0 $, fix now some $ k>0 $ and a collection $ \T \subseteq \mathcal{T}_{\ann} $ for some fixed $ \ann \in 3^{\Z} $. By using standard shifted grid techniques, see for example \cite{LernerNazarovDyadic}, we can find triadic grids $ \{ \mathcal{L}_{j} \}_{j=1}^{2^n} $, $ \{ \mathcal{I}_{j} \}_{j=1}^{2} $ such that for each $ t \in \T $ there exist $ \widetilde{L}_t \in \bigcup_{j=1}^{2^d} \mathcal{L}_{j} $ and $ \widetilde{I}_t \in \mathcal{I}_{1} \cup \mathcal{I}_{2} $ satisfying $ 3^k L_t \subseteq \widetilde{L}_t \subseteq 3^{k+2} L_t $ and $ 3^k I_t \subseteq \widetilde{I}_t \subseteq 3^{k+2} I_t $ with $ \ell(\widetilde{L}_t) = 3^{k+1} \ell(L_t) $ and $ \ell(\widetilde{I}_t) = 3^{k+1} \ell(I_t) $. Then, taking $ \widetilde{R}_t = R(\widetilde{L}_t,\widetilde{I}_t,v_{t}) $ we have
	\begin{equation}
		\label{Eq:MaxTriadicInclusion}
		3^k R_t \subseteq \widetilde{R}_t \subseteq 3^{k+2} R_t,
	\end{equation} 
	since $ 3^k R_t = R(3^k L_t, 3^k I_t, v_t) $. Now, recall that $ Q_t \in \mathcal{G} $ is the unique triadic cube with $ (Q_t \times \{1\})' \in (\mathcal{G} \times \{1\})' $ and such that $ (c(Q_t) \times \{1\})' = v_t $ and $ \ell(Q_t) = \ell(I_t) / \ell(L_t) $. Then, we have that $ \widetilde{Q}_t = Q_t $, where $ \widetilde{Q}_t \in \mathcal{G} $ is the unique triadic cube with $ (\widetilde{Q}_t \times \{1\})' \in (\mathcal{G} \times \{1\})' $ and such that $ (c(\widetilde{Q}_t) \times \{1\})' = \tilde{v}_t \coloneqq v_t $ and $ \ell(\widetilde{Q}_t) = \ell(\widetilde{I}_t) / \ell(\widetilde{L}_t) = \ell(Q_t) $. Thus, $ \ecc(R_t) = \ecc(\widetilde{R}_t) $. Given $ t $ and $ \widetilde{R}_t $ as above we can define
	\begin{equation}
		\widetilde{\omega}_t \coloneqq \left\{ \xi \in \R^n :\, \frac{1}{2} \abs{ \widetilde{I}_t }^{-1} < \abs{\xi} < 6 \abs{ \widetilde{I}_t }^{-1}, \, \xi' \in \ecc(\widetilde{R}_t) \right\}
	\end{equation}
	and by splitting $ \T $ into $ O_n(1) $ subcollections we can assume that 
	\begin{equation}
		\widetilde{\T} 
		\coloneqq \left\{ \tilde{t} \coloneqq \widetilde{R}_t \times \widetilde{\omega}_t :\, t \in \T \right\} 
		\subset \mathcal{T}_{\widetilde{\ann}}, 
		\qquad \widetilde{\ann} = 3^{-k-1} \ann,
	\end{equation}
	are tiles with $ 3^k R_t \subseteq R_{\tilde{t}} = \widetilde{R}_t \subseteq 3^{k+2} R_t $ and $ \ecc(\tilde{t}) = \ecc( \widetilde{R}_t ) = \ecc(R_t) = \ecc(t) $.
	
	Now apply Lemma \ref{Lem:DecomAlgorithm} with $ \mu = k+2 $ so that
	\begin{equation}
		\sum_{t\in\T} \abs{ R_t } 
		= \sum_{ j=1 }^{N} \sum_{ t\in\T_{\mu}(P_j) } \abs{ R_t } 
		= \sum_{ j=1 }^{N} \Abs{ \bigcup_{ t\in\T_{\mu}(P_j) } R_{t}}
	\end{equation}
	by (ii) of Lemma \ref{Lem:DecomAlgorithm} since $ \T $ consists of pairwise incomparable tiles and by (iii) of the same Lemma \ref{Lem:DecomAlgorithm}
	\begin{equation}
		\sum_{t\in\T} \Abs{ R_t } 
		\lesssim \sum_{j=1}^{N} \Abs{ K_n 3^k R_{P_j} } 
		\lesssim \sum_{j=1}^{N} \Abs{ R_{\widetilde{P}_j} }.
	\end{equation}
	By construction $ \{ P_j \}_{j=1}^{N} \subseteq \T $. Clearly each $ \{ \widetilde{P}_j \}_{j=1}^{N} $ satisfies the assumptions \eqref{Eq:MaxKeyHypo1} and \eqref{Eq:MaxKeyHypo2} for $ k=0 $, for some $ \widetilde{\mu} \simeq \mu $ and $ \widetilde{\lambda} \simeq \lambda $ with implicit constants depending only on dimension. Furthermore, the collection $ \{ \widetilde{P}_j \}_{j=1}^{N} $ is pairwise incomparable. Indeed, if $ \ecc(\widetilde{P}_1) \cap \ecc(\widetilde{P}_2) \neq \emptyset $ for some $ \widetilde{P}_1, \widetilde{P_2} \in \{ \widetilde{P}_j \}_{j=1}^{N} $ then, since $ \ecc(\widetilde{P}_j) = \ecc(P_j) $ for $ j\in\{1,2\} $ this implies that $ \ecc(P_1) \cap \ecc(P_2) \neq \emptyset $ and so $ 3^{k+2} R_{P_1} \cap 3^{k+2} R_{P_2} = \emptyset $ by (iv) of Lemma \ref{Lem:DecomAlgorithm}. By \eqref{Eq:MaxTriadicInclusion}, $ R_{\widetilde{P}_j} \subseteq 3^{k+2} R_{P_j} $ for every $ j \in \{1,\ldots,N\} $, so this shows that the collection $ \{ \widetilde{P}_j \}_{j=1}^{N} $ consists of pairwise incomparable tiles. Thus, using Proposition \ref{Prop:MaxKey} for $ \{ \widetilde{P}_j \}_{j=1}^{N} $ and $ k=0 $ we get
	\begin{equation}
		\sum_{ t \in \T } \Abs{ R_t } \lesssim \sum_{j=1}^{N} \Abs{ R_{\widetilde{P}_j}} \lesssim \frac{ \abs{F} }{ \widetilde{\mu} \widetilde{\lambda}^{1+\varepsilon} } \simeq \frac{ \abs{F} }{ \mu \lambda^{1+\varepsilon} }
	\end{equation}  
	as desired. 
	
	\subsubsection{Proof of Proposition \ref{Prop:MaxKey} for $ k=0 $:}
	It remains to proof the proposition for the case $ k=0 $. Let $ \varepsilon > 0 $ and $ \T \subseteq \mathcal{T}_{\ann} $ be a finite collection of pairwise incomparable tiles satisfying the hypotheses \eqref{Eq:MaxKeyHypo1} and \eqref{Eq:MaxKeyHypo2}. The desired estimate will be a consequence of the second hypothesis \eqref{Eq:MaxKeyHypo2} and H\"older's inequality, together with the estimate
	\begin{equation}
		\label{Eq:MaxKeyLemDesired2}
		\int \bigg( \sum_{ t \in \T }  \indf{ R_t } \bigg)^{\frac{1+\varepsilon}{\varepsilon}} \lesssim_{\varepsilon} \frac{ 1 }{ \mu^{ 1/\varepsilon } } \sum_{ t \in \T } \Abs{ R_t }
	\end{equation}
	which we show below. Firstly, we can assume that $ 1/\varepsilon $ is an integer by taking $ \varepsilon $ smaller since the desired estimate becomes trivial when $\varepsilon\to \infty$. Then, expanding the power
	\begin{equation}
		\label{Eq:MaxKeyDefF}
		\begin{split}
			\int \bigg( \sum_{ t \in \T } \indf{ R_t } \bigg)^{ \frac{1+\varepsilon}{\varepsilon} } 
			& \leq \left( \frac{1+\varepsilon}{\varepsilon} \right)! \int \sum_{ t_0 \in \T } \sum_{ \substack{ t_1,\ldots,t_{\varepsilon^{-1}} \in \T \\ L_{t_1} \subseteq \cdots \subseteq L_{ t_{ \varepsilon^{-1} } } \subseteq L_{t_0} } } \indf{ R_{t_0} } \indf{ R_{t_1} } \cdots \indf{ R_{ t_{ \varepsilon^{-1} } } } \\
			& \lesssim_{\varepsilon} \sum_{ t_0 \in \T }  \int_{ R_{t_0} } \bigg( \sum_{ t \in \T, \, L_{t} \subseteq L_{t_0} } \indf{ R_{t} } \bigg)^{ 1/\varepsilon } 
			\eqcolon \sum_{ t_0 \in \T }  \int_{ R_{t_0} } f_{t_0}(x)^{ 1/\varepsilon },
		\end{split}
	\end{equation}
	where in the first line we have used symmetry and the grid property of the triadic cubes $ L_t $. Then, \eqref{Eq:MaxKeyLemDesired2} reduces to showing
	\begin{equation}
		\label{Eq:MaxKeyEstF}
		\int_{ R_{t_{0}} } f_{t_{0}} (x)^{ 1/\varepsilon } \lesssim_{\varepsilon} \frac{ 1 }{ \mu^{ 1/\varepsilon } } \Abs{ R_{t_0} }, \qquad t_0 \in \T.
	\end{equation}
	
	Our goal now is to prove \eqref{Eq:MaxKeyEstF}. The main idea of the proof is to count the number of pairwise incomparable tiles that can exist such that their space component contains a point $x$ and satisfies the first hypothesis \eqref{Eq:MaxKeyHypo1} of the lemma. We do that by defining the permissible set of tiles in Definition \ref{Def:Permissible}. Fix a tile $ t_0 \in \T $. Observe that by definition,
	\begin{equation}
		\label{Eq:MaxKeyEstFProofCardP}
		f_{t_0} ( x ) = \# \mathcal{P}_{ L_{t_0} } (x), 
		\qquad 
		\mathcal{P}_{ L_{t_0} } (x) \coloneqq \left\{ t \in \T :\, x \in R_t, \, L_{t} \subseteq L_{t_0} \right\}.
	\end{equation}
	On the other hand, if $ x \in R_{t} $ then $ \underline{x} \coloneqq (x_1,\ldots,x_d) \in L_{t} $. Thus, to estimate $ \# \mathcal{P}_{ L_{t_0} } (x) $ we will count the number of pairwise incomparable tiles $ t \in \mathcal{T}_{\ann} $ that we can construct satisfying the first hypothesis \eqref{Eq:MaxKeyHypo1} and such that $ \underline{x} \in L_{t} \subseteq L_{t_0} $. To do that, we need some notation.
	
	Given a tile $ p = R_p \times \omega_p \in \mathcal{T}_{\ann} $, recall the notation $ R_{p} = R(L_{p}, I_{p},v_{p}) $ in \eqref{Eq:DefParallelepiped} and define the projected tile $ p_{\pi} = L_{p} \times \omega_{p} $. In general, given a collection of tiles $ \P $ we will write $ \P_{\pi} \coloneqq \{ p_{\pi} :\, p \in \P \} $. Just as with the tiles $ p $, we use the following relation for the projected tiles:
	\begin{equation}
		p_{\pi} \leq p_{\pi}' \qquad \overset{\mathrm{def}}{\Longleftrightarrow} \qquad \ecc(p') \subseteq \ecc(p)  \quad \text{and} \quad L_{p} \cap L_{p'} \neq \emptyset.
	\end{equation}
	As with the standard tiles, $ \ecc(p) \subseteq \mathbb{S}^d $ is the projection of $ \omega_{p} $ into the sphere $ \mathbb{S}^d $.	Next, we need the following definition introduced in \cite[(3.20)]{Bateman09Proc}. The idea is to define all the projected tiles that can appear in $ ( \mathcal{P}_{ L_{t_0} } (x) )_{\pi} $ but removing the ones that would lead to comparable tiles. A key fact in achieving this is that since $ \sigma $ depends only on the first $ d $ coordinates, then
		\begin{equation}
			\Abs{ v_{ \sigma }^{-1}(\alpha_{t,\tau_0}) \cap R_{t} } \geq \mu \Abs{ R_{t} } \qquad \Longleftrightarrow \qquad \Abs{ v_{ \sigma }^{-1}(\alpha_{t,\tau_0}) \cap L_{t} } \geq \mu \Abs{ L_{t} }.
		\end{equation}

	\begin{definition}
		\label{Def:Permissible}
		Let $ t_0 \in \T $. First, define the \emph{permissible set of tiles} of the triadic cube $ L_{t_0} $ by 
		\begin{equation}
			B \left( L_{t_0} \right) 
			\coloneqq \left\{ p_{\pi} \in ( \mathcal{T}_{\ann} )_{\pi} :\, L_{p} = L_{t_0} , \, \Abs{ v_{ \sigma }^{-1}(\alpha_{p,\tau_0}) \cap L_{p} } \geq \mu \Abs{ L_{p} } \right\}.
		\end{equation}
		Next, let $ L \subseteq L_{t_0} $ be triadic and suppose that we have defined $ B( L' ) $ for every $ L' $ triadic, with $ L \subsetneq L' \subseteq L_{t_0} $. Then, we define the \emph{permissible set of tiles} of the cube $ L $ by 
		\begin{equation}
			\label{Eq:PermissibleDef}
			B( L ) 
			\coloneqq \left\{ p_{\pi} \in ( \mathcal{T}_{\ann} )_{\pi} :\, L_{p} = L, \, \frac{ \Abs{ v_{ \sigma }^{-1}(\alpha_{p,\tau_0}) \cap L_{p} } }{\abs{ L_{p} }} \geq \mu , \, p_{\pi} \not\leq p_{\pi}' \text{ for any } p_{\pi}' \in B(L') , \, L' \supseteq L \right\}.
		\end{equation}
	\end{definition}
	
	We now claim that with $ B(L) $ defined in \eqref{Eq:PermissibleDef},
	\begin{equation}
		\label{Eq:MaxKeyDefG}
		f_{t_0}(x) \leq \sum_{ L \subseteq L_{t_0} } \indf{ L }( \underline{ x } ) \# B ( L ) \eqqcolon g_{t_0} ( \underline{x} ), \qquad x = (\underline{x},x_n) \in \R^n .
	\end{equation}
	Indeed, \eqref{Eq:MaxKeyDefG} follows from the observation that 
	\begin{equation}
		g_{t_0}(\underline{x}) = \# \mathcal{Q}_{ L_{t_0} }(\underline{x}), 
		\qquad 
		\mathcal{Q}_{ L_{t_0} }(\underline{x}) \coloneqq \left\{ p_{\pi} \in (\mathcal{T}_{\ann})_{\pi} : \,\underline{x} \in L_{p} \subseteq L_{t_0} , \, p_{\pi} \in B(L_{p}) \right\},
	\end{equation}
	and that there are more elements in $ \mathcal{Q}_{ L_{t_0} }(\underline{x}) $ than in $ \mathcal{P}_{ L_{t_0} } (x) $, since for each $ t \in \mathcal{P}_{ L_{t_0} } (x) $ we can associate a different $ p_{\pi} \in \mathcal{Q}_{ L_{t_0} } (x) $ such that $ L_{t} \subseteq L_{p} $ and $ \omega_{p} \subseteq \omega_{t} $. Next, we bound $ g_{t_0} $ by
	\begin{equation}
		\label{Eq:MaxKeyDefH}
		g_{t_0}( \underline{x} ) \leq \frac{1}{\mu} \sum_{ L \subseteq L_{t_0} } \sum_{ p_{\pi} \in B( L ) } \indf{ L }( \underline{ x } ) \frac{ \Abs{ v_{ \sigma }^{-1}(\alpha_{p,\tau_0}) \cap L } }{ \abs{ L } } \eqqcolon \frac{1}{\mu} h_{ t_0 } ( \underline{x} ), \qquad \underline{x} \in \R^d,
	\end{equation}
	where the inequality is satisfied by the definition of $ g_{t_0} $ and the condition $ \abs{ v_{ \sigma }^{-1}(\alpha_{p,\tau_0}) \cap L_{p} } \geq \mu \abs{ L_{p} } $, for $ p_{\pi} \in B(L) $. 
	
	Next, we show that $ h_{t_0} $ is a $ \BMO_{ \triadic } $-function; $ \BMO_{ \triadic } $ is just the standard $ \BMO $ space, but considering only triadic cubes.

	\begin{lemma}
		\label{Lem:MaxKeyHBMO}
		There holds $ \norm{ h_{ t_0 } }{ \BMO_{ \triadic }( L_{t_0} ) } \leq 1 $.
	\end{lemma}
	
	\begin{proof}
		It is enough to find, for each triadic cube $ L \subseteq L_{t_0} $, a number $ b_{ L } $ such that
		\begin{equation}
			\label{Eq:MaxKeyHBMODesired}
			\frac{ 1 }{ \abs{ L } } \int_{ L } \Abs{ h_{ t_0 }( \underline{x} ) - b_{ L } } d \underline{x} \leq 1.
		\end{equation}
		First, for every triadic cube $ L $, define the number $ \mu_{ L } \coloneqq \sum_{ p_{\pi} \in B( L ) } \abs{ v_{ \sigma }^{-1}(\alpha_{p,\tau_0}) \cap L_{p} } $. Note that by the definition of $ B(L) $, the collection $ \{ v_{ \sigma }^{-1}(\alpha_{p,\tau_0}) \cap L_{p} :\, p_{\pi} \in B(L') , \, L' \subseteq L \} $ is pairwise disjoint. Thus, $ \mu_{ L } $ is a Carleson sequence, that is,
		\begin{equation}
			\begin{split}
				\sum_{ L' \subseteq L } \mu_{ L' } = \sum_{ L' \subseteq L } \sum_{ p_{\pi} \in B( L ) } \Abs{ v_{ \sigma }^{-1}(\alpha_{p,\tau_0}) \cap L_{p} } \leq \abs{ L }.
			\end{split}
		\end{equation}
		Now, for every triadic cube $ L \subseteq L_{t_0} $, we choose the number
		\begin{equation}
			b_{ L } \coloneqq \sum_{ L' \supseteq L } \frac{ \mu_{ L' } }{ \abs{ L' } }.
		\end{equation}
		Thus, using the grid property of the triadic cubes,
		\begin{equation}
			\begin{split}
				\frac{ 1 }{ \abs{ L } } \int_{ L } \Abs{ h_{ t_0 } ( \underline{ x } ) - b_{ L } } d \underline{x}
				& = \frac{ 1 }{ \abs{ L } } \int_{ L } \sum_{ \substack{ \tilde{L} \subseteq L_{t_0}, \, \tilde{L} \supseteq L } } \frac{ \mu_{ \tilde{L} } }{ \abs{ \tilde{L} } } + \sum_{ \substack{ \tilde{L} \subseteq L_{t_0} , \, \tilde{L} \subsetneq L } } \indf{ \tilde{L} }( \underline{ x } ) \frac{ \mu_{ \tilde{L} } }{ \abs{ \tilde{L} } } - \sum_{ L' \supseteq L } \frac{ \mu_{ L' } }{ \abs{ L' } } d \underline{x} \\ 
				& = \frac{ 1 }{ \abs{ L } } \int_{ L } \sum_{ L' \subsetneq L } \indf{ L' }( \underline{ x } ) \frac{ \mu_{ L' } }{ \abs{ L' } } d \underline{x} = \frac{ 1 }{ \abs{ L } } \sum_{ L' \subsetneq L } \mu_{ L' } \leq 1,
			\end{split}
		\end{equation}
		where we have used that $ \mu_{ L } $ is a Carleson sequence.
	\end{proof}
	
	Finally, we are ready to prove the desired inequality \eqref{Eq:MaxKeyEstF}. We just need to apply \eqref{Eq:MaxKeyDefG} and \eqref{Eq:MaxKeyDefH}, together with the John-Nirenberg inequality, since $ h_{ t_0 } \in \BMO_{ \triadic }( L_{t_0} ) $, to get 
	\begin{equation}
		\begin{split}
			\int_{ R_{t_0}} f_{ t_0 } (x)^{ 1/\varepsilon } dx 
			& \leq \int_{ R_{t_0}} \left( \frac{ 1 }{ \mu } h_{ t_0 }( \underline{x} ) \right)^{ 1/\varepsilon } dx 
			= \frac{ \abs{ I_{t_0} } }{ \mu^{ 1/\varepsilon } } \Norm{ h_{ t_0 } }{ L^{ 1/\varepsilon }( L_{t_0} ) }^{ 1/\varepsilon }
			\lesssim \frac{ \abs{ I_{t_0} } }{ \mu^{ 1/\varepsilon } } \Abs{ L_{t_0} }
			= \frac{ \abs{ R_{t_0} } }{ \mu^{ 1/\varepsilon } }.
		\end{split}
	\end{equation}

	\subsection{Big average with respect to size}\label{Ss:BigSizeBigInter} 
	In order to apply Proposition \ref{Prop:MaxKey} we will need three types of assumptions. The first is that we will have to apply it to a pairwise incomparable collection of tiles. We will reduce to that case in Subsection \ref{Ss:MaxProof} by using the properties of the collection $ \T_{\siz}^{\tops} $, which contains the tops of the trees constructed in Proposition \ref{Prop:CombinedDensitySize}, together with suitable applications of Lemma \ref{Lem:DecomAlgorithm}. On the other hand, estimate \eqref{Eq:MaxKeyHypo1} will be a consequence of the large density of the trees in Proposition \ref{Prop:CombinedDensitySize}. Here we devote our attention to the tools that will eventually allow us to verify \eqref{Eq:MaxKeyHypo2}. The first such tool gives big average of the set $ F $ with $ \abs{f} \leq \indf{F} $ used in $ \size $, with respect to dilated tops of lacunary trees of big size. A two-dimensional version was proved in \cite[Lemma 6.3]{Bateman2013Revista}, while the idea goes back to \cite[Lemma 4.55]{LaceyLiMemoirs}.
	
	\begin{lemma}
		\label{Lem:MaxBigSizeBigInter}
		Let $ \mathbf{T} \subseteq \mathcal{T}_{\ann} $ for fixed $ \ann \in 3^{\Z} $ be a lacunary tree satisfying
		\begin{equation}
			\sum_{t\in\mathbf{T}} F_{10n}[f](t)^2 \gtrsim \sigma^2 \abs{ R_{\mathbf{T}} }. 
		\end{equation}
		Then, for every $ \varepsilon>0 $ there holds
		\begin{equation}
			\Abs{ F \cap (1+\sigma^{-\varepsilon}) R_{\mathbf{T}} } \gtrsim_{\varepsilon} \sigma^{1+\varepsilon} \Abs{ (1+\sigma^{-\varepsilon}) R_{\mathbf{T}} }.
		\end{equation}
	\end{lemma}

	\subsubsection{Proof of Lemma \ref{Lem:MaxBigSizeBigInter}} 
	For the proof of Lemma~\ref{Lem:MaxBigSizeBigInter}, first we will first need some auxiliary results which are best stated in terms of a Littlewood-Paley type square function which we describe below. 

	Let $M>N\geq 8n$ and define the square function
	\begin{equation}\label{Eq:BigSizeBigInterDeltaDef}
		f \quad 
		\longmapsto \quad
		\Delta_{\mathbf T,N} f (x) \coloneqq \left( \sum_{ t \in \mathbf{T} }  F_N[f](t)^2  \frac{\indf{ R_t }(x)}{ \Abs{ R_t } } \right)^{ 1/2 }, \qquad x \in \R^n.
	\end{equation}
	Clearly, for some collection of adapted functions $ \{ \varphi_t \in \mathcal{F}_t^{N} :\, t \in \mathbf{T} \}$
	\begin{equation}
		\label{Eq:BigSizeBigInterDeltaL2}
		\norm{ \Delta_{\mathbf T,N} f }{ L^2( \R^n ) }^2 = \sum_{ t \in \mathbf{T} } \Abs{ \left\langle f ,\varphi_t \right\rangle }^2 \lesssim \norm{f}{L^2(\R^n)}^2,
	\end{equation}
	by the first orthogonality estimate of Lemma \ref{Lem:OrthoLacTree}. In the following Lemma \ref{Lem:BigSizeBigInterSquareFuncLp} we will prove some $ L^p $-boundedness properties of $ \Delta_{\mathbf T,N} $ which will be used to prove Lemma \ref{Lem:MaxBigSizeBigInter}. In particular, the following lemma is a higher dimensional version of \cite[(4.85)]{LaceyLiMemoirs}.

	\begin{lemma}
		\label{Lem:BigSizeBigInterSquareFuncLp} For positive integers $\,2n\leq K\leq N/4$ there holds
		\begin{equation}
			\Norm{ \Delta_{\mathbf T,N} f }{ L^{ p }( \R^n ) } \lesssim  \norm{ f \chi_{ R_{\mathbf{T}} , \infty }^{2K} }{ L^{ p }( \R^n ) }, \qquad 1 < p \leq 2,
		\end{equation}
		where the implicit constant depends on $p$, $ K $ and the dimension, but not on $ \mathbf{T} $.
	\end{lemma}
	
	We have stated the above Lemma \ref{Lem:BigSizeBigInterSquareFuncLp} for the range $ 1 < p \leq 2 $ since it is what we need. However the same estimate holds for $ 2 < p < \infty $ by a $ \BMO $-type estimate.
	
	\begin{proof}  
		Let us fix a lacunary tree $ \mathbf{T} \subseteq \mathcal{T}_{\ann} $ and a positive integer $N\geq 8n$. By rotational invariance we can assume that $ \xi_{\mathbf{T} } = e_n $. By the geometric Lemma \ref{Lem:Geom} and using $ O_n(1) $ shifted grids we can assure that for every $ t\in\mathbf{T} $ there holds
		\begin{equation}
			R_{t} = R(L_t,I_t,v_t) \subseteq R(\widetilde{L}_t,\widetilde{I}_t,e_n) \eqqcolon R_{\tilde{t}}
		\end{equation}
		with $ \{ \widetilde{I}_t \}_{t\in\mathbf{T}} \subseteq \bigcup_{j=1}^{2} \mathcal{I}_j $ and $ \{ \widetilde{L}_t \}_{t\in\mathbf{T}} \subseteq \bigcup_{j=1}^{2^d} \mathcal{L}_j $, each $ \mathcal{I}_j $ and $ \mathcal{L}_j $ being a triadic grid and $ I_t \subseteq \widetilde{I}_t = I_{\tilde{t}} \subseteq c_n I_t $, $ L_t \subseteq \widetilde{L}_t = L_{\tilde{t}} \subseteq c_n L_t $. Thus, $ R_t \subseteq R_{\tilde{t}} \subseteq c_n R_t $. We use the same triadic grid $ \mathcal{G} $ for the eccentricities of $ \tilde{t} $, defining $ Q_{\tilde{t}} $ to be the unique triadic cube in $ \mathcal{G} $ with center $ 0 $ and sidelength $ \ell(I_{\tilde{t}})/\ell(L_{\tilde{t}}) = \ell(I_{t})/\ell(L_{t}) $. This ensures that $ \ecc(R_t) = \ecc(R_{\tilde{t}}) $. That way we see that there exists a universal dimensional constant $ \gamma_n $ such that for each $ \varphi_t \in \mathcal{F}_t^{N} $ the wave packet $ \varphi_{\tilde{t}} \coloneqq \gamma_n \varphi_t  \in \mathcal{F}_{\tilde{t}}^{N} $ for every $ t \in \mathbf{T} $.  We now rename $ \widetilde{\mathbf{T}} $ again to $ \mathbf{T} \subseteq \mathcal{T}_{\ann} $ for some $ \ann \in 3^{\Z} $ and proceed with the assumption that all the space components $\{R_t:\, t\in\mathbf T\}$ have sides parallel to the coordinate axes. Furthermore, we claim that it is enough to show the weaker estimate
		\begin{equation}\label{Eq:BigSizeBigInterSquareFuncLpLemWeak}
			\Norm{ \Delta_{\mathbf T,N} f }{ L^{ p }( \R^n ) } 
			\lesssim \Norm{ f }{ L^{ p }( \R^n ) }, 
			\qquad 1 < p \leq 2,\qquad N\geq 4n.
		\end{equation}
		This is because $ \chi_{ R_{\mathbf{T}} , \infty }^{-2K} \varphi_t \in\mathcal F_t ^{N/2} $ is still adapted to the tile $ t $ and we can write for some choice of wave packets $ \{ \varphi_t \in \mathcal{F}_t^{N} :\, t \in \mathbf{T} \}$
		\begin{equation}
			\Delta_{\mathbf T,N} f 
			= \left( \sum_{ t \in \mathbf{T} } \Abs{ \left\langle f ,\varphi_t \right\rangle }^2 \frac{\indf{ R_t }}{ \Abs{ R_t } } \right)^{ 1/2 } 
			= \left( \sum_{ t \in \mathbf{T} } \Abs{ \left\langle f \chi_{ R_{\mathbf{T}} , \infty }^{2K} , \chi_{ R_{\mathbf{T}} , \infty }^{-2K} \varphi_t \right\rangle }^2 \frac{\indf{ R_t }}{ \Abs{ R_t } } \right)^{ 1/2 } 
			\leq  \Delta_{\mathbf T,N/2} \left( f \chi_{ R_{\mathbf{T}} , \infty }^{2K} \right),
		\end{equation}
		so that the result of the lemma will be proved by applying \eqref{Eq:BigSizeBigInterSquareFuncLpLemWeak} to $ f \chi_{ R_{\mathbf{T}} , \infty }^{2K} $.
		
		With these reductions in mind, we now write $\Delta f$ in place of $\Delta_{\mathbf T,N} f$ in order to simplify the notation. Then, our goal is to prove \eqref{Eq:BigSizeBigInterSquareFuncLpLemWeak}. First, the case $ p=2 $ is \eqref{Eq:BigSizeBigInterDeltaL2}. Then, we get the range $ 1 < p < 2 $ by interpolating the $ L^2 $-estimate with the weak $ L^1(\R^n) $-estimate
		\begin{equation}\label{Eq:BigSizeBigInterSquareFuncLpWeakDesired}
			\Abs{ \left\{ \Delta f > \lambda \right\} } \lesssim \frac{ \norm{f}{L^1(\R^n)} }{ \lambda }, \qquad \lambda>0.
		\end{equation}
		
		To prove \eqref{Eq:BigSizeBigInterSquareFuncLpWeakDesired}, we use a Calder\'on-Zygmund type decomposition for the function $ f $. Fix $ \lambda > 0 $. Denote by $ \mathcal{E} $ the maximal elements of the collection of rectangles $ R = R(L,I,v_{\mathbf{T}}) $, with $ L \subseteq \R^d $ a triadic cube and $ I \subseteq \R $ an interval of length $ \ann^{-1} $, and such that
		\begin{equation}
			\int_{\R^n} \abs{ f(x) } \chi_{ R , 1 }^{10n} (x) dx > \lambda.
		\end{equation}
		We consider the stopping condition with Schwartz tails as above, anticipating the application of Bernstein's inequality below. We define
		\begin{equation}
			b_{R}(x) 
			\coloneqq \left[ f(x) - \avgint_{L_R} f(\underline{y}, x_n) d \underline{y} \right] \indf{R}(x), \qquad R \in \mathcal{E};
		\end{equation}
		\begin{equation}
			g(x) 
			\coloneqq f(x) - \sum_{ R \in \mathcal{E} } b_{R}(x)
			= f(x) \indf{ ( \bigcup_{R \in \mathcal{E}} R )^c } (x) + \sum_{ R \in \mathcal{E} } \avgint_{L_R} f(\underline{y}, x_n) d \underline{y} \indf{R}(x) \eqcolon \text{I} + \text{II}.
		\end{equation}
		Then, to prove \eqref{Eq:BigSizeBigInterSquareFuncLpWeakDesired}, it is enough to show
		\begin{equation}
			\label{Eq:BigSizeBigInterSquareFuncLpWeakDesiredGood}
			\Abs{ \left\{ \Delta g > \lambda \right\} } \lesssim \frac{ \norm{f}{L^1(\R^n)} }{ \lambda },
			\qquad
			\Abs{ \left\{ \Delta \left( \sum_{R \in \mathcal{E}} b_{R} \right) > \lambda \right\} } \lesssim \frac{ \norm{f}{L^1(\R^n)} }{ \lambda }.
		\end{equation}
		
		The first inequality in \eqref{Eq:BigSizeBigInterSquareFuncLpWeakDesiredGood} is a consequence of the weak $ L^2 $-estimate together with $ g \lesssim \lambda $ almost everywhere. To see the pointwise inequality, observe that, since the adapted functions $ \varphi_t $ have frequency support contained in $ B(0,6 \ann) $, we can assume that $ \supp \widehat{f} \subseteq B(0,6 \ann ) $. Then, using a version of Bernstein's inequality from \cite[Lemma 2.6]{DLMV2022}, we have
		\begin{equation}
			\label{Eq:BigSizeBigInterBernstein}
			\Norm{ \chi_{ I,\infty}^{10n} f(\underline{x},\cdot) }{ L^{\infty}(\R) } 
			\lesssim \ann^{1/p} \Norm{ \chi_{ I, \infty}^{10n} f(\underline{x},\cdot) }{ L^p(\R) }, 
			\qquad p \geq 1,
		\end{equation}
		uniformly in $ \underline{x} \in \R^d $and for intervals $ I \subseteq \R^n $ with $ \ell(I) = \ann^{-1} $. Now, for $ \text{II} $, suppose that $ x \in R $ for some $ R \in \mathcal{E} $. Then,
		\begin{equation}
			\begin{split}
				g(x) 
				& \lesssim \int_{ \R^d } \abs{ f(\underline{y},x_n) } \chi_{ L_{R},1 }^{10n}(\underline{y}) d\underline{y}
				\leq \int_{ \R^d } \norm{ \chi_{ I_{R},\infty }^{10n} f(\underline{y},\cdot) }{L^{\infty}(\R)} \chi_{ L_{R},1 }^{10n}(\underline{y}) d\underline{y} \\ 
				& \lesssim \int_{ \R^d } \int_{ \R } \abs{ f(\underline{y},t) } \chi_{ I_{R},1}^{10n} (t) dt \chi_{ L_{R},1 }^{10n} (\underline{y}) d\underline{y}
				\lesssim \int_{ \R^n } \abs{ f(y) } \chi_{ R( 3 L_{R}, I_{R}, v_{\mathbf{T}} ), 1}^{10n} (y) dy 
				\leq \lambda,
			\end{split}
		\end{equation}
		where we have used Bernstein's inequality \eqref{Eq:BigSizeBigInterBernstein} with $ p=1 $ in the second line and the maximality of $ \mathcal{E} $ in the last inequality. This shows the pointwise inequality for $ \text{II} $. Similarly, if $ x \not\in \bigcup_{ R \in \mathcal{E} } R $ then for every rectangle $ R' \ni x $ we have
		\begin{equation}
			\begin{split}
				\avgint_{L_{R'}} \abs{ f(\underline{y},x_n) } d\underline{y} 
				& \lesssim \int_{ \R^d } \int_{ \R } \abs{ f(\underline{y},t) } \chi_{ I_{R'},1 }^{10n} (t) dt \chi_{ L_{R'},1 }^{10n}(\underline{y}) d\underline{y} 
				= \int_{ \R^n } \abs{ f(y) } \chi_{ R',1}^{10n} (y) dy
				\leq \lambda,
			\end{split}
		\end{equation}
		so $ g(x) = f(x) \lesssim \lambda $ for $ \text{I} $ as well.
		
		Next, for the proof of the second inequality in  \eqref{Eq:BigSizeBigInterSquareFuncLpWeakDesiredGood},  first of all observe that by the definition of $ \mathcal{E} $ we have that
		\begin{equation}
			\label{Eq:BigSizeBigInterSquareUnionEstimate}
			\Abs{ \bigcup_{ R \in \mathcal{ E } } R } \leq \Abs{ \left\{ M_{ \mathcal{R} } f (x) \gtrsim \lambda \right\} } \lesssim \frac{ \norm{ f }{ 1 } }{ \lambda },
		\end{equation}
		where $ M_{\mathcal{R}} $ denotes the maximal function over rectangles with sides parallel
			to the coordinate axes and with $ \ell(I_R) = \ann^{-1} $. Then, we claim that it is enough to show
		\begin{equation}
			\label{Eq:BigSizeBigInterSquareFuncLpBad}
			\int_{ \left( K_n R \right)^c } \Abs{ \Delta \left( b_{R}(x) \right) } dx \lesssim \norm{ b_{R} }{L^1(\R^n)}, \qquad R \in \mathcal{E}. 
		\end{equation}
		Indeed, assuming \eqref{Eq:BigSizeBigInterSquareFuncLpBad}, the second inequality in \eqref{Eq:BigSizeBigInterSquareFuncLpWeakDesiredGood} follows from \eqref{Eq:BigSizeBigInterSquareUnionEstimate} for the points $ x \in K_n R $ for some $ R \in \mathcal{ E } $; combined with Chebyshev's inequality, which gives
		\begin{equation}
			\begin{split}
				\Abs{ \left\{ x \in \left( \bigcup_{ R \in \mathcal{E} } K_n R \right)^c :\, \Delta \left( \sum_{R \in \mathcal{E} } b_{R} \right) > \lambda \right\} } 
				& \leq\frac{1}{\lambda} \sum_{R \in \mathcal{E}} \int_{ \left( K_n R \right)^c } \Abs{ \Delta \left( b_{R}(x) \right) } dx,
			\end{split}
		\end{equation}
		together with
		\begin{equation}
			\begin{split}
				\norm{ b_{R} }{L^1(\R^n)} 
				& \simeq \int_{ R } \Abs{ f(x) } dx 
				\lesssim \abs{ R } \int_{ R } \Abs{ f(x) } \chi_{ R( 3 L_{R}, I_{R}, v_{\mathbf{T}} ),1 }^{10n} dx
				\lesssim \abs{ R } \lambda,
			\end{split}
		\end{equation}
		for the rest of the points. In the last calculation, we have used the maximality of the elements in $ \mathcal{E} $.

		It only remains to prove \eqref{Eq:BigSizeBigInterSquareFuncLpBad}. First, note that by the definition of $ \Delta $ we need to integrate only on $ \Gamma = K_n R_{ \mathbf{T} } \cap ( K_n R )^{c} $. Then, split the tiles in the definition of $ \Delta $ by $ \scl(t) \leq \scl(R) $ and $ \scl(t) > \scl(R) $. In the first case, we use the fast decay of the wave packets $ \varphi_t $ together with the fact that $ \abs{ x - c(R_{t}) } \gtrsim \dist(R_t,R) $ for $ x \in R $. Then, \eqref{Eq:BigSizeBigInterSquareFuncLpBad} follows by splitting the sum over scales $ \scl(t) = 3^m \scl(R) $ and using that for $ x \in R_t $ we have $ \dist(x,R) \geq \dist(R_t,R) $, together with the definition of $ \Gamma $, which gives that we are integrating on $ \{ \dist( \underline{x} , L_{R} ) \gtrsim \scl(R) \} $. The second case is similar but one needs to use also the cancellation of $ b_{R} $ to get a better decay of the wave packets $ \varphi_t $ and also split between the cases $ \dist( R,R_t ) \leq \scl(R_t) $ and $ \dist( R,R_t ) \geq \scl(R_t) $.
	\end{proof}

	Next, we need the following bound for the square function $ \Delta_{\mathbf T,n} $, which is an expression of exponential square integrability of $\Delta_{\mathbf T,n}$. The proof is the same as in the two-dimensional case in  \cite[Lemma 10.3]{Bateman2013Revista} and we omit it.
	
	\begin{lemma}
		\label{Lem:BigSizeBigInterSquareFuncL2}
		For $N\geq 4n$ there holds
		\begin{equation} 
			\label{Eq:BigSizeBigInterSquareFuncL2}
			\Norm{ \Delta_{\mathbf T,N} f }{ L^{ 2 }( \R^n ) } \lesssim \frac{ 1 }{ \Abs{ R_{ \mathbf{T} } }^{ 1/2 } } \int_{ K_n R_{ \mathbf{T} } } \Delta_{\mathbf T,N} f.
		\end{equation}
	\end{lemma}
	
	With these two Lemmas \ref{Lem:BigSizeBigInterSquareFuncLp} and \ref{Lem:BigSizeBigInterSquareFuncL2}, we are ready to prove the desired Lemma \ref{Lem:MaxBigSizeBigInter}.
	
	\begin{proof}[Proof of Lemma \ref{Lem:MaxBigSizeBigInter}] 
		It is clearly enough to prove the lemma under the assumption $\sigma<1/2$, otherwise the conclusion follows trivially by the $L^2$-bound for $\Delta_{\mathbf T,N}$. Also, by hypothesis, there exists a lacunary tree $ \mathbf{T} $, a collection of adapted functions $ \{ \varphi_t \in \mathcal{F}_t^{10n} :\, t \in \mathbf{T} \} $ such that, for every large positive integer $N\geq 3n$, there holds
		\begin{equation}
			\begin{split}
			\sigma 
			& \lesssim \left( \frac{ 1 }{ \Abs{ R_{ \mathbf{T} } } } \sum_{ t \in \mathbf{T} } \Abs{ \left\langle f ,\varphi_t \right\rangle }^2 \right)^{ 1/2 } 
			\leq \frac{ \norm{ \Delta_{	\mathbf T,4N} f }{ L^2( \R^n ) }  }{ \Abs{ R_{ \mathbf{T} } }^{ 1/2 } } 
			\lesssim \frac{ 1 }{ \Abs{ R_{ \mathbf{T} } } } \int_{ K_n R_{ \mathbf{T} } } \Delta_{	\mathbf T,4N} f
			\\
			&\leq \frac{ \Abs{ K_n R_{ \mathbf{T} } }^{ \frac{ \varepsilon }{ 1+\varepsilon } } }{ \Abs{ R_{ \mathbf{T} } } } \Norm{ \Delta_{	\mathbf T,4N} f }{ L^{1+\varepsilon}(\R^n) },
			\end{split}
		\end{equation}
		where we have applied Lemma \ref{Lem:BigSizeBigInterSquareFuncL2} to bound the $ L^2 $-norm and H\"older's inequality in the last inequality. Thus, applying Lemma \ref{Lem:BigSizeBigInterSquareFuncLp} and taking the $ (1+\varepsilon) $-power,
		\begin{equation}
			\begin{split}
				\sigma^{1+\varepsilon} \Abs{ R_{ \mathbf{T} } }
				& \lesssim \int ( f \chi_{ R_{\mathbf{T}} , \infty }^{2N} )^{ 1+\varepsilon } = \int_{F} ( \chi_{ R_{\mathbf{T}} , \infty }^{2N} )^{ 1+\varepsilon },
			\end{split}
		\end{equation}
		where we have taken $ f = \indf{ F } $ in the last equality. Using the fast decay of $ \chi_{ R_{\mathbf{T}} , \infty }^{2N} $ we have
		\begin{equation}
			\begin{split}
				\int_{F} ( \chi_{ R_{\mathbf{T}} , \infty }^{2N} )^{ 1+\varepsilon } 
				& \lesssim \Abs{ F \cap \sigma^{-\varepsilon} R_{\mathbf{T}} } + \Abs{ \sigma^{ -\varepsilon} R_{\mathbf{T}}} \sum_{k=1}^{\infty} \sigma^{(2N-n) \varepsilon k} = \Abs{ F \cap \sigma^{-\varepsilon} R_{\mathbf{T}} } + \frac{\sigma^{(2N-n)\varepsilon}}{1-\sigma^{(2N-n)\varepsilon}} \Abs{ \sigma^{ -\varepsilon} R_{\mathbf{T}}}.
			\end{split}
		\end{equation}
		Combining the estimates above and taking $ M>N $ sufficiently large with respect to $\varepsilon $, we get
		\begin{equation}
			\sigma^{ 1+(n+1)\varepsilon } \Abs{ \sigma^{-\varepsilon} R_{\mathbf{T}} } 
			\lesssim \Abs{ F \cap \sigma^{ -\varepsilon} R_{\mathbf{T}}}
		\end{equation}
		which yields the desired estimate. Note that this is the point where the smoothness and decay  parameter $M$ of the wave packets has to be chosen sufficiently large depending on $\varepsilon$. This will incur a dependence of $M$ on the integrability exponent $p$ later on in the proof.
	\end{proof}

	\subsubsection{Refinement of Lemma \ref{Lem:MaxBigSizeBigInter}}
	In order to verify the second assumption of Proposition \ref{Prop:MaxKey} we need a certain technical refinement of Lemma \ref{Lem:MaxBigSizeBigInter} that we state and prove below.
	
	\begin{lemma}
		\label{Lem:MaxSizeHypo}
		Let $ \varepsilon>0 $ and $ t \in \mathcal{T}_{\ann} $ for some fixed $ \ann \in 3^{\Z} $. Suppose that $ \T_t \subseteq \mathcal{T}_{\ann} $ is a collection of pairwise incomparable tiles satisfying $ s \leq_2 t $ for every $ s \in \T_t $. We further assume that each $ s \in \T_t $ is the top of a lacunary tree $ \mathbf{T}_s $ with
		\begin{equation}
			\sum_{s'\in\mathbf{T}_s} F_{10n}[\indf{F}](s')^2 
			\gtrsim  \sigma^2 \abs{ R_{s} }. 
		\end{equation}
		Then, there holds
		\begin{equation}
			\Abs{ F \cap (1+\sigma^{-\varepsilon}) K_n^2 R_t } \gtrsim \sigma^{1+\varepsilon} \sum_{s\in\T_t} \abs{ R_{s} }.
		\end{equation}
	\end{lemma}
	
	\begin{proof}
		First of all, since $ s \leq_2 t $ for every $ s \in \T_t $ , we have $ (1+\sigma^{-\varepsilon}) K_n^2 R_t \supseteq \bigcup_{ t \in \T_s } (1+\sigma^{-\varepsilon}) R_s $. Thus,
		\begin{equation}
			\Abs{ F \cap (1+\sigma^{-\varepsilon}) K_n^2 R_t }
			\geq \Abs{ F \cap \bigcup_{s \in\T_t } (1+\sigma^{-\varepsilon}) R_s }.
		\end{equation}
		Choose a nonnegative integer $ \mu $ such that
		\begin{equation}
			3^{\mu} \geq K_n (1+\sigma^{-\varepsilon}) > 3^{\mu-1}.
		\end{equation} 
		Applying Lemma \ref{Lem:DecomAlgorithm} to the collection $ \T_t $ with parameter $ \mu $ we get
		\begin{equation}
			\bigcup_{ s\in \T_t } R_s = \bigcup_{j=1}^{N} \bigcup_{ s \in \T_{\mu}(P_{j,t}) } R_s, \qquad \T_t = \bigcup_{j=1}^{N} \T_{\mu}(P_{j,t})
		\end{equation}
		with $ \{ P_{j,t} \} \subseteq \T_t $. Thus,
		\begin{equation}
			\Abs{ F \cap \bigcup_{s \in\T_t } (1+\sigma^{-\varepsilon}) R_s }
			\geq \Abs{ F \cap \bigcup_{ j=1 }^{N} (1+\sigma^{-\varepsilon}) R_{P_{j,t}} }
			= \sum_{ j=1 }^{N} \Abs{ F \cap (1+\sigma^{-\varepsilon}) R_{P_{j,t}} }
		\end{equation}
		since $ \{ 3^{\mu} R_{P_{j,t}} \}_{ j=1 }^{N} $ is pairwise disjoint; this follows by (iv) of Lemma \ref{Lem:DecomAlgorithm} together with the assumption $ s \leq_2 t $ for every $ s \in \T_t $.
		
		Since $ P_{j,t} \in \T_t $ we know that $ P_{j,t} $ is the top of a lacunary tree $ \mathbf{T}_{P_{j,t}} $ with
		\begin{equation}
			\sum_{s'\in\mathbf{T}_{P_{j,t}}} F_{10n}[\indf{F}](s')^2 \gtrsim \sigma^2 \abs{ R_{P_{j,t}} } 
		\end{equation}
		and so Lemma \ref{Lem:MaxBigSizeBigInter} implies
		\begin{equation}
			\Abs{ F \cap (1+\sigma^{-\varepsilon}) K_n^2 R_t }
			\geq \sum_{ j=1 }^{N} \Abs{ F \cap (1+\sigma^{-\varepsilon}) R_{P_{j,t}} }
			\gtrsim \sum_{ j=1 }^{N} \sigma^{1+\varepsilon} \Abs{ (1+\sigma^{-\varepsilon}) R_{P_{j,t}} }
			\simeq \sum_{ j=1 }^{N} \sigma^{1+\varepsilon} \Abs{ 3^{\mu} K_n R_{P_{j,t}} }.
		\end{equation}
		By (iii) of Lemma \ref{Lem:DecomAlgorithm} there holds
		\begin{equation}
			3^{\mu} K_n R_{P_{j,t}} 
			\supseteq \bigcup_{ s \in \T_{\mu}(P_{j,t}) } R_s.
		\end{equation}
		Now, (ii) of Lemma \ref{Lem:DecomAlgorithm} yields that $ \{ R_s :\,  s \in \T_{\mu}(P_{j,t}) \} $ is a pairwise disjoint collection since $ \T_t $ was pairwise incomparable. It follows that
		\begin{equation}
			\Abs{ F \cap (1+\sigma^{-\varepsilon}) K_n^2 R_t }
			\gtrsim \sigma^{1+\varepsilon} \sum_{ j=1 }^{N} \sum_{ s \in \T_{\mu}(P_{j,t}) } \Abs{ R_{s} }
			= \sigma^{1+\varepsilon} \sum_{ s \in \T_{t} } \Abs{ R_{s} },
		\end{equation}
		as desired.
	\end{proof}

	\subsection{Proof of Proposition \ref{Prop:DecomMax}}
	\label{Ss:MaxProof}
	
	We now have all the necessary ingredients in order to conclude the proof of Proposition \ref{Prop:DecomMax}.
	
	\begin{proof}[Proof of Proposition \ref{Prop:DecomMax}]
		We have to show \eqref{Eq:MaxEstimate}, where the sum in the left hand side of \eqref{Eq:MaxEstimate} is over the trees $ \{ \mathbf{T}_{j,\nu}' : \, j, \nu =1,\ldots,V_n \} $ constructed in Proposition \ref{Prop:CombinedDensitySize}. Recall that for every $ j $ there exists $ \nu_0 \in \{1,\ldots,V_n\} $ and a lacunary tree $ \mathbf{T}_{j,\nu_0} \subseteq \mathbf{T}_{j,\nu_0}' $ such that
		\begin{equation}
			\sum_{ t \in \mathbf{T}_{j,\nu_0} } F_{10n}[f](t)^2 \gtrsim \size( \P )^2 \Abs{ R_{\mathbf{T}_{j,\nu_0}} }
		\end{equation}
		and the left hand side of the desired inequality \eqref{Eq:MaxEstimate} is bounded as
		\begin{equation}
			\sum_{j,\nu} \Abs{ R_{\mathbf{T}_{j,\nu}'}}
			\lesssim \sum_{j} \Abs{ R_{ \mathbf{T}_{j,\nu_{0}} } }.
		\end{equation}
		Write $ \mathbf{T}_{j} = \mathbf{T}_{j,\nu_0} $ for this proof and as in \eqref{Eq:DecomSizeTopsByDense} we have
		\begin{equation}
			\T^{\tops}_{\siz} 
			\coloneqq \{ \topt(\mathbf{T}_j) \}_j
			= \bigcup_{ t \in \T^{\tops}_{\dens} } \left\{ s \in \T^{\tops}_{\siz} :\, t(s) = t \right\}
			= \bigcup_{ t \in \T^{\tops}_{\dens} } \T^{\tops}_{\siz,t},
		\end{equation}
		where the map $ s \mapsto t(s) $ is as defined in Remark \ref{Rmrk:DenseMap} after Lemma \ref{Lem:Dense} and it implies that $ s \leq_2 t $ for every $ s \in \T^{\tops}_{\siz,t} $. Applying Lemma \ref{Lem:DecomAlgorithm} to each $ \T^{\tops}_{\siz,t} $ with $ \mu=0 $ we further decompose
		\begin{equation}
			\label{Eq:MaxIncomparableDecom}
			\T^{\tops}_{\siz} = \bigcup_{ t\in\T_{\dens}^{\tops} } \bigcup_{\ell=1}^{N_t} \T^{\tops}_{\siz,t}(P_{t,\ell})
		\end{equation}
		with $ P_{t,\ell} \in \T^{\tops}_{\siz,t} $ for every $ \ell \in\{1,\ldots,N_t\} $ and $ \T^{\tops}_{\siz,t}(P_{t,\ell}) $ are trees with top $ P_{t,\ell} $. Now, $ \T^{\tops}_{\siz} $ consists of pairwise incomparable tiles; see the discussion after \eqref{Eq:SizeTopsDecomAlgorithm}. Thus, we have by (ii) of Lemma \ref{Lem:DecomAlgorithm} that
		\begin{equation}
			\sum_{ s\in \T^{\tops}_{\siz,t}(P_{t,\ell}) } \Abs{ R_s }
			= \Abs{ \bigcup_{ s\in \T^{\tops}_{\siz,t}(P_{t,\ell}) } R_s }
			\lesssim \Abs{ R_{P_{t,\ell}} },
		\end{equation}
		where the last inequality is by (iii) of Lemma \ref{Lem:DecomAlgorithm}. Thus,
		\begin{equation}
			\sum_{ s\in\T^{\tops}_{\siz} } \Abs{R_{s}}
			= \sum_{ t\in\T^{\tops}_{\dens} } \sum_{ \ell=1 }^{N_t} \sum_{ s \in \T^{\tops}_{\siz,t}(P_{t,\ell}) } \Abs{ R_s }
			\lesssim \sum_{ t\in\T^{\tops}_{\dens} } \sum_{ \ell=1 }^{N_t} \Abs{ R_{P_{t,\ell}} }.
		\end{equation}
		Now decompose
		\begin{equation}
			\T^{\tops}_{\dens} 
			= \bigcup_{ u\in\Z } \T^{\tops}_{\dens,u}, 
			\qquad \T^{\tops}_{\dens,u} 
			\coloneqq \left\{ t \in \T^{\tops}_{\dens} :\, 2^{-u-1} \abs{ R_t } < \sum_{\ell=1}^{N_t} \abs{ R_{P_{t,\ell}} } \leq 2^{-u} \abs{ R_t } \right\}.
		\end{equation}
		Note that for every $ \ell \in \{1,\ldots,N_t\} $ we have that
		\begin{equation}
			P_{t,\ell} \in \T^{\tops}_{\siz,t} 
			\qquad \Longrightarrow \qquad 
			P_{t,\ell} \leq_2 t 
			\qquad \Longrightarrow \qquad 
			\bigcap_{\ell=1}^{N_t} \ecc(P_{t,\ell}) \neq \emptyset.
		\end{equation}
		Then we get by (iv) of Lemma \ref{Lem:DecomAlgorithm} that
		\begin{equation}
			\sum_{\ell=1}^{N_t} \abs{ R_{P_{t,\ell}} }
			= \Abs{ \bigcup_{\ell=1}^{N_t} R_{P_{t,\ell}} }
			\lesssim \abs{ R_t },
		\end{equation}
		which implies that we only need to consider $ u \geq u_0 $ for some negative integer $ u_0 $ depending only on the dimension. Also, since $ \T^{\tops}_{\dens,u} \subseteq \T^{\tops}_{\dens} $, we get by Lemma \ref{Lem:Dense}
		\begin{equation}
			\label{Eq:MaxDensity}
			\sum_{ t\in\T^{\tops}_{\dens,u} } \abs{R_t} \lesssim \frac{ \abs{E} }{ \delta }.
		\end{equation}
		Finally, as in the proof of Lemma \ref{Lem:Dense} we can decompose $ \T^{\tops}_{\dens,u} $ into disjoint subcollections
		\begin{equation}
			\T^{\tops}_{\dens,u} = \bigcup_{k=0}^{\infty} \T^{\tops}_{\dens,u,k},
		\end{equation}
		where $ k $ is the least integer for which
		\begin{equation}
			\label{Eq:MaxDenseLower}
			\Abs{ E \cap v_{ \sigma }^{-1}(\alpha_{t,\tau_0}) \cap 3^{k} K_n^2 R_{t} } \geq \frac{1}{100} \delta 3^{5 n k} \Abs{ 3^{k} K_n^2 R_{t} }.
		\end{equation}
		Summing up the estimates we have
		\begin{equation}
			\label{Eq:MaxSumSplit}
			\sum_{ s\in\T^{\tops}_{\siz} } \abs{ R_s }
			\lesssim \sum_{ u=u_0 }^{\infty} \sum_{ t\in\T^{\tops}_{\dens,u} } \sum_{\ell=1}^{N_t} \abs{ R_{P_{t,\ell}} }
			\simeq \sum_{ u=u_0 }^{\infty} \sum_{ t\in\T^{\tops}_{\dens,u} } 2^{-u} \abs{ R_t }
			\simeq \sum_{ k=0 }^{\infty} \sum_{ u=u_0 }^{\infty} 2^{-u} \sum_{ t\in\T^{\tops}_{\dens,u,k} } \abs{ R_t }.
		\end{equation}
		
		\emph{Completing the estimate for $ 3^k > (1+\sigma^{-\varepsilon}) $.} Let $ t \in \T^{\tops}_{\dens,u,k} $. Apply Lemma \ref{Lem:MaxSizeHypo} with $ \T_t = \{ P_{t,\ell} \}_{\ell=1}^{N_t} $ to get that 
		\begin{equation}
			\label{Eq:MaxBigSizeKBig}
			\Abs{ F \cap 3^k K_n^2 R_t } \gtrsim \sigma^{1+\varepsilon} \sum_{ \ell=1 }^{N_t} \abs{ R_{P_{t,\ell}} }
			\simeq \sigma^{1+\varepsilon} 2^{-u} \abs{ R_t }
			\simeq \sigma^{1+\varepsilon} 2^{-u} 3^{-nk} \Abs{ 3^k K_n^2 R_t }.
		\end{equation}
		Furthermore, since $ t \in \T^{\tops}_{\dens,u,k} $ we also have the density lower bound \eqref{Eq:MaxDenseLower}. Then, we can apply Proposition \ref{Prop:MaxKey} with $ 3^k K_n^2 $ in place of $ 3^k $, $ \mu \simeq 3^{5nk} \delta $ and $ \lambda \simeq \sigma^{ 1+\varepsilon } 2^{-u} 3^{-nk} $. Thus, the corresponding part of the right hand side of \eqref{Eq:MaxSumSplit} is
		\begin{equation}
			\begin{split}
				\sum_{ 3^k > 1+\sigma^{-\varepsilon} }^{\infty} \sum_{ u=u_0 }^{\infty} 2^{-u} & \sum_{ t\in\T^{\tops}_{\dens,u,k} } \abs{ R_t } \\
				& \lesssim \sum_{ 3^k > 1+\sigma^{-\varepsilon} }^{\infty} \sum_{ 2^{u_0} \leq 2^{u} \leq \frac{ \sigma \abs{ E } }{ \abs{ F } } } 2^{-u} \sum_{ t\in\T^{\tops}_{\dens,u,k} } \abs{ R_t }
				+\sum_{ 3^k > 1+\sigma^{-\varepsilon} }^{\infty} \sum_{ 2^{u} > \frac{ \sigma \abs{ E } }{ \abs{ F } } } 2^{-u} \sum_{ t\in\T^{\tops}_{\dens,u,k} } \abs{ R_t } \\
				& \lesssim \sum_{ 3^k > 1+\sigma^{-\varepsilon} }^{\infty} \sum_{ 2^{u_0} \leq 2^{u} \leq \frac{ \sigma \abs{ E } }{ \abs{ F } } } 2^{-u} \frac{1}{\delta} \frac{1}{3^{5nk}} \frac{ \abs{F} 2^{u(1+\varepsilon) } }{ \sigma^{ (1+\varepsilon)^2 } 3^{ -nk(1+\varepsilon) } }
				+\sum_{ 3^k > 1+\sigma^{-\varepsilon} }^{\infty} \sum_{ 2^{u} > \frac{ \sigma \abs{ E } }{ \abs{ F } } } 2^{-u} \frac{ \abs{ E } }{ \delta },
			\end{split}
		\end{equation}
		where we have used Proposition \ref{Prop:MaxKey} for the first summand and the density estimate \eqref{Eq:MaxDensity} for the second one. Note that the first summand in the last line of the display above is not relevant in the case $\sigma|E|\leq |F|$, in accordance with Remark~\ref{Remark:MaximalCase}. Thus, we get
		\begin{equation}
			\begin{split}
				\sum_{ 3^k > 1+\sigma^{-\varepsilon} }^{\infty} \sum_{ u=u_0 }^{\infty} 2^{-u} \sum_{ t\in\T^{\tops}_{\dens,u,k} } \abs{ R_t }
				& \lesssim \frac{ \abs{F} }{ \delta } \frac{ 1 }{ \sigma^{ (1+\varepsilon)^2 } } \sum_{ 2^{u_0} \leq 2^{u} \leq \frac{ \sigma \abs{ E } }{ \abs{ F } } } 2^{u \varepsilon }
				+ \left( \frac{ \sigma \abs{E} }{ \abs{F} } \right)^{-1} \frac{ \abs{E} }{ \delta } \\
				& \lesssim \frac{ \abs{F} }{ \delta } \frac{ \sigma^{\varepsilon} \abs{E}^{\varepsilon} }{ \sigma^{ (1+\varepsilon)^2 } \abs{F}^{\varepsilon} } + \frac{ 1 }{ \sigma \delta } \abs{ F } 
				= \frac{ \abs{F}^{1-\varepsilon} \abs{E}^{\varepsilon} }{ \delta \sigma^{ 1+\varepsilon+\varepsilon^2 } } + \frac{ \abs{ F }  }{ \sigma \delta } \\
				& \leq \frac{ \abs{F}^{1-\varepsilon-\varepsilon^2 } \abs{F}^{\varepsilon^2} \abs{E}^{\varepsilon} }{ \delta \sigma^{ 1+\varepsilon+\varepsilon^2 } }
				= \frac{ \abs{F}^{1-\varepsilon'} \abs{E}^{\varepsilon'} }{ \delta \sigma^{ 1+\varepsilon' } },
			\end{split}
		\end{equation}
		with $ \varepsilon' = \varepsilon+\varepsilon^2 $ and where in the last line we have used that $ \abs{F} \leq \abs{E} $ and $ \sigma \lesssim 1 $. This finishes the proof in the case where $ 3^k > 1+\sigma^{-\varepsilon} $.
		
		\emph{Completing the estimate for $ 3^k \leq (1+\sigma^{-\varepsilon}) $.} For all such $ k $ we just use the lower bounds for $ t \in \T^{\tops}_{\dens,u,k} $
		\begin{equation}
			\Abs{ F \cap (1+\sigma^{-\varepsilon}) K_n^2 R_t } \gtrsim \sigma^{1+\varepsilon} 2^{-u} \sigma^{\varepsilon n} \Abs{ (1+\sigma^{-\varepsilon}) K_n^2 R_t },
		\end{equation}
		which is a consequence of Lemma \ref{Lem:MaxSizeHypo} and
		\begin{equation}
			\Abs{ E \cap v_{\sigma}^{-1}(\alpha_{t,\tau_0}) \cap (1+\sigma^{-\varepsilon}) K_n^2 R_t } \gtrsim \delta \sigma^{\varepsilon n} \Abs{ (1+\sigma^{-\varepsilon}) K_n^2 R_t },
		\end{equation}
		which is a consequence of \eqref{Eq:MaxDenseLower}. Using these estimates, we can apply Proposition \ref{Prop:MaxKey} with $ (1+\sigma^{-\varepsilon}) K_n^2 $ in place of $ 3^k $, $ \mu \simeq \delta \sigma^{\varepsilon n} $ and $ \lambda \simeq \sigma^{ 1+\varepsilon } 2^{-u} \sigma^{\varepsilon n} $ to get
		\begin{equation}
				\sum_{ 3^k \leq 1+\sigma^{-\varepsilon} }  \sum_{ u=u_0 }^{\infty} 2^{-u}  \sum_{ t\in\T^{\tops}_{\dens,u,k} } \abs{ R_t } 
				\lesssim  \sum_{ 2^{u_0} \leq 2^{u} \leq \frac{ \sigma \abs{ E } }{ \abs{ F } } } 2^{-u} \frac{ \abs{F} }{ \delta \sigma^{ \varepsilon n } } \frac{ 2^{ u(1+\varepsilon) } }{ \sigma^{ (1+\varepsilon)^2 } \sigma^{ \varepsilon n (1+\varepsilon) } }
				+   \sum_{ 2^{u} > \frac{ \sigma \abs{ E } }{ \abs{ F } } } 2^{-u} \frac{ \abs{ E } }{ \delta },
		\end{equation}
		where again we have used Proposition \ref{Prop:MaxKey} for the first summand and the density estimate \eqref{Eq:MaxDensity} for the second one. Thus, we get
		\begin{equation}
			\begin{split}
				\sum_{ 3^k \leq 1+\sigma^{-\varepsilon} }^{\infty} \sum_{ u=u_0 }^{\infty} 2^{-u} \sum_{ t\in\T^{\tops}_{\dens,u,k} } \abs{ R_t }
				& \lesssim \frac{ \abs{ F } \left( \frac{ \sigma \abs{E} }{ \abs{F} } \right)^{\varepsilon} }{ \delta \sigma^{ 1+\varepsilon^2+2\varepsilon+\varepsilon n+\varepsilon^2 n } } + \frac{ \abs{F} }{ \sigma \delta }
				\leq \frac{ \abs{F}^{1-\varepsilon'} \abs{E}^{\varepsilon'} }{ \delta \sigma^{ 1+\varepsilon' } },
			\end{split}
		\end{equation}
		with $ \varepsilon' = \varepsilon^2+2\varepsilon+\varepsilon n+\varepsilon^2 n $, which tends to $ 0 $ as $ \varepsilon \to 0 $. This completes the desired estimate for the tops and with that the proof of Proposition \ref{Prop:DecomMax}.
	\end{proof}

	\section{Removing the single annulus restriction} 
	\label{S:Main}
	
	In this section we will prove Theorem~\ref{Thm:Main} by removing the restriction to a single annulus from Theorem~\ref{Thm:MainSingleBand}. As in previous sections, we will fix $ d=n-1 $.
	
	Recall that by the arguments in Subsection \ref{Ss:Reductions}, it suffices to prove the theorem for the case $ \apl{ \sigma(\cdot) }{ \R^n }{ \Sigma_{ \gamma } } $ with $ \gamma $ a small dimensional constant and $ f $ frequency supported in the cone $ \widetilde{\Gamma}_{\gamma} $. Thus it will be enough to prove
	\begin{equation}
		\Norm{ T_{\mathscr{M},\sigma(\cdot)} P_{\cone,\gamma} f }{ L^p(\R^n) } 
		\lesssim \Norm{ f }{ L^p(\R^n) } 
	\end{equation}
	under the assumption that $ \apl{\sigma}{ \R^n }{ \Sigma_{ \gamma } } $ is measurable with $ \sigma(x) = \sigma(\underline{x}) $, for $ x = (\underline{x},x_n) \in \R^n $ and $ \gamma $ a small dimensional constant. Note that under these assumptions, Theorem~\ref{Thm:MainSingleBand} implies
	\begin{equation}
		\label{Eq:MainSingleBand}
		\Norm{ T_{\mathscr{M},\sigma(\cdot)} P_{\cone,\gamma} P_{\ann} f }{ L^p(\R^n) } 
		\lesssim \Norm{ f }{ L^p(\R^n) } 
	\end{equation}
	uniformly in $ \ann \in 3^{\Z} $, where recall that $ P_{\ann} $ is defined in \eqref{Eq:ProjBandDef}. Indeed, Theorem~\ref{Thm:MainSingleBand} implies
	\begin{equation}
		\sup_{ \ann \in 3^{\Z} } \Norm{ T_{\mathscr{M},\sigma(\cdot)} P_{\ann} f }{ L^p(\R^n) } 
		\lesssim \Norm{ f }{ L^p(\R^n) } 
	\end{equation}
	and \eqref{Eq:MainSingleBand} follows since
	\begin{equation}
		\Abs{ T_{ \mathscr{M},\sigma(\cdot) } [ ( \Id - P_{\cone,\gamma} ) \circ P_{\ann} f ] } \lesssim M f
	\end{equation}
	uniformly in $ \ann \in 3^{\Z} $, where $ M $ is the Hardy--Littlewood maximal function.
	
	Let $ \apl{\Phi}{\R}{[0,1]} $ with $ \supp( \Phi ) \subseteq [\frac{1}{3},3] $ and
	\begin{equation}
		\sum_{ \ann \in 3^{\Z} } \Phi\left( \frac{ \abs{t} }{ \ann } \right) = 1, \qquad t \neq 0.
	\end{equation}
	Let 
	\begin{equation}
		\mathcal{F}\left( S_{\ann} f \right)(\xi) \coloneqq \Phi\left( \frac{ \abs{\xi_n} }{ \ann } \right) \widehat{f}(\xi)
	\end{equation}
	be a smooth band restriction. Using \eqref{Eq:MainSingleBand} we can conclude that
	\begin{equation}
		\label{Eq:MainSingleBandSmooth}
		\sup_{ \ann \in 3^{\Z} } \Norm{ T_{\mathscr{M},\sigma(\cdot)} S_{\ann} P_{\cone,\gamma} f }{ L^p(\R^n) }
		\lesssim \Norm{ f }{ L^p(\R^n) },
	\end{equation}
	for all $ 1<p<\infty $ with $ \sigma(\cdot) $ as above. Defining
	\begin{equation}
		T_{\ann} f 
		\coloneqq T_{\mathscr{M},\sigma(\cdot)} S_{\ann} P_{\cone,\gamma} f
	\end{equation}
	and using Littlewood--Paley theory in the last variable of $ \R^n $ we see that Theorem~\ref{Thm:Main} reduces to the proof of the vector-valued bound
	\begin{equation}
		\Norm{ \left( \sum_{ \ann \in 3^{\Z} } \Abs{ T_{\ann} f_{\ann} }^2 \right)^{1/2} }{ L^p(\R^n) } 
		\lesssim \Norm{ \left( \sum_{ \ann \in 3^{\Z} } \Abs{ f_{\ann} }^2 \right)^{1/2} }{ L^p(\R^n) }
	\end{equation}
	for $ 3/2 < p < \infty $.
	
	By Marcinkiewicz interpolation for $ \ell^2(\Z) $-valued operators, it will be enough to show that
	\begin{equation}
		\label{Eq:MainReductionWeak}
		\Abs{ \left\langle \bigg( \sum_{ \ann \in 3^{\Z} } \abs{ T_{\ann} f_{\ann} }^2 \bigg)^{1/2} , \indf{G} \right\rangle } \lesssim \abs{ H }^{ 1/p } \abs{ G }^{1/p'}, \qquad 3/2 < p < \infty,
	\end{equation}
	with $ \sum_{ \ann \in 3^{\Z} } \Abs{ f_{\ann} }^2 \leq \indf{H} $ with $ G,H $ measurable subsets of $ \R^n $ of finite measure.
	
	Now, note that \eqref{Eq:MainReductionWeak} for $ p=2 $ is an immediate consequence of \eqref{Eq:MainSingleBandSmooth}. Furthermore, the case $ p<2 $ of \eqref{Eq:MainReductionWeak} with $ \abs{H} \geq \abs{G} $ and the case $ p>2 $ of \eqref{Eq:MainReductionWeak} with $ \abs{H} \leq \abs{G} $ follow from the case $ p=2 $ of \eqref{Eq:MainReductionWeak}. It thus remains to show
	\addtocounter{equation}{-1}
	\begin{subequations}
		\begin{align}
		\label{Eq:MainReductionWeak<2}
		 \Abs{ \left\langle \bigg( \sum_{ \ann \in 3^{\Z} } \abs{ T_{\ann} f_{\ann} }^2 \bigg)^{1/2} , \indf{G} \right\rangle } \lesssim \abs{ H }^{ 1/p } \abs{ G }^{1/p'},& \qquad 3/2 < p < 2, \quad \abs{H} \leq \abs{G},
		\tag{5.3$_{\leq}$}
		\\
		\label{Eq:MainReductionWeak>2}
		 \Abs{ \left\langle \bigg( \sum_{ \ann \in 3^{\Z} } \abs{ T_{\ann} f_{\ann} }^2 \bigg)^{1/2} , \indf{G} \right\rangle } \lesssim \abs{ H }^{ 1/p } \abs{ G }^{1/p'},& \qquad 2 < p < \infty, \quad \abs{H}\geq  \abs{G} .
		\tag{5.3$_{\geq}$} 
		\end{align}
	\end{subequations}
	We will prove \eqref{Eq:MainReductionWeak<2} and \eqref{Eq:MainReductionWeak>2} in Subsections \ref{Ss:MainPL2} and \ref{Ss:MainPG2} respectively.

	\subsection{The case $ p < 2 $}
	\label{Ss:MainPL2}
	
	In this subsection, we suppose that $ p < 2 $ and $ \abs{ H } \leq \abs{ G } $ and the goal is to show \eqref{Eq:MainReductionWeak<2}. To do that we will recursively apply the following lemma.

	\begin{lemma}
		\label{Lem:MainInductionL2}
		Let $ G , H \subseteq \R^n $ be two measurable subsets of finite measure with $ \abs{ H } \leq \abs{ G } $ and $ 3/2 < p < 2 $. Then, there exists a subset $ G' \subseteq G $ depending on $ p, G $ and $ H $ with $ \abs{ G' } \geq \abs{ G }/2 $ such that \eqref{Eq:MainReductionWeak<2} holds for $ G' $ instead of $ G $.
	\end{lemma}
	
	Before proving Lemma \ref{Lem:MainInductionL2}, we will see how to use it to deduce \eqref{Eq:MainReductionWeak<2} for arbitrary $ G $ with $ \abs{G} \geq \abs{H} $. Write $ G_0 = G $ and apply Lemma \ref{Lem:MainInductionL2} to get the subset $ G'_0 $ with $ \abs{ G'_0 } \geq \frac{1}{2} \abs{G_0} $ and satisfying \eqref{Eq:MainReductionWeak<2}. In particular, it satisfies \eqref{Eq:MainReductionWeak} with an implicit constant $ C $ independent of $ \{ f_{\ann} \}_{\ann} $, $ H $ and $ G_0 $. Then, splitting the left hand side of \eqref{Eq:MainReductionWeak}, we get
	\begin{equation}
		\Abs{ \left\langle \bigg( \sum_{ \ann \in 3^{\Z} } \abs{ T_{\ann} f_{\ann} }^2 \bigg)^{1/2} , \indf{G_0} \right\rangle }
		\leq C \abs{ H }^{ 1/p } \abs{ G_0' }^{1/p'} + \Abs{ \left\langle \bigg( \sum_{ \ann \in 3^{\Z} } \abs{ T_{\ann} f_{\ann} }^2 \bigg)^{1/2} , \indf{G_1} \right\rangle },
	\end{equation}
	where we have defined $ G_1 \coloneqq G_0 \setminus G'_0 $. Observe that we have $ \abs{G_1} \leq \frac{1}{2} \abs{ G_0 } $. Iterating we construct $ G_{j} \subseteq G_{j-1} \subseteq \cdots \subseteq G_0 = G $ with $ \abs{G_j} \leq 2^{-j} \abs{G} $. Letting $ j_0 $ be the first positive integer $ j $ such that $ 2^{-j_0} \abs{G} \leq \abs{H} $ we get that $ 2^{j_0} \simeq \abs{G}/\abs{H} $ and
	\begin{equation}
		\Abs{ \left\langle \bigg( \sum_{ \ann \in 3^{\Z} } \abs{ T_{\ann} f_{\ann} }^2 \bigg)^{1/2} , \indf{G_0} \right\rangle }
		\lesssim \log \left( e + \frac{\abs{G}}{\abs{H}} \right) \abs{H}^{1/p} \abs{G}^{1/p'},
	\end{equation}
	for every $ 3/2 < p < 2 $. Clearly this implies \eqref{Eq:MainReductionWeak<2} and thus \eqref{Eq:MainReductionWeak} in the open range $ 3/2 < p < 2 $.
	
	Given $ G $ and $ 3/2 < p < 2 $ we now need to construct $ G' \subseteq G $, $ \abs{G'} > \frac{1}{2} \abs{G} $ such that 
	\begin{equation}
		\label{Eq:MainReductionWeak<2Subset}
		\Abs{ \left\langle \bigg( \sum_{ \ann \in 3^{\Z} } \abs{ T_{\ann} f_{\ann} }^2 \bigg)^{1/2} , \indf{G'} \right\rangle }
		\lesssim \abs{H}^{1/p} \abs{G'}^{1/p'}, \qquad \sum_{ \ann \in 3^{\Z} } \abs{ f_{\ann} }^2 \leq \indf{H}.
	\end{equation}
	Defining 
	\begin{equation}
		T_{\ann,G',H} f \coloneqq \indf{G'} T_{\ann} \left( f \indf{H} \right)
	\end{equation}
	we claim that \eqref{Eq:MainReductionWeak<2Subset} is a consequence of
	\begin{equation}
		\label{Eq:MainReductionWeak<2Subset2}
		\sup_{ \ann \in 3^{\Z} } \Abs{ \langle T_{\ann,G',H} f , \indf{E} \rangle } 
		\lesssim \left( \frac{ \abs{ G' } }{ \abs{ H } } \right)^{ \frac{1}{2} - \frac{1}{p} } \abs{ F }^{ 1/2 } \abs{ E }^{ 1/2 },
	\end{equation}
	for every $ 3/2 < p < 2 $ and for every $ F,E \subseteq \R^n $ measurable of finite measure with $ \abs{f} \leq \indf{F} $ and $ G',H $ fixed.
	
	To see this first note that Theorem~\ref{Thm:MainSingleBand} implies 
	\begin{equation}
		\label{Eq:MainSingleEstimateWeak<2}
		\sup_{ \ann \in 3^{\Z} } \Abs{ \langle T_{\ann,G',H} f , \indf{E} \rangle } 
		\lesssim \Abs{ F }^{ 1/q } \Abs{ E }^{ 1/q' },
	\end{equation}
	for every $ 1<q<\infty $. Now, interpolating between \eqref{Eq:MainReductionWeak<2Subset2} and \eqref{Eq:MainSingleEstimateWeak<2} for $ q < 2 $ and  for $ q > 2 $,  and interpolating again between the resulting interpolation estimates, we get
	\begin{equation}
		\label{Eq:MainSingleEstimateInterpolation<2}
		\norm{ T_{\ann,G',H} f }{ L^2(\R^n) } \lesssim \left( \frac{ \abs{ G } }{ \abs{ H } } \right)^{ \frac{1}{2} - \frac{1}{p} } \norm{ f }{ L^2(\R^n) }, \qquad 1<p<\infty.
	\end{equation}
	Finally, applying the Cauchy--Schwarz inequality,
	\begin{equation}
		\Abs{ \left\langle \bigg( \sum_{\ann \in 3^{\Z} } \abs{ T_{\ann} f_{\ann} }^2 \bigg)^{1/2} , \indf{G'} \right\rangle }^2 
		\leq \abs{ G } \sum_{ \ann \in 3^{\Z} } \int \Abs{ T_{\ann,G',H} f_{\ann} }^2 
		\leq \abs{ G }^{ 2 - \frac{2}{p} } \abs{ H }^{ 2/p }.
	\end{equation}
	This is equivalent to \eqref{Eq:MainReductionWeak<2Subset}, which shows the sufficiency of \eqref{Eq:MainReductionWeak<2Subset2}. Furthermore, note that in order to prove \eqref{Eq:MainReductionWeak<2Subset2} it suffices to assume that $ E \subseteq G' $ and $ F \subseteq H $.
	
	In order to prove \eqref{Eq:MainReductionWeak<2Subset2}, we will exploit the decomposition
	\begin{equation}
		\P = \P^{\light} \cup \P^{\heavy}_{\smal} \cup \bigcup_{j} \bigcup_{\nu=1}^{V_n} \mathbf{T}_{j,\nu}'
	\end{equation}
	from Proposition \ref{Prop:CombinedDensitySize} together with the single tree estimate of \cite[Lemma 6.14]{BDPPR}. We remember that the trees $ \mathbf{T}_{j,\nu}' $ had $ \densesup(\mathbf{T}_{j,\nu}') \simeq \densesup(\P) $ for every $j,\nu$ and that for every $ j $ there exists $ \nu_0 \in \{1,\ldots,V_n\} $ and a lacunary tree $ \mathbf{T}_{j,\nu_0} \subseteq \mathbf{T}_{j,\nu_0}' $ with $ R_{ \mathbf{T}_{j,\nu_0} } = R_{ \mathbf{T}_{j,\nu_0}' } $ such that  
	\begin{equation}
		\sum_{ t \in \mathbf{T}_{j,\nu_0} } F_{10n}[f](t)^2 \gtrsim \size( \P )^2 \Abs{ R_{ \mathbf{T}_{j,\nu_0} } }.
	\end{equation}
	Besides the estimates for $ \sum_{j,\nu} \abs{ R_{ \mathbf{T}_{j,\nu}' } } $ from Proposition \ref{Prop:CombinedDensitySize} and Proposition \ref{Prop:DecomMax}, we will prove here improved estimates reflecting the localizations $ F \subseteq H $ and $ E \subseteq G' $, affecting the definitions of $ \size $ and $ \densesup $, respectively. To that end we will use the definitions and notations introduced in Section \ref{S:Decom} for the proofs of Proposition \ref{Prop:CombinedDensitySize} and Proposition \ref{Prop:DecomMax}.

	\subsubsection{Improved estimate for the proof of\hspace{.2em} \eqref{Eq:MainReductionWeak<2Subset2}} 
	\label{Sss:MainInductionSubsetL2}
	
	The proof of \eqref{Eq:MainReductionWeak<2Subset2} relies on the following lemma.
	
	\begin{lemma}
		\label{Lem:MaxImproved}
		There exists $ G' \subseteq G $ with $ \abs{G'} \geq \frac{1}{2} \abs{G} $ such that the trees $ \{ \mathbf{T}_{j,\nu}' \}_{j,\nu} $ constructed in Proposition \ref{Prop:CombinedDensitySize} satisfy
		\begin{equation}
			\sum_{j} \sum_{ \nu=1 }^{V_n}  \Abs{ R_{ \mathbf{T}_{j,\nu}' } } \lesssim_{\varepsilon} \abs{ E } \left( \frac{ \abs{H} }{ \abs{G} } \right)^{ \frac{1}{2} } \densesup(\P)^{ -\frac{3}{2} - \varepsilon } \size(\P)^{ -1 - \varepsilon },
		\end{equation}
		for every $ \varepsilon>0 $, where $ \size $ is defined with respect to $ F \subseteq H $ and $ \densesup $ with respect to $ E \subseteq G' $.
	\end{lemma}
	
	\begin{proof}
		We abbreviate, as usual, $ \delta \coloneq \densesup(\P) $ and $ \sigma \coloneqq \size(\P) $. Recall that $ \P \subseteq \mathcal{T}_{\ann} $ for some $ \ann \in 3^{\Z} $. Consider any $ G' $ with $ E \subseteq G' \subseteq G $, which will be made specific shortly. Remember that for every $ j $ there exists $ \nu_0 \in \{1,\ldots,V_n\} $ and a lacunary tree $ \mathbf{T}_{j,\nu_0} \subseteq \mathbf{T}_{j,\nu_0}' $ such that
		\begin{equation}
			\label{Eq:MaxImprovedBigSize}
			\sum_{ t \in \mathbf{T}_{j,\nu_0} } F_{10n}[f](t)^2 \gtrsim \sigma^2 \Abs{ R_{\mathbf{T}_{j,\nu_0}} }
		\end{equation}
		and
		\begin{equation}
			\sum_{j,\nu} \Abs{ R_{\mathbf{T}_{j,\nu}'}}
			\lesssim \sum_{j} \Abs{ R_{ \mathbf{T}_{j,\nu_{0}} } }.
		\end{equation}
		Write $ \mathbf{T}_{j} = \mathbf{T}_{j,\nu_0} $ and decompose as in \eqref{Eq:DecomSizeTopsByDense}
		\begin{equation}
			\T^{\tops}_{\siz} 
			\coloneqq \{ \topt(\mathbf{T}_j) \}_j
			= \bigcup_{ t \in \T^{\tops}_{\dens} } \left\{ s \in \T^{\tops}_{\siz} :\, t(s) = t \right\}
			= \bigcup_{ t \in \T^{\tops}_{\dens} } \T^{\tops}_{\siz,t},
		\end{equation}
		with the map $ s \mapsto t(s) $ as defined through Lemma \ref{Lem:Dense} and Remark \ref{Rmrk:DenseMap}. We have 
		\begin{equation}
			\sum_{j} \Abs{ R_{ \mathbf{T}_{j} } }
			= \sum_{ t\in\T^{\tops}_{\dens} } \sum_{ s\in\T^{\tops}_{\siz,t} } \Abs{ R_{ s } }
			\lesssim \sum_{ t\in\T^{\tops}_{\dens} } \sum_{ \ell=1 }^{N_t} \Abs{ R_{P_{t,\ell}} },
		\end{equation}
		with $ \{ P_{t,\ell} \}_{\ell=1}^{N_t} \subseteq \T^{\tops}_{\siz,t} $ constructed via Lemma \ref{Lem:DecomAlgorithm} as in \eqref{Eq:MaxIncomparableDecom} in the proof of Proposition \ref{Prop:DecomMax}. Remember the definition 
		\begin{equation}
			\T^{\tops}_{\dens} 
			= \bigcup_{ u = u_0 }^{\infty} \T^{\tops}_{\dens,u}, 
			\qquad \T^{\tops}_{\dens,u} 
			\coloneqq \left\{ t \in \T^{\tops}_{\dens} :\, 2^{-u-1} \abs{ R_t } < \sum_{\ell=1}^{N_t} \Abs{ R_{P_{t,\ell}} } \leq 2^{-u} \abs{ R_t } \right\},
		\end{equation}
		where $ u \geq u_0 $ for some negative integer $ u_0 $ depending only on the dimension. Thus,
		\begin{equation}
			\sum_{j} \abs{ R_{ \mathbf{T}_{j} } }
			\lesssim \sum_{ u \geq u_0 } 2^{-u} \sum_{ t\in\T^{\tops}_{\dens,u} } \Abs{ R_{t} }
			\lesssim \sum_{ u \geq u_0 } 2^{-u} \frac{ \abs{E} }{ \delta },
		\end{equation}
		the last estimate following by density as in \eqref{Eq:MaxDensity} in the proof of Proposition \ref{Prop:DecomMax}. 
		
		Since $ \dense(t) \geq \delta/2 $ for every $ t \in \T^{\tops}_{\dens,u} $ we decompose as in \eqref{Eq:MaxDenseLower} 
		\begin{equation}
			\T^{\tops}_{\dens,u} = \bigcup_{k=0}^{\infty} \T^{\tops}_{\dens,u,k},
		\end{equation}
		where $ k $ is the least integer for which
		\begin{equation}
			\label{Eq:MainDenseLower}
			\Abs{ G' \cap v_{ \sigma }^{-1}(\alpha_{t,\tau_0}) \cap 3^{k} K_n^2 R_{t} } \geq \Abs{ E \cap v_{ \sigma }^{-1}(\alpha_{t,\tau_0}) \cap 3^{k} K_n^2 R_{t} }
			\geq \frac{1}{100} \delta 3^{5 n k} \Abs{ 3^{k} K_n^2 R_{t} }
		\end{equation}
		for every $ t \in \T^{\tops}_{\dens,u,k} $. Note that we used $ E \subseteq G' $ in the first inequality in \eqref{Eq:MainDenseLower}. Now we make a specific choice for $ G' $. First, construct $ \widetilde{G} $ as follows:
		\begin{equation}
			\widetilde{G} 
			\coloneqq \bigcup_{\ell=0}^{\infty} \bigcup_{\mu=0}^{\infty} \left\{ 3^{\mu} K_n^2 R_{t} :\, t \in \mathcal{T}_{\ann}, \, \frac{ \abs{ v_{ \sigma }^{-1}(\alpha_{t,\tau_0}) \cap 3^{\mu} K_n^2 R_{t} } }{ \abs{ 3^{\mu} K_n^2 R_{t} } } \geq 3^{-\ell}, \, \frac{ \abs{ H \cap 3^{\mu} K_n^2 R_{t} } }{ \abs{ 3^{\mu} K_n^2 R_{t} } } \geq \lambda_{\ell,H,G} \right\},
		\end{equation}
		with $ \lambda_{\ell,H,G} \coloneqq C_{\varepsilon} 3^{(\frac{1}{2} + \varepsilon) \ell } \left( \frac{\abs{H}}{\abs{G}} \right)^{1/2} $ and some constant $ C_{\varepsilon} $ big enough to be chosen. Then, set $ G' \coloneqq G \setminus \widetilde{G} $. Note that the tiles $ t \in \T^{\tops}_{\dens,u,k} $ satisfy $ 3^{k} K_n^2 R_{t} \cap G' \neq \emptyset $ because of \eqref{Eq:MainDenseLower}. Therefore $ 3^{k} K_n^2 R_{t} \not\subseteq \widetilde{G} $. This, together with estimate \eqref{Eq:MainDenseLower}, implies that
		\begin{equation}
			\label{Eq:MainSmallSize}
			\Abs{ H \cap 3^k K_n^2 R_t } < C_{\varepsilon} 3^{ (\frac{1}{2}+\varepsilon) \ell_0 } \left( \frac{ \abs{H} }{ \abs{G} } \right)^{1/2} \Abs{ 3^k K_n^2 R_t }
		\end{equation}
		for $ 3^{-\ell_0} \simeq \delta 3^{ 5nk } $ so that
		\begin{equation}
			\Abs{ H \cap 3^k K_n^2 R_t } \lesssim C_{\varepsilon} \left( 3^{5nk} \delta \right)^{-(\frac{1}{2}+\varepsilon)} \left( \frac{ \abs{H} }{ \abs{G} } \right)^{1/2} \Abs{ 3^k K_n^2 R_t }.
		\end{equation}
		
		If $ \T^{\tops}_{\dens,u,k} \neq \emptyset $ for some $ 3^k > 1+\sigma^{-\varepsilon} $, then there exists $ t_0 \in \T^{\tops}_{\dens,u,k} $ and we apply Lemma \ref{Lem:MaxSizeHypo} with $ \T_t = \T_{t_0} \coloneqq \{ P_{t_0,\ell} \}_{\ell=1}^{N_{t_0}} $ to get
		\begin{equation}
			\Abs{ H \cap 3^k K_n^2 R_{t_0} } 
			\geq \Abs{ F \cap (1+\sigma^{-\varepsilon}) K_n^2 R_{t_0} }
			\gtrsim \sigma^{1+\varepsilon} 2^{-u} 3^{-kn} \Abs{ 3^k K_n^2 R_{t_0} },
		\end{equation}
		similarly as in \eqref{Eq:MaxBigSizeKBig}. In order to apply Lemma \ref{Lem:MaxSizeHypo} we used that each $ P_{ t_0,\ell } $ is the top of a lacunary tree $ \mathbf{T}_{ t_0,\ell } $ satisfying \eqref{Eq:MaxImprovedBigSize}. The above calculation yields 
		\begin{equation}
			2^{-u} 
			\lesssim 2^{-u_1} 
			\coloneqq \sigma^{-(1+\varepsilon)} \delta^{-\frac{1}{2}-\varepsilon} \left( \frac{\abs{H}}{\abs{G}} \right)^{1/2}.
		\end{equation}
		Thus,
		\begin{equation}
			\sum_{j} \abs{ R_{ \mathbf{T}_{j} } }
			\lesssim \sum_{ 2^{-u} \lesssim 2^{-u_1} } 2^{-u} \frac{ \abs{E} }{ \delta }
			\lesssim \sigma^{-(1+\varepsilon)} \delta^{-\frac{3}{2}-\varepsilon} \abs{E} \left( \frac{\abs{H}}{\abs{G}} \right)^{1/2}.
		\end{equation}
		
		If $ \T^{\tops}_{\dens,u,k} \neq \emptyset $ for some $ 3^k \leq 1+\sigma^{-\varepsilon} $, then there exists $ t_0 \in \T^{\tops}_{\dens,u,k} $ for which
		\begin{equation}
			\begin{split}
				\Abs{ (1+\sigma^{-\varepsilon}) K_n^2 R_{t_0} \cap v_{\sigma}^{-1}(\alpha_{t_0,\tau_0}) }
				& \geq \Abs{ 3^k K_n^2 R_{t_0} \cap v_{\sigma}^{-1}(\alpha_{t_0,\tau_0}) }
				> \frac{1}{100} \delta 3^{kn} \Abs{ 3^k K_n^2 R_{t_0} } \\
				& \gtrsim \delta \Abs{ K_n^2 R_{t_0} } 
				= \delta (1+\sigma^{-\varepsilon})^{-n} \Abs{ (1+\sigma^{-\varepsilon}) K_n^2 R_{t_0} }.
			\end{split}
		\end{equation}
		Thus, setting $ 3^{-\ell_0} \simeq \delta (1+\sigma^{-\varepsilon})^{-n} $ we see by \eqref{Eq:MainSmallSize} that
		\begin{equation}
			\Abs{ H \cap (1+\sigma^{-\varepsilon}) K_n^2 R_{t_0} }
			\lesssim C_{\varepsilon} \left( \frac{ \delta }{ (1+\sigma^{-\varepsilon})^n } \right)^{-(\frac{1}{2}+\varepsilon)} \left( \frac{ \abs{H} }{ \abs{G} } \right)^{1/2} \Abs{ (1+\sigma^{-\varepsilon}) K_n^2 R_{t_0} }.
		\end{equation}
		By Lemma \ref{Lem:MaxSizeHypo} we conclude again that
		\begin{equation}
			\Abs{ H \cap (1+\sigma^{-\varepsilon}) K_n^2 R_{t_0} }
			\gtrsim \sigma^{1+\varepsilon} 2^{-u} (1+\sigma^{-\varepsilon})^{-n} \Abs{ (1+\sigma^{-\varepsilon}) K_n^2 R_{t_0} }.
		\end{equation} 
		Combining these estimates
		\begin{equation}
			2^{-u} \sigma^{1+\varepsilon}
			\lesssim \left( \frac{ \delta }{ (1+\sigma^{-\varepsilon})^n } \right)^{-(\frac{1}{2}+\varepsilon)} \left( \frac{ \abs{H} }{ \abs{G} } \right)^{1/2} (1+\sigma^{-\varepsilon})^{n}
		\end{equation}
		which implies
		\begin{equation}
			2^{-u} \lesssim 2^{-u_2} \coloneqq \delta^{-(\frac{1}{2}+\varepsilon)} \left( \frac{ \abs{H} }{ \abs{G} } \right)^{1/2} \sigma^{-1-\varepsilon n}.
		\end{equation}
		Summing for $ 2^{u} \gtrsim 2^{u_2} $ as before yields
		\begin{equation}
			\sum_{j} \Abs{ R_{ \mathbf{T}_{j} } }
			\lesssim \sum_{ 2^{-u} \lesssim 2^{-u_2} } 2^{-u} \frac{ \abs{E} }{ \delta }
			\lesssim \sigma^{-(1+\varepsilon')} \delta^{-\frac{3}{2}-\varepsilon} \abs{E} \left( \frac{\abs{H}}{\abs{G}} \right)^{1/2}.
		\end{equation}
		
		In order to conclude the proof we need to show that $ \abs{G'} \geq \frac{1}{2} \abs{G} $. This requires some additional effort and we give the proof in Subsection \ref{Sss:MainPL2Major} below.
	\end{proof}

	\subsubsection{Proving that $G'$ is a major subset}
	\label{Sss:MainPL2Major}
	
	The main tool is a covering lemma originally due to Lacey and Li \cite{LaceyLiMemoirs}, revisited in \cite[Theorem 8]{BatemanThiele} and in \cite[Appendix A]{DPGTZ}. We note here that a version of this lemma is valid assuming that $ \sigma(\underline{x},x_n) $ is Lipschitz in $ x_n $ for almost every $ x_n \in \R $, and that $ \sigma(\R^n)$ is almost horizontal. In that scenario one needs the restriction 
	\begin{equation}
		\ell(L)
		\leq \sup_{ \underline{x} \in \R^d } \norm{ \sigma(\underline{x},\cdot) }{\lip}^{-1}
	\end{equation}
	together with a corresponding truncation of the operator at a comparable scale. As our map $ \apl{\sigma}{\R^n}{\Sigma_{\gamma}} $ depends only on $ \underline{x} \in \R^d $, it is constant along vertical lines $ \{ (\underline{x},t) : \, t \in \R \} $ for fixed $ \underline{x} \in \R^d $, and the restriction becomes vacuous.
	
	\begin{lemma}
		\label{Lem:MainConstructionG}
		Let $ \mu \geq 0 $ and $ \lambda \leq 1 $. For a measurable set $ H \subseteq \R^n $ and a finite collection of tiles $ \P \subseteq \mathcal{T} $ such that
		\begin{equation}
			\label{Eq:MainConstructionGLemHyp}
			\Abs{ v_{ \sigma }^{-1}(\alpha_{p,\tau_0}) \cap R_p } \geq \mu \abs{ R_p }, \qquad \Abs{ H \cap R_{p} } \geq \lambda \abs{ R_p }, \qquad \text{ for every } \quad p \in \P
		\end{equation}
		there holds
		\begin{equation}
			\bigg| \bigcup_{ p \in \P } R_p \bigg| \lesssim \mu^{-1} \lambda^{-2} \abs{H}.
		\end{equation}
	\end{lemma}
	
	The proof of Lemma \ref{Lem:MainConstructionG} follows as in the two-dimensional case of \cite[Theorem 8]{BatemanThiele}, with adjustments for the higher-dimensional geometry, provided by Lemmas \ref{Lem:Geom2} and \ref{Lem:Geom3}.
	
	\begin{proof}[Proof of $ \abs{G'} > \frac{1}{2} \abs{G} $]
		Following the same arguments as in the proof of Proposition \ref{Prop:MaxKey} involving shifted triadic grids, we can find for each $ \mu,t $ a tile $ \tilde{t} $ with $ Q_{\tilde{t}} = Q_t $ and $ R_{\tilde{t}} = R( \widetilde{L}_{t}, \widetilde{I}_{t}, v_{R_t} ) $ with $ \widetilde{L}_{t} \in \bigcup_{j=1}^{2^d} \mathcal{L}_j $, $ \widetilde{I}_{t} \in \mathcal{I}_1 \cup \mathcal{I}_2 $ such that $ 3^{\mu} K_n^2 R_t \subseteq R_{\tilde{t}} \subseteq 3^{\mu+2} K_n^2 R_t $ and such that
		\begin{equation}
			\begin{split}
				\widetilde{G} 
				& \subseteq \bigcup_{ \ell \geq 0 } \bigcup_{ j_{L} = 1 }^{2^d} \bigcup_{ j_{I} = 1 }^{2} \left\{ R_{\tilde{t}} :\, \tilde{t} \in \mathcal{T}_{j_{L},j_{I}}, \, \frac{ \abs{ v_{ \sigma }^{-1}(\alpha_{\tilde{t},\tau_{0}}) \cap R_{\tilde{t}} } }{ \abs{ R_{\tilde{t}} } } \gtrsim 3^{-\ell}, \, \frac{ \abs{ H \cap R_{\tilde{t}} } }{ \abs{ R_{\tilde{t}} } } \gtrsim C_{\varepsilon} 3^{(\frac{1}{2} + \varepsilon) \ell } \left( \frac{\abs{H}}{\abs{G}} \right)^{1/2} \right\} \\
				& \eqqcolon \bigcup_{ \ell \geq 0 } \bigcup_{ j_{L} = 1 }^{2^d} \bigcup_{ j_{I} = 1 }^{2} \widetilde{G}_{j_{L},j_{I},\ell}.
			\end{split}
		\end{equation}
		Thus, using Lemma \ref{Lem:MainConstructionG} for each $ \widetilde{G}_{j_{L},j_{I},\ell} $
		\begin{equation}
			\abs{ \widetilde{G} } 
			\lesssim \sum_{ j_{L} = 1 }^{2^d} \sum_{ j_{I} = 1 }^{2} \sum_{ \ell \geq 0 } \abs{ \widetilde{G}_{j_{L},j_{I},\ell} }
			\lesssim C_{\varepsilon}^{-2} \sum_{ \ell \geq 0 } 3^{ -2 \varepsilon \ell } \abs{G}
		\end{equation}
		and so $ \abs{ \widetilde{G} } \leq \frac{1}{2} \abs{G} $, upon choosing $ C_{\varepsilon} $ sufficiently large depending only on $ \varepsilon>0 $ and the dimension. Thus, $ \abs{G'} \geq \frac{1}{2} \abs{G} $ and the proof is complete.
	\end{proof}

	\subsubsection{Completing the proof of \eqref{Eq:MainReductionWeak<2Subset2}}
		
	Using the same reduction to the model operator as in Section \ref{S:Model}, the proof of \eqref{Eq:MainReductionWeak<2Subset2} for $ 3/2 < p < 2 $ reduces to proving 
	\begin{equation}
		\Lambda_{\P;\sigma,\tau_0,M} (f\indf{F}, g\indf{E})
		\lesssim \left( \frac{ \abs{G} }{ \abs{H} } \right)^{\frac{1}{2}-\frac{1}{p}} \abs{F}^{1/2} \abs{E}^{1/2}
	\end{equation}
	whenever $ \abs{f} \leq \indf{F} $, $ \abs{g} \leq \indf{E} $, $ F \subseteq H $, $ E \subseteq G' $ with $ G' \subseteq G $ as in Lemma~\ref{Lem:MaxImproved} satisfying $ \abs{G'} > \frac{1}{2} \abs{G} $ and $ \P \subseteq \mathcal{T}_{\ann} $; the implicit constant above should be independent of $ \ann, F, H, E, G $.
	
	Iterating Proposition \ref{Prop:CombinedDensitySize} as in Lemma \ref{Lem:Decom} and combining with the improved estimates for the sums of tops of Lemma~\ref{Lem:MaxImproved} we conclude 
	\begin{equation}
		\Lambda_{\P;\sigma,\tau_0,M} (f\indf{F}, g\indf{E})
		\lesssim \sum_{ j=J }^{\infty} \sum_{ k=K }^{\infty} 2^{-j} 2^{-k} \sum_{ \tau=1 }^{N_{k,j}} \Abs{ R_{ \mathbf{T}_{k,j,\tau} } },
	\end{equation}
	where $ 2^J, 2^K \gtrsim 1 $ and $ \sum_{ \tau=1 }^{N_{k,j}} \Abs{ R_{ \mathbf{T}_{k,j,\tau} } } $ satisfies the estimate
	\begin{equation}\label{Eq:MainImprovedEstimatesL2}
		\sum_{\tau =1}^{N_{k,j}} \Abs{ R_{ \mathbf{T}_{k,j,\tau} } } 
		\lesssim \min \left\{ 2^{k} \abs{ E } , \, 2^{2j} \abs{F} , \, \abs{ E } \left( \frac{ \abs{H} }{ \abs{G'} } \right)^{ 1/2 } 2^{(\frac{3}{2}+\varepsilon) k} 2^{ ( 1+ \varepsilon) j } \right\}.
	\end{equation}
	With this, we are ready to show \eqref{Eq:MainReductionWeak<2Subset2}. First of all, we claim that we can suppose
	\begin{equation}
		\label{Eq:MainpLeq2Hypo}
		\left( \frac{ \abs{ H } }{ \abs{ G' } } \right)^{ 1/3 } \leq \frac{ \abs{ F } }{ \abs{ E } }.
	\end{equation}
	Otherwise, inequality \eqref{Eq:MainReductionWeak<2Subset2} is a consequence of Theorem~\ref{Thm:MainSingleBand} since for $ 1 < q < 2 $, we get
	\begin{equation}
		\begin{split}
			\Abs{ \langle T_{\ann,G',H} f \indf{F} , \indf{E} \rangle } 
			& \lesssim \left( \frac{ \abs{ E } }{ \abs{ F } } \right)^{ \frac{1}{2} - \frac{1}{q} } \abs{ F }^{ 1/2 } \abs{ E }^{ 1/2 } 
			\leq \left( \frac{ \abs{ G } }{ \abs{ H } } \right)^{ \frac{1}{2} - \left( \frac{1}{3} + \frac{1}{3 q} \right) } \abs{ F }^{ 1/2 } \abs{ E }^{ 1/2 },
		\end{split}
	\end{equation}
	and \eqref{Eq:MainReductionWeak<2Subset2} follows for $ 3/2 < p < 2 $. Then, suppose \eqref{Eq:MainpLeq2Hypo} and let
	\begin{equation}\label{Eq:MainpLeq2J0}
		j_{0} \coloneqq \log \left[ \frac{ \abs{ G } }{ \abs{ H } } \frac{ \abs{ E } }{ \abs{ F } } \right]^{ 1/4 }, 
		\qquad 
		k_{0,1}(j) \coloneqq \frac{2}{3} j + \log \left[ \left( \frac{ \abs{ G } }{ \abs{ H } } \right) \left( \frac{ \abs{ F } }{ \abs{ E } } \right)^{2} \right]^{1/3}, 
		\qquad 
		k_{0,2}(j) \coloneqq 2 j + \log \frac{ \abs{ F } }{ \abs{ E } }.
	\end{equation}
	Then,
	\begin{equation}
		\begin{split}
			\Lambda_{\P;\sigma,\tau_0,M} (f\indf{F}, g\indf{E})
			& \lesssim \sum_{j=J}^{j_0} \sum_{k=K}^{k_{0,1}(j)} 2^{-k} 2^{-j} \abs{ E } \left( \frac{ \abs{H} }{ \abs{G'} } \right)^{ 1/2 } 2^{(\frac{3}{2}+\varepsilon) k} 2^{ ( 1+ \varepsilon) j } 
			+ \sum_{j=J}^{j_0} \sum_{k=k_{0,1}(j)}^{\infty} 2^{-k} 2^{-j} 2^{2j} \abs{F} \\
			& \phantom{==} + \sum_{j=j_0}^{\infty} \sum_{k=K}^{k_{0,2}(j)} 2^{-k} 2^{-j} 2^{k} \abs{ E }
			+ \sum_{j=j_0}^{\infty} \sum_{k=k_{0,2}(j)}^{\infty} 2^{-k} 2^{-j} 2^{2j} \abs{F} \\
			& \lesssim \abs{F}^{ \frac{1}{4}+\varepsilon } \abs{ E }^{ \frac{3}{4}-\varepsilon } \left( \frac{ \abs{H} }{ \abs{G} } \right)^{ \frac{1}{4} - \varepsilon }
			+ \abs{F}^{1/4} \abs{ E }^{3/4} \left( \frac{ \abs{H} }{ \abs{G} } \right)^{ 1/4 } .
		\end{split}
	\end{equation}
	Finally, using \eqref{Eq:MainpLeq2Hypo},
	\begin{equation}
		\begin{split}
			\Lambda_{\P;\sigma,\tau_0,M} (f\indf{F}, g\indf{E}) 
			& \lesssim \abs{F}^{ 1/2 } \abs{ E }^{ 1/2 } \left( \frac{ \abs{G} }{ \abs{H} } \right)^{ - \frac{1}{6} + \frac{2}{3}\varepsilon }
			+ \abs{F}^{1/2} \abs{ E }^{1/2} \left( \frac{ \abs{G} }{ \abs{H} } \right)^{ -1/6 }.
		\end{split}
	\end{equation}
	Since this holds for every $ \varepsilon > 0 $, and we are assuming that $ \abs{ H } \leq \abs{ G } $, this shows \eqref{Eq:MainReductionWeak<2Subset2} as long as we take $ - \frac{1}{6} < \frac{1}{2} - \frac{1}{p} $, which is equivalent to $ p > 3/2 $.

	\subsection{The case $ p > 2 $}
	\label{Ss:MainPG2}
	
	In this subsection, we suppose that $ p > 2 $ and $ \abs{ G } \leq \abs{ H } $ and the goal is to show \eqref{Eq:MainReductionWeak>2}. As we have seen in Subsection \ref{Ss:MainPL2}, it will be enough to apply induction over the following lemma.
	
	\begin{lemma}
		\label{Lem:MainInductionG2}
		Let $ G , H \subseteq \R^n $ be two measurable subsets of finite measure with $ \abs{ G } \leq \abs{ H } $ and $ p > 2 $. Then, there exists a subset $ H' \subseteq H $ depending on $ p, G $ and $ H $ with $ \abs{ H' } \geq \abs{ H }/2 $ such that \eqref{Eq:MainReductionWeak>2} holds for $ H' $ instead of $ H $.
	\end{lemma}
	
	One readily checks that Lemma \ref{Lem:MainInductionG2} implies \eqref{Eq:MainReductionWeak>2} by using exactly the same argument that showed that Lemma \ref{Lem:MainInductionL2} implies \eqref{Eq:MainReductionWeak<2}. Furthermore, using exactly the same arguments as in the case $p<2$ which  showed that \eqref{Eq:MainReductionWeak<2Subset2} was sufficient for the proof of \eqref{Eq:MainReductionWeak<2Subset}, the conclusion of Lemma \ref{Lem:MainInductionG2} will follow once we prove
	\begin{equation}
		\label{Eq:MainReductionWeak>2Subset2}
		\sup_{ \ann \in 3^{\Z} } \Abs{ \langle T_{\ann,G,H'} f , g \rangle } 
		\lesssim \left( \frac{ \abs{ G } }{ \abs{ H' } } \right)^{ \frac{1}{2} - \frac{1}{p} } \abs{ F }^{ 1/2 } \abs{ E }^{ 1/2 },
	\end{equation}
	for every $ \abs{f} \leq \indf{F} $, $ \abs{g} \leq \indf{E} $ and $ 2<p<\infty $ with $ F \subseteq H' $ and $ E \subseteq G $.
	
	The construction of $ H' \subseteq H $ will guarantee improved size bounds that will allow us to conclude estimates for the sums of the tops of the trees constructed in Proposition \ref{Prop:CombinedDensitySize}. We turn to that task in the next part of this subsection.
	
	\subsubsection{Improved size estimates}
	
	\begin{lemma} \label{Lem:SizeImproved} Let $m>1$ be a positive integer. Then there exist a constant $C_m>1$ and $ H' \subseteq H $ with $ \abs{H'} \geq \frac{1}{2} \abs{H} $ such that for every tree $ \mathbf{T} \in \{ \mathbf{T}_{j,\nu}' \}_{j,\nu}$, where $ \{ \mathbf{T}_{j,\nu}' \}_{j,\nu}$ is the collection of the trees constructed in Proposition \ref{Prop:CombinedDensitySize}, there holds
		\begin{equation}
			\size( \mathbf{T} ) \lesssim_m \left( \frac{ \mu_{G,H} }{ \densesup( \mathbf{T} ) } \right)^{m}, 
			\qquad 
			\mu_{G,H} \coloneqq C_{m} \left( \frac{ \abs{ G } }{ \abs{ H } } \right)^{ \frac{m-1}{m} },
		\end{equation} 
		for every $ m \in \N $, $ m > 1 $, where $ \size $ is defined with respect to $ F \subseteq H' $ and $ \densesup $ with respect to $ E \subseteq G $.
	\end{lemma}
	
	The constant $ C_m $ will be chosen below.
	
	\begin{proof}
		We abbreviate $ \sigma \coloneqq \size(\mathbf{T}) $ and $ \delta \coloneqq \densesup(\mathbf{T}) $. Remember that for the tree $ \mathbf{T} $ there exists a tile $ t \in \T_{\dens}^{\tops} $ with $ \topt(\mathbf{T}) \leq_2 t $ and $ \dense(t) \simeq \delta $. It follows that
		\begin{equation}
			\label{Eq:SizeImprovedDensityDecomp}
			\Abs{ G \cap v_{ \sigma }^{-1}(\alpha_{t,\tau_0}) \cap 3^{k} K_n^3 R_t } 
			\geq \Abs{ E \cap v_{ \sigma }^{-1}(\alpha_{t,\tau_0}) \cap 3^{k} K_n^3 R_t }
			> \gamma_n \delta 3^{5 n k} \Abs{ 3^{k} K_n^3 R_t },
		\end{equation}
		for some nonnegative integer $ k $ and some dimensional constant $ 0 < \gamma_n \leq 1 $. Define $\mu_{G,H}$ for some constant $C_m$ to be chosen momentarily. By the trivial estimate $ \delta \lesssim 1 $ and the orthogonality estimate $ \sigma \lesssim 1 $ of Lemma \ref{Lem:OrthoLacTree}, we can assume that
		\begin{equation}
			\label{Eq:SizeImprovedMinDelta}
			\delta 
			\geq \gamma_n^{-1} \mu_{G,H},
		\end{equation}
		otherwise the desired estimate follows from 
		\begin{equation}
			\sigma \lesssim 1 \leq \left( \gamma_n^{-1} \frac{ \mu_{G,H} }{ \delta } \right)^{m}.
		\end{equation}
		We thus proceed assuming \eqref{Eq:SizeImprovedMinDelta} and then there exists some nonnegative integer $ \ell $ such that
		\begin{equation}
			\label{Eq:SizeImprovedEllDef}
			\gamma_n^{-1} 3^{\ell n} \leq \frac{\delta}{\mu_{G,H}} < \gamma_n^{-1} 3^{(\ell+1) n}.
		\end{equation}
		If $ \ell \leq k $ then \eqref{Eq:SizeImprovedDensityDecomp} implies
		\begin{equation}
			\Abs{ G \cap v_{ \sigma }^{-1}(\alpha_{t,\tau_0}) \cap 3^{k} K_n^3 R_t } 
			> \gamma_n \delta \Abs{ 3^{k} K_n^3 R_t }
			\geq 3^{\ell n} \mu_{G,H} \Abs{ 3^{k} K_n^3 R_t }
			\geq \mu_{G,H} \Abs{ 3^{k} K_n^3 R_t }.
		\end{equation}
		If $ \ell > k $ then \eqref{Eq:SizeImprovedDensityDecomp} implies
		\begin{equation}
			\Abs{ G \cap v_{ \sigma }^{-1}(\alpha_{t,\tau_0}) \cap 3^{k} K_n^3 R_t } 
			> 3^{\ell n} \mu_{G,H} \Abs{ K_n^3 R_t }
			= \mu_{G,H} \Abs{ 3^{\ell} K_n^3 R_t }.
		\end{equation}
		Defining $ \rho_0 \coloneqq \max\{\ell,k\} $ we therefore have
		\begin{equation}
			\label{Eq:SizeImprovedRNotInH}
			\Abs{ G \cap v_{ \sigma }^{-1}(\alpha_{t,\tau_0}) \cap 3^{\rho_0} K_n^3 R_t } 
			> \mu_{G,H} \Abs{ 3^{\rho_0} K_n^3 R_t }.
		\end{equation}
		Define
		\begin{equation}
			\widetilde{H} \coloneqq \bigcup_{\mu=0}^{\infty} \left\{ 3^{\mu} K_n^3 R_{t} : \, t \in \mathcal{T}_{\ann} , \, \Abs{ G \cap v_{ \sigma }^{-1}(\alpha_{t,\tau_0} ) \cap 3^{\mu} K_n^3 R_{t} } \geq \mu_{G,H} \Abs{ 3^{\mu} K_n^3 R_{t} } \right\}
		\end{equation}
		and set $ H' \coloneqq H \setminus \widetilde{H} $. Using \eqref{Eq:SizeImprovedRNotInH} we have
		\begin{equation}
			3^{\rho_0} K_n^3 R_{t} \subseteq \widetilde{H} 
			\qquad \Longrightarrow \qquad
			3^{\rho_0} K_n^3 R_{t} \cap H' = \emptyset,
		\end{equation}
		for $ \rho_0 $ as above.
		
		Now, for any lacunary tree $ \mathbf{T}^{\lac} \subseteq \mathbf{T} $ with top in $ \mathbf{T} $ we have
		\begin{equation}
			3^{\rho_0} R_{\mathbf{T}^{\lac}} \subseteq 3^{\rho_0} K_n R_{\mathbf{T}} \subseteq 3^{\rho_0} K_n^3 R_{t} 
		\end{equation}
		and so $ 3^{\rho_0} R_{\mathbf{T}^{\lac}} \subseteq (H')^{c} $. Since $ \abs{f} \leq \indf{F} \leq \indf{H'} $ we can estimate for any collection of adapted wave packets $ \{ \varphi_t \in \mathcal{F}_t^{10n} :\, t \in \mathbf{T} \}$
		\begin{equation}
			\begin{split}
				\left( \sum_{ p \in \mathbf{T}^{\lac} } \abs{ \langle f ,\varphi_p \rangle }^2 \right)^{1/2}
				& \leq \sum_{ \rho \geq \rho_0 } \left( \sum_{ p \in \mathbf{T}^{\lac} } \Abs{ \big\langle f \indf{ 3^{\rho+1} R_{\mathbf{T}^{\lac}} \setminus 3^{\rho} R_{\mathbf{T}^{\lac}} },\varphi_p \big\rangle }^2 \right)^{1/2}
				\lesssim_m \sum_{ \rho \geq \rho_0 } 3^{ - 10n \rho m } \Abs{ R_{\mathbf{T}^{\lac}} }^{1/2} \\
				& \lesssim 3^{ - 10n \rho_0 m } \Abs{ R_{\mathbf{T}^{\lac}} }^{1/2}
				\lesssim 3^{ - n \ell m } \Abs{ R_{\mathbf{T}^{\lac}} }^{1/2}
				\simeq \left( \frac{\mu_{G,H}}{\delta} \right)^m \Abs{ R_{\mathbf{T}^{\lac}} }^{1/2}.
			\end{split}
		\end{equation}
		This is the desired estimate for $ \size' $ and by Lemma \ref{Lem:ModifiedSize} we get the estimate also for $ \size $. It remains to prove that $ \abs{H'} \geq \frac{1}{2} \abs{H} $, which we do in the next subsection.
	\end{proof}

	\subsubsection{Proving that $ H' $ is a major subset} 
	\label{Sss:MainInductionSubsetG2}
	
	The proof hinges on the estimate of the following lemma, which is the higher dimensional version of \cite[Lemma 7]{BatemanThiele}.
	
	\begin{lemma}\label{Lem:MainConstructionH}
		For $ 0 \leq \mu < 1 $ and $ q > 1 $, let $ G $ be a measurable set and $ \P \subseteq \mathcal{T} $ a collection of tiles such that
		\begin{equation}
			\label{Eq:MainConstructionHLemHyp}
			\Abs{ G \cap v_{ \sigma }^{-1}(\alpha_{p,\tau_0}) \cap R_p } \geq \mu \abs{ R_p }  \qquad \text{ for every } \quad p \in \P.
		\end{equation}
		Then,
		\begin{equation}
			\bigg| \bigcup_{ p \in \P } R_p \bigg| \lesssim \mu^{-q} \abs{G}.
		\end{equation}
	\end{lemma}
	
	The proof of Lemma \ref{Lem:MainConstructionH} follows as in the two-dimensional case, using the geometric Lemma \ref{Lem:Geom2}.
	
	\emph{Proof of $ \abs{H'} \geq \frac{1}{2} \abs{H} $:} Using the familiar argument involving shifted triadic grids, as in the proof of Proposition \ref{Prop:MaxKey}, we can assume that for each $ \mu \geq 0 $ and $ t \in \P $ there exists $ \tilde{t} $ such that
	\begin{equation}
		3^{\mu} K_n^3 R_t 
		\subseteq R_{\tilde{t}} 
		\subseteq 3^{\mu+2} K_n^3 R_t, 
		\qquad Q_{t} = Q_{\tilde{t}}.
	\end{equation}
	It follows that
	\begin{equation}
		\abs{ \widetilde{H} } 
		\lesssim_{q} \mu_{G,H}^{-q} \abs{G}
		= C_{m}^{-q} \left( \frac{\abs{H}}{\abs{G}} \right) ^{ \frac{m-1}{m} q } \abs{G}
		\leq C_m^{-\frac{m}{m-1}} \abs{H}
	\end{equation}
	upon choosing $ q = \frac{m}{m-1} $. Now choosing $ C_m $ sufficiently large, depending on $m$, yields $ \abs{ \widetilde{H} } \leq \frac{1}{2} \abs{H} $ so that $ \abs{ H' } \geq \frac{1}{2} \abs{H} $ as desired.

	\subsubsection{Completing the proof of \eqref{Eq:MainReductionWeak>2Subset2}} \label{Sss:MainInductionEstimateG2}
	
	Using the reduction to the model operator as in Section \ref{S:Model} it suffices to prove, uniformly for any $ \P \subseteq \mathcal{T}_{\ann} $
	\begin{equation}
		\label{Eq:MainG2ModelDesired}
		\Lambda_{\P;\sigma,\tau_0,M} (f\indf{F}, g\indf{E})
		\lesssim \left( \frac{ \abs{G} }{ \abs{H} } \right)^{\frac{1}{2}-\varepsilon} \abs{F}^{1/2} \abs{E}^{1/2},
	\end{equation}
	for every $ 0 < \varepsilon < 1/2 $, $ \abs{f} \leq \indf{F} $, $ \abs{g} \leq \indf{E} $, $ F \subseteq H' $ and $ E \subseteq G $ with $ H' \subseteq H $ as in Lemma~\ref{Lem:SizeImproved} satisfying $ \abs{H'} > \frac{1}{2} \abs{H} $, $ \abs{G} < \abs{H} $.
	
	Since $ \abs{G}/\abs{H} < 1 $, it is clearly enough to prove \eqref{Eq:MainG2ModelDesired} for arbitrarily small $ 0 < \varepsilon \ll 1/2 $. Using the decomposition in Lemma \ref{Lem:Decom} together with Lemma \ref{Lem:SizeImproved} we have that for every positive integer $m>1$ there exists $H'\subset H$ and a constant $C_m>1$ intervening in the definition of $\mu_{G,H}$ below, such that for $\size$ defined with respect to $F\subset H'$
	\begin{equation}
		\begin{split}
			\Lambda_{\P;\sigma,\tau_0,M} (f\indf{F}, g\indf{E})
			& \lesssim_{m} \sum_{ (j,k) \in \mathcal{A}_{m} } 2^{-j} 2^{-k} \sum_{\tau=1}^{N_{k,j}} \Abs{ R_{\mathbf{T}_{k,j,\tau}}} \\
			& \lesssim \sum_{ (j,k) \in \mathcal{A}_{m} } 2^{-j} 2^{-k} \min \left\{ 2^k \abs{E}, \, 2^{2j} \abs{F} \right\},
		\end{split}
	\end{equation}
	where 
	\begin{equation}
		\mathcal{A}_{m} 
		\coloneqq \left\{ (j,k) :\, 2^{j} \gtrsim 1, \, 2^{k} \gtrsim 1, \, 2^{-j} \lesssim 2^{km} \mu_{G,H}^{m} \right\}
	\end{equation}
and, as before
	\begin{equation}
		\mu_{G,H} \coloneqq C_{m} \left( \frac{ \abs{ G } }{ \abs{ H} } \right)^{ \frac{m-1}{m} }.
	\end{equation}
	Now fix $ 0 < \varepsilon \ll 1/2 $. In the case of the  argument below, $ m=m(\varepsilon) $ will be chosen sufficiently large. We define the values
	\begin{equation}
		2^{k_0} \coloneqq \mu_{G,H}^{-1}, \qquad 2^{j_0(k)} \coloneqq 2^{-km} \mu_{G,H}^{-m}
	\end{equation}
	and note that since $ (j,k) \in \mathcal{A}_{m} $ we have
	\begin{equation}
		2^{-j} 
		\lesssim \min\{ 2^{-j_0(k)}, \, 1 \}
		= \left\{ 
		\begin{array}{lcc} 
			2^{-j_0(k)}, & \text{if} & k \leq k_0 \\ 
			1, & \text{if} & k > k_0.
		\end{array} 
		\right.
	\end{equation}
	We now estimate for $ \abs{f} \leq \indf{F} $, $ \abs{g} \leq \indf{E} $
	\begin{equation}
		\Lambda_{\P;\sigma,\tau_0,M} (f\indf{F}, g\indf{E})
		\lesssim \left[ \sum_{ 2^k \leq 2^{k_0} } \sum_{ 2^{-j} \leq 2^{-j_0(k)} } + \sum_{ 2^k > 2^{k_0} } \sum_{ 2^{-j} \lesssim 1 } \right] \sum_{\tau=1}^{N_{k,j}} \Abs{ R_{\mathbf{T}_{k,j,\tau}}} 
		\eqqcolon \text{I} + \text{II}.
	\end{equation}
	In order to estimate I and II it will be useful to define
	\begin{equation}
		2^{k_1} \coloneqq \frac{ \abs{E} }{ \abs{F} },
		\qquad 2^{k_{0,1}} \coloneqq \mu_{G,H}^{-1} \left( \frac{ \abs{F} }{ \abs{E} } \right)^{1/m},
		\qquad 2^{k_{0,2}} \coloneqq \mu_{G,H}^{-\frac{2m}{2m+1}},
		\qquad 2^{j_1(k)} \coloneqq 2^{k/2} \left( \frac{ \abs{E} }{ \abs{F} } \right)^{1/2}.
	\end{equation}
	For II we estimate
	\begin{equation}
		\begin{split}
			\text{II}
			& \leq \sum_{ 2^k > 2^{k_0} } \sum_{ 2^{-j_1(k)} \leq 2^{-j} \lesssim 1 } 2^{-k} 2^{-j} 2^{2j} \abs{F}
			+ \sum_{ 2^k > 2^{k_0} } \sum_{ 2^{-j} < 2^{-j_1(k)} } 2^{-k} 2^{-j} 2^{k} \abs{E} \\
			& \lesssim \abs{F}^{1/2} \abs{E}^{1/2} \mu_{G,H}^{1/2}
			\simeq_{m} \abs{F}^{1/2} \abs{E}^{1/2} \left( \frac{ \abs{G} }{ \abs{H} } \right)^{\frac{1}{2} - \frac{1}{2m}},
		\end{split}
	\end{equation}
	which yields the desired estimate upon choosing $ m > 1/(2 \varepsilon) $.
	
	The calculation for I is more involved, requiring a careful case-study in order to determine which estimate should be used in which range of parameters. First, we make some further reductions in order to simplify the calculations for I. We will consider two cases: $ k_{0,1} \leq k_1 $ and $ k_{0,1} \geq k_1 $. Observe that these cases correspond to
	\begin{equation}
		\label{Eq:MainG2Cases}
		\frac{ \abs{F} }{ \abs{E} } \leq C_{m}^{\frac{m}{m+1}} \left( \frac{ \abs{G} }{ \abs{H'} } \right)^{\frac{m-1}{m+1}} 
		\qquad \text{and} \qquad
		\frac{ \abs{F} }{ \abs{E} } \geq C_{m}^{\frac{m}{m+1}} \left( \frac{ \abs{G} }{ \abs{H'} } \right)^{\frac{m-1}{m+1}} 
	\end{equation}
	respectively. If $ k_{0,1} \leq k_1 $ then using Theorem~\ref{Thm:MainSingleBand} we get that for every $ 1<q<2 $ and $ \abs{f} \leq \indf{F} $, $ \abs{g} \leq \indf{E} $, 
	\begin{equation}
		\begin{split}
			\Lambda_{\P;\sigma,\tau_0,M} (f\indf{F}, g\indf{E})
			& \lesssim \abs{F}^{1/q} \abs{E}^{1-1/q}
			= \abs{F}^{1/2} \abs{E}^{1/2} \left( \frac{ \abs{F} }{ \abs{E} } \right)^{ \frac{1}{q} - \frac{1}{2} }
			\lesssim \abs{F}^{1/2} \abs{E}^{1/2} \left( \frac{ \abs{G} }{ \abs{H} } \right)^{ (1 - \frac{2}{m+1}) (\frac{1}{q} - \frac{1}{2}) },
		\end{split}
	\end{equation}
	which yields the desired estimate upon choosing $ 1/q = 1 - \frac{\varepsilon}{2} $ and $ m+1\geq \frac{4}{\varepsilon}$.  
	
	Thus, we can assume the second inequality in \eqref{Eq:MainG2Cases}, in which case one can check by a direct calculation that $ k_1 < k_{0,2} < k_{0,1} $.
	
	\emph{Case $ k_0 < k_{0,2} $:} In this case,
	\begin{equation}
		\label{Eq:MainG2Case1}
		\abs{E}^{1/2} < C_{m}^{1/2} \abs{F}^{1/2} \left( \frac{ \abs{G} }{ \abs{H} } \right)^{ \frac{m-1}{2m} }.
	\end{equation}
	Then,
	\begin{equation}
		\begin{split}
			\text{I}
			& \lesssim \sum_{ 2^k \leq 2^{k_0} } \sum_{ 2^{-j} \leq 2^{-j_0(k)} } 2^{-j} \abs{E}
			\lesssim \sum_{ 2^k \leq 2^{k_0} }  2^{-j_0(k)} \abs{E}
			= \sum_{ 2^k \leq 2^{k_0} } 2^{km} \mu_{G,H}^{m} \abs{E}
			\simeq \abs{E}
			\lesssim \abs{F}^{\frac{1}{2}} \abs{E}^{\frac{1}{2}} \left( \frac{ \abs{G} }{ \abs{H} } \right)^{\frac{1}{2} - \frac{1}{2m}},
		\end{split}
	\end{equation}
	where we have used \eqref{Eq:MainG2Case1} in the last inequality. This shows the desired inequality when $ k_0 < k_{0,2} $.
	
	\emph{Case $ k_0 \geq k_{0,2} $:} In this case,
	\begin{equation}
		\begin{split}
			\text{I}
			& \leq \sum_{ 2^k \leq 2^{k_{0,2}} } \sum_{ 2^{-j} \leq 2^{-j_0(k)} } 2^{-j} \abs{E}
			+ \sum_{ 2^{k_{0,2}} \leq 2^k \leq 2^{k_0} } \sum_{ 2^{-j_1(k)} \leq 2^{-j} \leq 2^{-j_0(k)} } 2^{j} 2^{-k} \abs{F}
			+ \sum_{ 2^{k_{0,2}} \leq 2^k \leq 2^{k_0} } \sum_{ 2^{-j} \leq 2^{-j_1(k)} } 2^{-j} \abs{E} \\
			& \eqqcolon \text{I}_1 + \text{I}_2 + \text{I}_3.
		\end{split}
	\end{equation}
	Doing the calculations for each sum we get, for $ \text{I}_1 $,
	\begin{equation}
		\text{I}_1 
		\lesssim \abs{E} \mu_{G,H}^{m} 2^{k_{0,2} m}
		\lesssim \abs{F}^{1/2} \abs{E}^{1/2} \left( \frac{ \abs{E} }{ \abs{F} } \right)^{\frac{1}{2} - \frac{m}{2m+1}} \left( \frac{ \abs{G} }{ \abs{H} } \right)^{\frac{m-1}{2m+1}}
		\lesssim \abs{F}^{1/2} \abs{E}^{1/2} \left( \frac{ \abs{G} }{ \abs{H} } \right)^{ \frac{m-1}{2m+1} - \frac{1}{2} \frac{1}{2m+1} \frac{m-1}{m+1} },
	\end{equation}
	where the last inequality holds by the second inequality in \eqref{Eq:MainG2Cases}. Note that the last exponent tends to $1/2$ as $ m\to\infty $, so this shows the desired inequality in this case. For $ \text{I}_2 $,
	\begin{equation}
		\begin{split}
			\text{I}_2 
			& \lesssim \sum_{ 2^{k} \geq 2^{k_{0,2}} } \abs{F} 2^{-k/2} \left( \frac{ \abs{E} }{ \abs{F} } \right)^{1/2}
			\lesssim \abs{F}^{1/2} \abs{E}^{1/2} 2^{-k_{0,2}/2} \\
			& \simeq \abs{F}^{1/2} \abs{E}^{1/2} \mu_{G,H}^{\frac{m}{2m+1}} \left( \frac{ \abs{E} }{ \abs{F} } \right)^{\frac{1}{2} \frac{1}{2m+1}}
			\lesssim \abs{F}^{1/2} \abs{E}^{1/2} \left( \frac{ \abs{G} }{ \abs{H} } \right)^{ \frac{m-1}{2m+1} - \frac{1}{2} \frac{1}{2m+1} \frac{m-1}{m+1} },
		\end{split}
	\end{equation}
	where again the last inequality is by the second inequality in \eqref{Eq:MainG2Cases} and the last exponent tends to $1/2$ as $ m\to\infty $. Finally,
	\begin{equation}
		\text{I}_3
		\lesssim \sum_{ 2^{k} \geq 2^{k_{0,2}} } 2^{-j_1(k)} \abs{E}
		\leq \sum_{ 2^{k} \geq 2^{k_{0,2}} } 2^{-k/2} \abs{F}^{1/2} \abs{E}^{1/2}
		\lesssim \abs{F}^{1/2} \abs{E}^{1/2} 2^{-k_{0,2}/2},
	\end{equation}
	which is the same as above and the proof is complete.

	\section{On the necessity of the non-degeneracy assumption}
	\label{S:Counterexamples}
	
	One limitation of Theorem~\ref{Thm:Main}, especially when compared to \cite[Theorem 1]{BatemanThiele}, is the assumption $ \sigma \in \Sigma_{1-\varepsilon} $ for some $ \varepsilon > 0 $. One may wonder if Theorem~\ref{Thm:Main} remains true without the non-degeneracy assumption $ \sigma \in \Sigma_{1-\varepsilon} $. Indeed, this assumption can be removed for $0$-homogeneous multipliers in the two-dimensional setting, but this essentially specifies the directional Hilbert transform; see \cite[Theorem 1]{BatemanThiele}.
	
	Failing that, one may ask if the single band version of Theorem~\ref{Thm:Main} remains true in the non-degenerate case. For the rest of the section we fix $ d=n-1 $.
	
	\begin{question}
		\label{Q:NonDegeneracy}
		Suppose that, in $\R^n $, $ \supp(\widehat{f}) \subseteq \{ A\leq |\xi_n|\leq 3A \} $ for some number $ A > 0 $ and $ \apl{\sigma}{\R^n}{\Grdn} $ depends only on the first $ d $ variables
		\begin{equation}
			\sigma(x) = \sigma(x_1,\ldots ,x_d), \qquad x=(x_1,\ldots,x_d,x_n) \in \R^n,
		\end{equation}
		and $ \sigma $ is measurable but otherwise arbitrary. Let $m\in\mathcal M_A(d)$ for large $A$ and consider the operator
		\[
        T_{m,\sigma(\cdot)}f(x)\coloneqq \int_{\R^n} m(V_{\sigma(x)}\Pi_{\sigma(x)}\xi)\widehat f(\xi) e^{2\pi i \langle x,\xi\rangle} d\xi,\qquad x\in \R^n.	
		\]
		Is it true that
		\begin{equation}
			\Norm{ T_{m,\sigma(\cdot)} f }{ L^p(\R^n) }  
			\lesssim \Norm{ f }{ L^p(\R^n) }
		\end{equation}
		for some $p\in (1,\infty)$?
	\end{question}
	
	For $ n=2 $ and $ T_{m,\sigma(\cdot)} $ being the maximal directional Hilbert transform this of course follows from the stronger estimate of \cite[Theorem 1]{BatemanThiele} for $3/2<p<\infty$; in the weaker single-band formulation of Question \ref{Q:NonDegeneracy}, this is \cite[Theorem 3]{BatemanThiele}, essentially proved by Bateman in \cite{Bateman2013Revista}. It is critical to note here that the multiplier of the Hilbert transform is a $0$-homogeneous function in $\R$ and thus Hilbert transform is dilation invariant.
	
	In the following Subsection \ref{Ss:Counterexamples2}, we show that the answer to Question \ref{Q:NonDegeneracy} is negative, in general, already in dimension $ 2 $; for this we use a counterexample from \cite{CGHS} showing the unboundedness of maximally dilated multipliers in $ \R^n $. This provides a negative answer to Question \ref{Q:NonDegeneracy} in any dimension for non-homogeneous multipliers $m\in\mathcal M_A(d)$; this of course excludes the case of the Hilbert transform. However in Subsection \ref{Ss:CounterexamplesN}, we provide examples of \emph{homogeneous} $ d $-dimensional multipliers yielding a negative answer to Question \ref{Q:NonDegeneracy} in any dimension $ n=d+1\geq3 $.
	
	Thus, in contrast with the results in \cite{BatemanThiele}, one cannot hope for a positive answer to Question \ref{Q:NonDegeneracy}, even for homogeneous multipliers unless $n=2$ whence $T_{m,\sigma(\cdot)}$ is the maximal directional Hilbert transform. Thus the non-degeneracy assumption cannot be removed, in general, from the statement of Theorem~\ref{Thm:Main}.

	\subsection{Counterexamples for non dilation invariant multipliers in $ \R^2 $}
	\label{Ss:Counterexamples2}
	
	The authors in \cite{CGHS} construct a multiplier $ m \in \mathcal{M}_{\infty}(1) $ for which the operator
	\begin{equation}
		\label{Eq:Counterexamples2MaxOp}
		f \quad 
		\mapsto \quad \sup_{ k \in \Z } \Abs{ T_{m(2^{-k} (\cdot))} f } \coloneqq \sup_{ n\in\Z } \Abs{ \int_{\R} m(2^{-k} \eta) \widehat{f}(\eta) e^{2 \pi i \langle (\cdot), \eta \rangle} d\eta }
	\end{equation}
	is unbounded on $ L^p(\R) $, for $ 1 < p < \infty $. We will see that assuming a positive answer to Question \ref{Q:NonDegeneracy} would imply the boundedness of \eqref{Eq:Counterexamples2MaxOp}, which would be a contradiction. In particular, Theorem~\ref{Thm:Main} fails without the non-degeneracy assumption.
	
	We first assume that $ m\in C_c(\R)$; this assumption will be removed at the end of the proof. Fix a Schwartz function $ f \in \mathcal{S}(\R) $. For each $ x = (x_1,x_2) \in \R^2 $, take 
	\[
	 \sigma(x) \coloneqq \spn \{ (2^{-k(x_1)} , \sqrt{1-2^{-2k(x_1)}} ) \},\qquad x=(x_1,x_2)\in\R^2,
	 \]
	 where $ \apl{k}{\R}{\N} $ is measurable. Now we consider a Schwartz function $ \psi \in \mathcal{S}(\R) $ such that $ \supp \widehat{ \psi } \subseteq [1,2] $ and define
	\begin{equation}
		F_{\varepsilon}(x) \coloneqq f(x_1) \varepsilon^{1/p} \psi( \varepsilon x_2 ), \qquad x = (x_1,x_2) \in \R^2.
	\end{equation}
	Then, we have
	\begin{equation}
		T_{m,\sigma(\cdot)} F_{\varepsilon}(x) 
		= \varepsilon^{1/p} \int_{\R^2} m \left(\xi_1 2^{-k(x_1)} + \varepsilon \xi_2 \sqrt{1-2^{-2k(x_1)}} \right) \widehat{f}(\xi_1)  \widehat{ \psi }(\xi_2) e^{2 \pi i (x_1 \xi_1 + \varepsilon x_2 \xi_2 ) }  d\xi, \quad x \in \R^2.
	\end{equation}
	Thus,
	\begin{equation}
		\norm{ T_{m,\sigma(\cdot)} F_{\varepsilon} }{ L^p(\R^2) }^p 
		= \int_{\R^2} \Abs{ \int_{\R^2} m \left(\xi_1 2^{-k(x_1)} + \varepsilon \xi_2 \sqrt{1-2^{-2k(x_1)}} \right) \widehat{f}(\xi_1)  \widehat{ \psi }(\xi_2) e^{2 \pi i (x_1 \xi_1 + x_2 \xi_2 ) } d\xi }^p dx.
	\end{equation}
	Applying the dominated convergence theorem and using that $ m $ is continuous and bounded we get, for $ f,\psi \in \mathcal{S}(\R) $,
	\begin{equation}
		\begin{split}
			& \Abs{ \int_{\R} m(\xi_1 2^{-k(x_1)}) \widehat{f}(\xi_1) e^{2 \pi i x_1 \xi_1 } d\xi_1 } \abs{ \psi(x_2) }
			= \Abs{ \int_{\R^2} m(\xi_1 2^{-k(x_1)}) \widehat{f}(\xi_1) \widehat{\psi}(\xi_2) e^{2 \pi i \langle x , \xi \rangle } d\xi } \\
			& \hspace*{40mm} = \lim_{\varepsilon \to 0}  \Abs{ \int_{\R^2} m \left(\xi_1 2^{-k(x_1)} + \varepsilon \xi_2 \sqrt{1-2^{-2k(x_1)}} \right) \widehat{f}(\xi_1)  \widehat{ \psi }(\xi_2) e^{2 \pi i \langle x , \xi \rangle } d\xi },
		\end{split}
	\end{equation}
	so taking the $ L^p(\R^2) $ norm yields
	\begin{equation}
		\Norm{ \int_{\R^2} m(2^{-k(x_1)} \xi_1) \widehat{f}(\xi_1) e^{2 \pi i (\cdot) \xi_1 } d\xi_1 }{ L^p(\R) } \norm{ \psi }{ L^p(\R) } \lesssim \liminf_{ \varepsilon \to 0 } \Norm{ T_{m} F_{\varepsilon} }{ L^p(\R^2) },
	\end{equation}
	where we have applied Fatou's theorem to take the limit out of the last integral. Finally, if we assume a positive answer to Question \ref{Q:NonDegeneracy} in the degenerate case $ \apl{\sigma}{\R^2}{\mathbb{S}^1} $ with $ \sigma(x) = \sigma(x_1) $ and any $ f $ frequency supported in the band $ \abs{ \xi_2 } \simeq 1 $, rescaling in $ x_2 $ implies the same result for function $ F \in \mathcal{S}(\R^2) $ with $ \supp(\widehat{F}) \subseteq \{ \varepsilon^{-1} \leq \abs{ \xi_2 } \leq 2 \varepsilon^{-1} \} $, uniformly in $ \varepsilon > 0 $. In particular, we would have
	\begin{equation}
		\label{Eq:Counterexamples2Assumption}
		\norm{ T_{m,\sigma(\cdot)} F_{\varepsilon} }{L^p(\R^2)}
		\lesssim \norm{ F_{\varepsilon} }{L^p(\R^2)} 
		= \norm{ f }{L^p(\R)} \norm{ \psi }{L^p(\R)}
	\end{equation}
	uniformly in $ \varepsilon $. Choosing $ \apl{ k_1 }{ \R }{ \N } $ such that $ \sup_{ k \in \N } T_{m(2^{-k} (\cdot))} f(x_1) = T_{ m( 2^{-k_1(x_1)} (\cdot) ) } f(x_1) $ yields
	\begin{equation}
		\label{Eq:Counterexample2Assumption}
		\left\| \sup_{ k \in \N } T_{m(2^{-k} (\cdot))} f \right\|_{L^p(\R)}
		\lesssim \norm{ f }{L^p(\R)} 
	\end{equation}
	uniformly over continuous multipliers $ \apl{ m }{ \R }{ \C } $ with $ m \in \mathcal{M}_A(1) $ for $ A $ sufficiently large.
	
	Now the same result holds if we drop the continuity assumption, using an approximation of $ m $ by continuous multipliers $ m_{\varepsilon} $ such that $ m_{\varepsilon} \mapsto m $ in $ \R \setminus \{0\} $ and $ \norm{ m_{\varepsilon} }{\mathcal{M}_A(1)} \lesssim \norm{ m }{\mathcal{M}_A(1)} $ uniformly in $ \varepsilon $.
	
	Also it is not hard to see that we can replace the supremum over $ k \in \N $ in \eqref{Eq:Counterexamples2Assumption} with supremum over $ k \in \Z $. Indeed,
	\begin{equation}
		\sup_{ k \in \Z } \Abs{ T_{ m(2^{-k} (\cdot)) } f }
		= \lim_{ N \to \infty } \sup_{ k \geq -N } \Abs{ T_{ m(2^{-k} (\cdot)) } f }
		= \lim_{ N \to \infty } \sup_{ \ell \in \N } \Abs{ T_{ m(2^{N} 2^{-\ell} (\cdot)) } f }
	\end{equation}
	and the multiplier $ m(2^{N} (\cdot)) $ satisfies the same H\"ormander--Mihlin bounds as $ m $, uniformly in $ N $. This contradicts the unboundedness of \eqref{Eq:Counterexamples2MaxOp}.

	\subsection{$0$-homogeneous counterexamples for $ n \geq 3 $}
	\label{Ss:CounterexamplesN}
	
	In this subsection we construct examples of $0$-homogeneous multipliers in $\R^d$ that provide a negative answer to Question \ref{Q:NonDegeneracy} in any dimension $ n = d+1 \geq 3 $. To that end let $ \varphi \in \mathcal{S}(R) $ be a Schwartz function such that $ \indf{[-1/2,1/2]} \leq \varphi \leq \indf{[-1,1]} $. Take any H\"ormander--Mihlin multiplier $ m \in \mathcal{M}_{\infty}(d-1) $ and define, for $ \eta \in \R^d $,
	\begin{equation}
		M(\eta) 
		\coloneqq m(\eta_1,\ldots ,\eta_{d-1}) \varphi \left( \frac{\eta_{d}}{\abs{ (\eta_1,\ldots ,\eta_{d-1})}} \right).
	\end{equation}
	It is not hard to check that $ M \in \mathcal{M}_{\infty}(d) $. Note that $ M $ is $0$-homogeneous whenever $ m $ is chosen to be $0$-homogeneous on $ \R^{d-1} $.
	
	Next, we construct the vector field. Consider for every $ j \in \{ 1,\ldots,d-1 \} $ a vector field $ \apl{ u_j }{\R^d}{ \mathbb{S}^{d-1} \subseteq \R^d } $ such that $ \{ u_1(y),\ldots,u_{d-1}(y) \} $ is an orthonormal set in $ \R^{d} $ for every $ y \in \R^d $. For $ x \in \R^n $ we define the map
	\begin{equation}
		\R^n = \R^d \times \R \ni x = (\underline{x},x_n) \mapsto \sigma(x)=\sigma(\underline{x}) \coloneqq \spn \left\{ u_1(\underline{x}) ,\ldots , u_{d-1}(\underline{x}) , e_n \right\} \in \Grdn;
	\end{equation}
	see Figure \ref{Fig:CounterexampleVectorField}.	This map depends only on the first $ d $ variables of $ \R^n $ but it does not satisfy the non-degeneracy condition since $ v_{\sigma(x)} \in \R^d $ for every $ x \in \R^n $.
	
	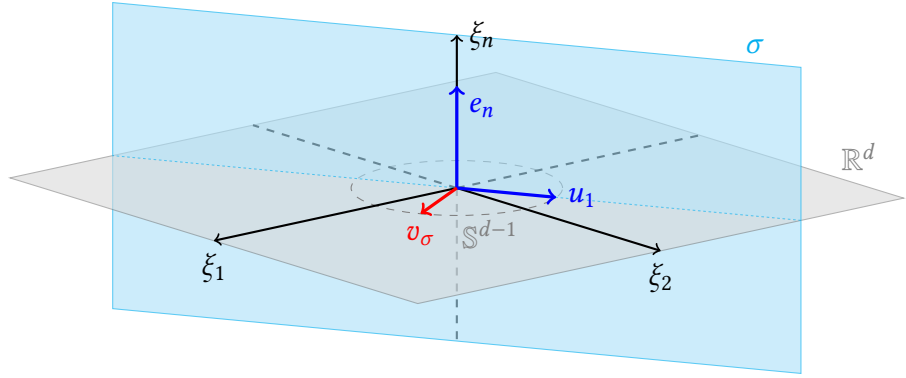
\begin{figure}
		\tdplotsetmaincoords{75}{140}
		\resizebox{0.75\textwidth}{!}{%
		\begin{tikzpicture}[scale=1.4,tdplot_main_coords] 
				
			\draw[dashed,black,thick] (0,0,0) -- (0,0,-1.5);
			
			\filldraw[draw=cyan,fill=cyan!30,fill opacity=0.6,draw opacity=0] 
			({-0.5*3/sqrt(0.75)},3,0) -- ({0.5*3/sqrt(0.75)},-3,0) --
			({0.5*3/sqrt(0.75)},-3,-1.5) -- ({-0.5*3/sqrt(0.75)},3,-1.5) -- cycle;
			
			\draw[draw=cyan,opacity=0.6] 
			({0.5*3/sqrt(0.75)},-3,0) --
			({0.5*3/sqrt(0.75)},-3,-1.5) -- ({-0.5*3/sqrt(0.75)},3,-1.5) -- ({-0.5*3/sqrt(0.75)},3,0) ;
			
			\filldraw[draw=gray, fill=gray!30, opacity=0.6] 
			(-3,-3,0) -- (-3,3,0) -- (3,3,0) -- (3,-3,0) -- cycle;
			
			\draw[gray,dashed] (0,0,0) ellipse (1cm and 0.26cm); 
			
			\draw[thick,->,black] (0,0,0) -- (3,0,0) node[anchor=north]{$\xi_1$};
			\draw[dashed,black,thick] (0,0,0) -- (-3,0,0);
			
			\draw[dashed,black,thick] (0,0,0) -- (0,-3,0);

			\filldraw[draw=cyan,fill=cyan!30,fill opacity=0.6,draw opacity=0] 
			({-0.5*3/sqrt(0.75)},3,0) -- ({0.5*3/sqrt(0.75)},-3,0) --
			({0.5*3/sqrt(0.75)},-3,1.5) -- ({-0.5*3/sqrt(0.75)},3,1.5) -- cycle;
			
			\draw[draw=cyan,draw opacity=0.6] 
			({0.5*3/sqrt(0.75)},-3,0) --
			({0.5*3/sqrt(0.75)},-3,1.5) -- ({-0.5*3/sqrt(0.75)},3,1.5) -- ({-0.5*3/sqrt(0.75)},3,0) ;
			\draw[draw=cyan,dashed,dash pattern=on 1pt off 1pt,draw opacity=0.6] 
			({-0.5*3/sqrt(0.75)},3,0) -- ({0.5*3/sqrt(0.75)},-3,0) ;
			
			\draw[thick,->,black] (0,0,0) -- (0,3,0) node[anchor=north]{$\xi_2$};
			
			\draw[thick,->,black] (0,0,0) -- (0,0,1.5) node[anchor=west]{$\xi_n$};
			
			\draw[very thick,->,blue] (0,0,0) -- (-0.5,{sqrt(0.75)},0) node[anchor=west]{$u_1$};
			
			\draw[very thick,->,blue] (0,0,0) -- (0,0,1) node[anchor=north west]{$e_n$};
			
			\draw[very thick,->,red] (0,0,0) -- ({sqrt(0.75)},0.5,0) node[anchor=north]{$v_{\sigma}$};

			\node[cyan] at (-2,2,1.45) {$\sigma$}; 
			\node[gray] at (0.6,1.2,-0.1) {$\mathbb{S}^{d-1}$};
			\node[gray] at (-2.7,2.7,0.4) {$\mathbb{R}^d$};		
				
		\end{tikzpicture}
		}
		\caption{The vector field for $ \sigma $.}
		\label{Fig:CounterexampleVectorField}
	\end{figure}
	
	Furthermore, for every $ x = (\underline{x},x_n) \in \R^n $, choose $ V_{\sigma(x)} \in \SO(n) $ such that $ V_{\sigma(x)} u_j(\underline{x}) = e_j $, $ V_{\sigma(x)} e_n = e_n $. Then,
	\begin{equation}
		V_{\sigma(x)} \Pi_{\sigma(x)} \xi
		= \left( u_1(\underline{x}) \cdot \underline{\xi} ,\ldots, u_{d-1}(\underline{x}) \cdot \underline{\xi} , \xi_n \right), \qquad \xi = (\underline{\xi},\xi_n) \in \R^d \times \R.
	\end{equation}
	
	Finally, let $ f \in \mathcal{S}(\R^d) $ and $ \psi \in \mathcal{S}(\R) $ be Schwartz functions and $ \psi $ also satisfying $ \supp (\widehat{\psi}) \subseteq [1/2,2] $ and $ 0 \leq \widehat{\psi} \leq 1 $. For $ \varepsilon>0 $, we consider the function
	\begin{equation}
		F_{\varepsilon} (x) \coloneqq f(x_1,\ldots ,x_d) \varepsilon^{1/p} \psi(\varepsilon x_n), \qquad x \in \R^n,
	\end{equation}
	which is frequency supported in the single band $ \{ \xi \in \R^n :\, \varepsilon / 2 \leq \xi_n \leq 2 \varepsilon \} $. Let
	\begin{equation}
		\begin{split}
			T_{M,\sigma(\cdot)} & F_{\varepsilon}(x) 
			\coloneqq \int_{\R^n} M\left( u_1 \cdot \underline{\xi} , \ldots , u_{d-1} \cdot \underline{\xi}, \xi_n \right) \widehat{ \psi } \left( \frac{ \xi_n }{ \varepsilon } \right) \widehat{ f }(\underline{\xi}) e^{ 2 \pi i \langle x , \xi \rangle } \varepsilon^{-1/p'} d\xi \\
			& = \varepsilon^{1/p} \int_{\R^n} m\left( u_1 \cdot \underline{\xi} , \ldots , u_{d-1} \cdot \underline{\xi} \right) \varphi \left( \frac{ \varepsilon \xi_n }{ \Abs{ \left( u_1 \cdot \underline{\xi} , \ldots , u_{d-1} \cdot \underline{\xi} \right) } } \right) \widehat{ \psi } \left( \xi_n \right) \widehat{ f }(\underline{\xi}) e^{ 2 \pi i ( \underline{x} \cdot \underline{\xi} + \varepsilon x_n \xi_n ) } d\xi.
		\end{split}
	\end{equation}
	Assuming a positive answer to Question \ref{Q:NonDegeneracy} yields
	\begin{equation}
		\begin{split}
			& \int_{\R^n} \Abs{ \int_{\R^n} m\left( u_1 \cdot \underline{\xi} , \ldots , u_{d-1} \cdot \underline{\xi} \right) \varphi \left( \frac{ \varepsilon \xi_n }{ \Abs{ \left( u_1 \cdot \underline{\xi} , \ldots , u_{d-1} \cdot \underline{\xi} \right) } } \right) \widehat{ \psi } \left( \xi_n \right) \widehat{ f }(\underline{\xi}) e^{ 2 \pi i \langle x , \xi \rangle } d\xi }^p dx \\
			& \hspace{120mm} \lesssim \int_{\R^d} \abs{f}^p \int_{\R} \abs{ \psi }^p,
		\end{split}
	\end{equation}
	with implicit constant independent of $ \varepsilon $. Since 
	\begin{equation}
		\lim_{ \varepsilon \to 0 } \varphi \left( \frac{ \varepsilon \xi_n }{ \Abs{ \left( u_1 \cdot \underline{\xi} , \ldots , u_{d-1} \cdot \underline{\xi} \right) } } \right) = \varphi(0) = 1, \qquad \text{for a.e.} \quad \xi \in \R^n,
	\end{equation}
	dominated convergence implies that the operator
	\begin{equation}
		\label{Eq:CounterexampleNMax}
		f \quad \mapsto \quad \sup_{ u_1,\ldots ,u_{d-1} \in \mathbb{S}^{d-1} } \Abs{ \int_{ \R^d } m( u_1 \cdot \eta, \ldots , u_{d-1} \cdot \eta) \widehat{ f } (\eta) e^{2 \pi i \langle (\cdot) , \eta \rangle} d\eta },
	\end{equation}
	where the supremum is taken over orthonormal sets $ u_1,\ldots,u_{d-1} $, is bounded on $ L^p(\R^d) $, upon suitable choosing measurable vector fields $ u_1,\ldots,u_{d-1} $. However these operators are in general unbounded, even if $ m $ is $0$-homogeneous.
	
	For example, if $ n = d+1 = 3 $ let us choose $ m \in \mathcal{M}_{\infty}(1) $ to be the multiplier of the Hilbert transform. Then, the operator in \eqref{Eq:CounterexampleNMax} becomes the maximal directional Hilbert transform in $ \R^2 $. This operator is known to be unbounded on all $ L^p $-spaces whenever the set of directions in the supremum is infinite; see \cites{Karagulyan,LMP}. Alternatively, the Kakeya set provides a counterexample since the supremum in \eqref{Eq:CounterexampleNMax} is over $ u \in \mathbb{S}^1 $.
	
	For $ n = d+1 = 4 $ take $ m \in \mathcal{M}_{\infty}(d-1) $ to be the Riesz transform
	\begin{equation}
		m(\eta) = \frac{ \eta_1 }{ \abs{\eta} }, \qquad \eta \in \R^{d-1}
	\end{equation}
	and writing $ \R^d = \R^2 \times \R^{d-2} $ consider $ u_1 \in \mathbb{S}^1 \subseteq \R^2 $, $ u_2,\ldots ,u_{d-1} \in \R^{d-2} $. The boundedness of the corresponding maximal directional Riesz transform would imply the boundedness of the maximal directional Hilbert transform in $ \R^2 $, so Question \ref{Q:NonDegeneracy} has a negative answer for $0$-homogeneous multipliers in any dimension $ n = d+1 \geq 3 $.

	\bibliography{codim1}{}
	\bibliographystyle{amsplain}

\end{document}